\documentclass[a4paper,11pt]{amsart}
\usepackage[latin1]{inputenc}
\usepackage[english]{babel}
\usepackage{amsmath}
\usepackage{amsthm}
\usepackage{amsfonts}
\usepackage{hyperref}
\usepackage{enumerate}
\usepackage{amssymb}
\usepackage{lmodern}
\usepackage{tikz,tikz-cd}
\usepackage[margin=1.2in]{geometry}
\usepackage{microtype}
\usepackage{thm-restate}
\usepackage{float}
\usepackage{enumitem}
\setlist[enumerate,1]{label={(\arabic*)}}

\newtheorem{thm}{Theorem}[section]
\newtheorem*{thm*}{Theorem}
\newtheorem{cor}[thm]{Corollary}
\newtheorem{lem}[thm]{Lemma}
\newtheorem{prop}[thm]{Proposition}
\newtheorem*{prop*}{Proposition}
\theoremstyle{remark}
\newtheorem{rem}[thm]{Remark}
\newtheorem{ques}[thm]{Question}
\newtheorem{conv}[thm]{Convention}
\theoremstyle{definition}
\newtheorem{defn}[thm]{Definition}
\newtheorem*{defn*}{Definition}
\newtheorem{exmp}[thm]{Example}
 
\allowdisplaybreaks

\title[QI classification of RAAGs that split over cyclic subgroups]{Quasi-isometry classification of RAAGs that split over cyclic subgroups}
\author[A. Margolis]{Alexander  Margolis}
\thanks{The author was supported by EPSRC Grant 1499630.}
\address{Alexander J. Margolis, Mathematics Department, Technion - Israel Institute of Technology, Haifa, 32000, Israel}
\email{amargolis@campus.technion.ac.uk}
\date{23rd August 2019}

\begin{document}

\begin{abstract}
For a one-ended right-angled Artin group, we give an explicit description of its JSJ tree of cylinders over infinite cyclic subgroups in terms of its defining graph. This is then used to classify certain right-angled Artin groups up to quasi-isometry. 
In particular, we show that if two right-angled Artin groups are quasi-isometric, then their JSJ trees of cylinders are weakly equivalent. Although the converse to this is not generally true, we define quasi-isometry invariants known as stretch factors that can distinguish quasi-isometry classes of RAAGs with weakly equivalence JSJ trees of cylinders.  We then show that for many right-angled Artin groups, being weakly equivalent and having matching stretch factors is a complete quasi-isometry invariant.

\end{abstract}
\maketitle
\setcounter{tocdepth}{1}
\section{Introduction}\label{sec:intro}
Given a finite simplicial graph $\Gamma$, the \emph{right-angled Artin group}  (RAAG) $A(\Gamma)$ is a group with  the following  presentation: generators correspond to vertices of $\Gamma$, and  relations are commutators $[v,w]$ for vertices $v$ and $w$ that are joined by an edge. RAAGs have become interesting objects of study in recent years due to their rich subgroup structure, e.g. \cite{bestvina1997morse}, \cite{haglundwise08}. The work of Agol and Wise uses RAAGs to understand  3-manifold groups \cite{wise2009announcement}, \cite{agol2013virthaken}. See the survey article  \cite{charney07intro} for more information about RAAGs.

One of the fundamental questions of geometric group theory, posed by Gromov \cite{gromov1993asymptotic}, is the following: given a class  $\mathcal{G}$ of finitely generated groups, when are two groups in $\mathcal{G}$ quasi-isometric to one another? 
We are interested in answering this question when $\mathcal{G}$ is the class of right-angled Artin groups. Much progress has been made on this problem and generalisations of it in recent years: see \cite{behrstock2008qigraphman}, \cite{bestvina08raags}, \cite{behrstock2010highdimRAAG}, \cite{huang2014quasi}, \cite{huang2015quasi}, \cite{huang2018groups} and  \cite{huang2018commensurability}.

The techniques used in \cite{bestvina08raags}, \cite{huang2014quasi} and \cite{huang2015quasi} completely break down if a RAAG contains a non-adjacent transvection; other approaches are needed to determine when RAAGs containing non-adjacent transvections are quasi-isometric. In this paper, we classify the quasi-isometry types of a large class of RAAGs via their JSJ decompositions over two-ended groups. This determines the quasi-isometry type of many RAAGs that cannot be determined by any other known method. This is a far-reaching generalisation of the classification of tree RAAGs in \cite{behrstock2008qigraphman}.

 The JSJ tree of cylinders of a finitely presented group $G$ is a canonical tree $T$ that encodes  information about all splittings of $G$ over two-ended groups. Each vertex of $T$ is one of three \emph{types}: cylindrical, hanging or rigid. Edge stabilizers of $T$ are not necessarily two-ended. 
A \emph{relative quasi-isometry} between vertex stabilisers of the JSJ tree of cylinders is a quasi-isometry that coarsely preserves the \emph{peripheral structure} --- the set of conjugates of incident edge stabilizers.

\begin{defn*}
Let $T$ and $T'$ be two JSJ trees of cylinders of two finitely presented groups.
We say that $T$ and $T'$ are \emph{weakly equivalent} if there exists a tree isomorphism  $\chi:T\rightarrow T'$ such that the following hold:
\begin{enumerate}
\item $\chi$ preserves vertex type (cylinder, hanging or rigid);
\item for all  $v\in VT$, there is a relative quasi-isometry between vertex stabilizers of $v$ and $\chi(v)$, such that the induced map on the peripheral structure agrees with $\chi|_{\mathrm{lk}(v)}$.
\end{enumerate}
We say such a map $\chi$ is a \emph{weak equivalence}.
\end{defn*}

 Proposition \ref{prop:equivbijec} characterises when two JSJ trees of cylinders are weakly equivalent in terms of their quotient graphs of groups. It is a special case of a more general statement of Cashen-Martin \cite{cashenmartin2017quasi}.    However, it is rather technical to state, so we defer doing so till Section \ref{sec:prelims}.
This characterisation is  significantly easier to use in practice than directly constructing a map $\chi$ as  above. We can thus algorithmically determine whether two RAAGs are weakly equivalent, provided we have a sufficiently detailed understanding of the quasi-isometry type of the rigid vertices in the JSJ decomposition. See Proposition \ref{prop:solvqiprob} for a precise statement of this.  

The following is a reformulation of a result of Papasoglu.
\begin{thm}[\cite{papasoglu2005quasi}]\label{thm:qiimpqwe}
If two one-ended finitely presented groups are quasi-isometric, then their JSJ trees of cylinders  are weakly equivalent.
\end{thm}
We use this to help us decide  whether two RAAGs are quasi-isometric. Firstly, a RAAG is one-ended if and only if its defining graph is connected and not a single vertex. Using the work of Papasoglu and Whyte \cite{papasoglu2002quasi}, we note that if two RAAGs are not one-ended, they are quasi-isometric if and only if they have quasi-isometric one-ended factors in their Gru\v{s}ko decomposition. We thus reduce to RAAGs that are one-ended.

Given a one-ended RAAG, we examine its splittings over infinite cyclic groups. The following proposition uses a result of Clay \cite{clay14raagsplit} that describes the JSJ decomposition of a RAAG over infinite cyclic groups. A more general JSJ decomposition of a RAAG over abelian subgroups was given by Groves and Hull \cite{groves2017abelian}. See also \cite{mihaliktschnatz2009visualcox} for a similar result concerning Coxeter groups.
\begin{prop*}
One can visually determine the JSJ tree of cylinders of a RAAG $A(\Gamma)$, i.e. given the graph $\Gamma$, one can write down a graph of groups $C(\Gamma)$  whose Bass-Serre tree is the JSJ tree of cylinders of $A(\Gamma)$. Every vertex and edge group of $C(\Gamma)$ is itself a RAAG.
\end{prop*}

\begin{figure}
\begin{tikzpicture} 
\draw (0,0) -- (1,0) -- (2,0) -- (3,0) -- (4,0);
\draw (3.5,-0.5) -- (3.5,0.5) -- (4,0) -- (3.5,-0.5) -- (3,0) -- (3.5,0.5);
\filldraw[fill=black] (0,0) circle [radius=0.04];
\filldraw[fill=black] (1,0) circle [radius=0.04];
\filldraw[fill=black] (2,0) circle [radius=0.04];
\filldraw[fill=black] (3,0) circle [radius=0.04];
\filldraw[fill=black] (4,0) circle [radius=0.04];
\filldraw[fill=black] (3.5,0.5) circle [radius=0.04];
\filldraw[fill=black] (3.5,-0.5) circle [radius=0.04];
\draw (4.5,-0.7) -- (4.5,0.7);

\draw (5,0) -- (6,0) -- (7,0) -- (8,0) -- (9,0);
\draw (8.5,-0.5) -- (8.5,0.5) -- (9,0) -- (8.5,-0.5) -- (8,0) -- (8.5,0.5);
\draw (5,0) -- (5.5,0.5) -- (6,0);
\filldraw[fill=black] (5,0) circle [radius=0.04];
\filldraw[fill=black] (6,0) circle [radius=0.04];
\filldraw[fill=black] (7,0) circle [radius=0.04];
\filldraw[fill=black] (8,0) circle [radius=0.04];
\filldraw[fill=black] (9,0) circle [radius=0.04];
\filldraw[fill=black] (8.5,0.5) circle [radius=0.04];
\filldraw[fill=black] (8.5,-0.5) circle [radius=0.04];
\filldraw[fill=black] (5.5,0.5) circle [radius=0.04];
\draw (2,-0.5) node {$\Gamma_1$};
\draw (7,-0.5) node {$\Gamma_2$};
\end{tikzpicture}
\caption{The defining graph of two RAAGs that don't have weakly equivalent JSJ tree of cylinders, hence are not quasi-isometric. This example is also considered in \cite{behrstock2017quasiflats}.}\label{fig:bhsgraphs}
\end{figure}
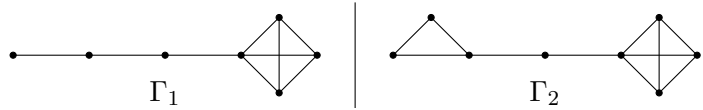

Proposition \ref{prop:relqicyl} gives a criterion to determine whether cylindrical vertex stabilizers in the JSJ tree of cylinders of a RAAG  are relatively quasi-isometric. Along with Proposition \ref{prop:equivbijec}, this gives a powerful method that is frequently able to demonstrate that two RAAGs have JSJ trees of cylinders that are not weakly equivalent, and hence the RAAGs are not quasi-isometric. 

We are not aware of such a statement being used in the literature concerning quasi-isometries of RAAGs. For example, the RAAGs considered in the introduction to \cite{behrstock2017quasiflats}, as well as similar RAAGs containing cut vertices, can be shown  not to be quasi-isometric using this method. 

Indeed, consider the RAAGs with defining graphs shown in Figure \ref{fig:bhsgraphs}, which were considered in a preliminary version of \cite{behrstock2017quasiflats}. The JSJ tree of cylinders decomposition can be described using the graphs of groups $C(\Gamma_1)$ and $C(\Gamma_2)$, as in Definition \ref{def:gogtoc}. We note that $C(\Gamma_2)$ has a cylinder stabilizer isomorphic to $\mathbb{Z}\times (\mathbb{Z}^2*\mathbb{Z})$. This can easily be seen not be quasi-isometric to any cylinder stabilizer of $C(\Gamma_1)$ --- the details of this  follow from Section \ref{sec:cyl}. Thus the RAAGs $A(\Gamma_1)$ and $A(\Gamma_2)$ cannot have weakly equivalent JSJ trees of cylinders, hence are not quasi-isometric.

Given a class $\mathcal{C}$ of RAAGs, we let $\mathcal{J}(\mathcal{C})$ denote the class of one-ended RAAGs whose JSJ trees of cylinders have rigid vertex stabilizers in the class $\mathcal{C}$.
 \begin{figure}
    \begin{tikzpicture}[semithick,scale=1]

      \draw (0, 0) -- (0.866025404, 0.5) -- (0.866025404, -0.5) -- (0,0);
      \draw (0, 0) -- (-0.866025404, 0.5) -- (-0.866025404, -0.5) -- (0,0);
      \filldraw[fill=black] (0,0) circle [radius=0.04];
      \filldraw[fill=black] (0.866025404, 0.5) circle [radius=0.04];
      \filldraw[fill=black] (-0.866025404, 0.5) circle [radius=0.04];
      \filldraw[fill=black] (0.866025404, -0.5) circle [radius=0.04];
      \filldraw[fill=black] (-0.866025404, -0.5) circle [radius=0.04];
      \draw (-0.866025404, 0.5) -- (-0.866025404-0.984807753, 0.5+0.173648178);
      \draw (-0.866025404, 0.5) -- (-0.866025404+0.342020143, 0.5+0.939692621);
      \draw (-0.866025404, -0.5) -- (-0.866025404-0.5, -0.5-0.866025404);
      \filldraw[fill=black] (-0.866025404-0.984807753, 0.5+0.173648178) circle [radius=0.04];
      \filldraw[fill=black] (-0.866025404+0.342020143, 0.5+0.939692621) circle [radius=0.04];
      \filldraw[fill=black] (-0.866025404-0.5, -0.5-0.866025404) circle [radius=0.04];
      \filldraw[fill=black] (0.866025404, 0.5) circle [radius=0.04];
      \draw (0.866025404, 0.5) -- (0.866025404+0.866025404, 0.5+0.5) -- (0.866025404, 0.5+1) -- (0.866025404, 0.5);
      \draw (0.866025404, -0.5) -- (0.866025404+0.866025404, -0.5-0.5) -- (0.866025404, -0.5-1) -- (0.866025404, -0.5);
      \draw (0,0) -- (0,1);
      \filldraw[fill=black] (0,1) circle [radius=0.04];
      \filldraw[fill=black] (0.866025404+0.866025404, 0.5+0.5) circle [radius=0.04];
      \filldraw[fill=black] (0.866025404, 0.5+1) circle [radius=0.04];
      \filldraw[fill=black] (0.866025404+0.866025404, -0.5-0.5) circle [radius=0.04];
      \filldraw[fill=black] (0.866025404, -0.5-1) circle [radius=0.04];
      \draw (0.866025404+0.866025404, -0.5-0.5) -- (0.866025404+0.866025404+1, -0.5-0.5);
      \draw (0.866025404+0.866025404, 0.5+0.5) -- (0.866025404+0.866025404+1, 0.5+0.5);  
      \draw (0.866025404, 0.5+1) -- (0.866025404-0.5, 0.5+1+0.866025404);
      \draw (0.866025404, -0.5-1) -- (0.866025404-0.5, -0.5-1-0.866025404);
      \filldraw[fill=black] (0.866025404+0.866025404+1, -0.5-0.5) circle [radius=0.04];
      \filldraw[fill=black] (0.866025404+0.866025404+1, 0.5+0.5) circle [radius=0.04];
      \filldraw[fill=black] (0.866025404-0.5, 0.5+1+0.866025404) circle [radius=0.04];
      \filldraw[fill=black] (0.866025404-0.5, -0.5-1-0.866025404) circle [radius=0.04];           
   \end{tikzpicture}
    \caption{A 3-clique tree-graded graph}\label{fig:tritree}
    \end{figure}
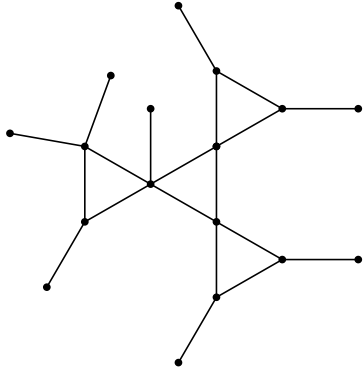

\begin{thm}\label{thm:abelweakequiv}
Let $\mathcal{C}_{\mathrm{ab}}$ denote the class of finitely generated free abelian groups. Then any RAAG quasi-isometric to a RAAG in  $\mathcal{J}(\mathcal{C}_{\mathrm{ab}})$ is also in $\mathcal{J}(\mathcal{C}_{\mathrm{ab}})$. Moreover, two RAAGs in $\mathcal{J}(\mathcal{C}_{\mathrm{ab}})$ are quasi-isometric if and only if they have weakly equivalent JSJ trees of cylinders.
\end{thm}
We note that $\mathcal{J}(\mathcal{C}_{\mathrm{ab}})$ includes the RAAGs whose defining graphs are shown in Figure \ref{fig:bhsgraphs}, as well as tree RAAGs considered in \cite{behrstock2008qigraphman}.
As a sample application of Theorem \ref{thm:abelweakequiv}, we give a complete quasi-isometry classification of the following class of RAAGs.

\begin{defn*} Let  $\Gamma$ be a finite, connected, simplicial graph,  and let $C_0$  denote the set of cut vertices of $\Gamma$. We say a subgraph of $\Gamma$ is \emph{biconnected} if it is connected and has no cut vertex. For $n\geq 2$, we say that $\Gamma$ is  an \emph{$n$-clique tree-graded graph} if:
\begin{itemize}
\item every vertex of $\Gamma$ is either  in $C_0$ or has  valence one;
\item $\Gamma$ has diameter at least three;
\item each  maximal biconnected subgraph is either an $n$-clique, all of whose vertices are in $C_0$, or a 2-clique containing exactly one vertex in $C_0$.
\end{itemize}
\end{defn*}

A simplicial graph is a  2-clique tree-graded graph if and only if it is a tree of diameter at least three.  An example of a 3-clique tree-graded graph is shown in Figure \ref{fig:tritree}. 
We say that $A(\Gamma)$ is an  \emph{$n$-clique tree-graded RAAG} whenever $\Gamma$ is an  $n$-clique tree-graded graph.
The following theorem generalises a result of Behrstock-Neumann \cite{behrstock2008qigraphman}, who prove the $n=2$ case.

\begin{restatable*}{thm}{ncliquetreeraags}\label{thm:ncliquetreeraags}
Fix some  $n\geq 2$. All $n$-clique tree-graded RAAGs are quasi-isometric to one another, and every RAAG quasi-isometric to an $n$-clique tree-graded RAAG is itself an $n$-clique tree-graded RAAG.
\end{restatable*}

Unfortunately, there exist RAAGs that are not quasi-isometric but have JSJ trees of cylinders that are weakly equivalent. The defining graphs of two such RAAGs are shown in Figure \ref{fig:raagsnotqistretchintro}. 
Both maximal biconnected subgraphs of $\Gamma$ are pentagons $P$, whereas one maximal biconnected subgraph of $\Lambda$ is a pentagon and the other is $P^*$, the double of a pentagon along the star of the vertex $v$. 

A \emph{standard geodesic} in a RAAG $A(\Gamma)$ is a geodesic in the Cayley graph of $A(\Gamma)$ (with respect to the standard generating set) whose vertex set is some coset $g\langle v \rangle $ for some $v\in V\Gamma$.
It can be seen that $A(P^*)$ is an index two subgroup of $A(P)$. Moreover, any relative quasi-isometry from $A(P)$ to $A(P^*)$  shrinks distances a factor of two along the standard geodesic corresponding to $\langle v\rangle$. In contrast, no quasi-isometry $A(P)\rightarrow A(P)$ can shrink distances by a factor of two along a standard geodesic.
We thus deduce that $A(\Gamma)$ and $A(\Lambda)$ cannot be quasi-isometric. 

\begin{figure}[ht!]
    \begin{tikzpicture}[semithick,scale=1.5]

\draw (0+0, 0) -- (0+-0.588, 0.809) -- (0+-1.54,  0.500) -- (0+-1.54, -0.500) -- (0+-0.588, -0.809) -- (0+0,0);
      \filldraw[fill=black] (0+0,0) circle [radius=0.04];
      \filldraw[fill=black] (0+-0.588, 0.809) circle [radius=0.04];
      \filldraw[fill=black] (0+-1.54, 0.500) circle [radius=0.04];
      \filldraw[fill=black] (0+-1.54, -0.500) circle [radius=0.04];
      \filldraw[fill=black] (0+-0.588, -0.809) circle [radius=0.04];
      \draw (0+0, 0) -- (0+0.588, 0.809) -- (0+1.54,  0.500) -- (0+1.54, -0.500) -- (0+0.588, -0.809) -- (0+0,0);
      \filldraw[fill=black] (0+0.588, 0.809) circle [radius=0.04];
      \filldraw[fill=black] (0+1.54, 0.500) circle [radius=0.04];
      \filldraw[fill=black] (0+1.54, -0.500) circle [radius=0.04];
      \filldraw[fill=black] (0+0.588, -0.809) circle [radius=0.04];     
	\draw (-0.2,0) node {$v$};
      
     \draw (4+0, 0) -- (4+-0.588, 0.809) -- (4+-1.54,  0.500) -- (4+-1.54, -0.500) -- (4+-0.588, -0.809) -- (4+0,0);
      \filldraw[fill=black] (4+0,0) circle [radius=0.04];
      \filldraw[fill=black] (4+-0.588, 0.809) circle [radius=0.04];
      \filldraw[fill=black] (4+-1.54, 0.500) circle [radius=0.04];
      \filldraw[fill=black] (4+-1.54, -0.500) circle [radius=0.04];
      \filldraw[fill=black] (4+-0.588, -0.809) circle [radius=0.04];
      \draw (4+0, 0) -- (4+0.588, 0.809) -- (4+1.54,  0.500) -- (4+1.54, -0.500) -- (4+0.588, -0.809) -- (4+0,0);
      \filldraw[fill=black] (4+0.588, 0.809) circle [radius=0.04];
      \filldraw[fill=black] (4+1.54, 0.500) circle [radius=0.04];
      \filldraw[fill=black] (4+1.54, -0.500) circle [radius=0.04];
      \filldraw[fill=black] (4+0.588, -0.809) circle [radius=0.04];     
      \draw (4+0.588, 0.809) -- (4+1.54+0.5,  0.600) -- (4+1.54+0.5, -0.600) -- (4+0.588, -0.809);
      \filldraw[fill=black] (4+1.54+0.5,  0.600) circle [radius=0.04];
      \filldraw[fill=black] (4+1.54+0.5, -0.600) circle [radius=0.04];   
\draw (2,-1.5) -- (2,1);      	
	\draw (4-0.2,0) node {$v$};
      \draw (0,-1.25) node {$\Gamma$};
      \draw (4,-1.25) node {$\Lambda$};      
   \end{tikzpicture}
    \caption{The defining graphs of RAAGs that are not quasi-isometric}\label{fig:raagsnotqistretchintro}
    \end{figure}
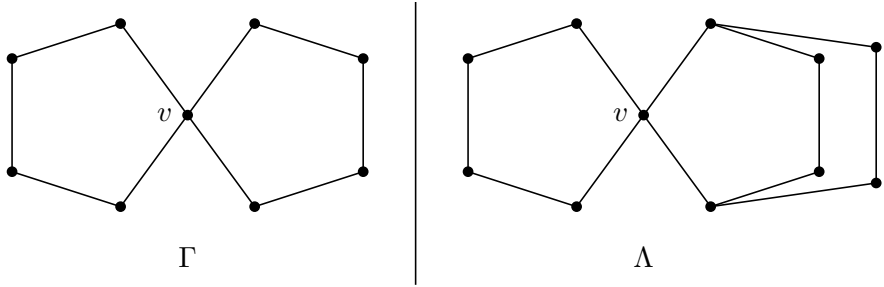

In Section \ref{sec:stretch}, we associate \emph{stretch factors} to certain standard geodesics.   These are similar to stretch factors of \cite{cashenmartin2017quasi}, and are well-defined on the subclass of standard geodesics, called \emph{rigid geodesics}, which exhibit the sort of rigidity phenonmenon considered above. If a RAAG has finite outer automorphism group, then work of Huang \cite{huang2014quasi} demonstrates that every standard geodesic is rigid. This actually holds for a much larger class of RAAGs, namely RAAGs of type II with trivial centre \cite{huang2015quasi}.

We say that two  JSJ trees of cylinders of one-ended RAAGs,  $T$ and $T'$, are \emph{equivalent} if there is a weak equivalence $\chi:T\rightarrow T'$ that preserves stretch factors of rigid geodesics.
\begin{thm}[See Corollary \ref{cor:embeldecgeom}]\label{thm:qiimpequivintro}
If two one-ended RAAGS are quasi-isometric, then their JSJ trees of cylinders are equivalent.
\end{thm}
This can be used to distinguish quasi-isometry classes of RAAGs that have weakly equivalent JSJ trees of cylinders, such  as the RAAGs $A(\Gamma)$ and $A(\Lambda)$ considered above. 
 
Our main result is a converse to Theorem \ref{thm:qiimpequivintro}  for a very large class of RAAGs. To motivate this class, we observe one can define a quasi-isometry $\mathbb{Z}^n\rightarrow \mathbb{Z}^n$ as the product of $n$ arbitrary  homotheties. Thus one can have arbitrary stretching on  standard geodesics in $\mathbb{Z}^n$.  We call such standard geodesics \emph{flexible}.  A \emph{dovetail  RAAG} is one in which every standard geodesic is either very rigid or very flexible in a particularly strong sense. This dichotomy between rigidity and flexibility is fairly typical. 

 We show that  the class  of dovetail RAAGs includes finitely generated free and free abelian groups, tree RAAGs, RAAGs with finite outer automorphism group and more generally, RAAGs of type II. It is closed under taking free products, direct products and amalgamating  along the infinite cyclic subgroup defined by a standard generator. In particular,
we do not have any examples of RAAGs that are not dovetail, which motivates the following question:
\begin{ques}Is every RAAG dovetail?\label{ques:univsysintro}
\end{ques} 

To prove a partial converse to Theorem \ref{thm:qiimpequivintro}, we want to promote an equivalence of JSJ trees of cylinders to a quasi-isometry of RAAGs by gluing together quasi-isometries of vertex spaces that agree on common edge spaces. This can always be done if vertex stabilizers are in $\mathcal{D}$, the class of dovetail RAAGs.  Indeed, this motivates the name \emph{dovetail} since these are RAAGs that can be made to fit together nicely.

\begin{thm}[See Theorem \ref{thm:mainthm}]\label{thm:mainthmintro}
Two RAAGs in $\mathcal{J}(\mathcal{D})$  are quasi-isometric if and only if they have equivalent JSJ trees of cylinders.
\end{thm}

We remark that the notions of weak equivalence and equivalence, as defined in this introduction, are special cases of a more general notion of equivalence; see Definition \ref{defn:equiv}.
We emphasise that even if RAAGs are not known to be contained in $\mathcal{J}(\mathcal{D})$, we can still apply Theorem \ref{thm:qiimpequivintro}. Thus the machinery  of this article can be frequently used to determine that RAAGs are \emph{not} quasi-isometric. The assumption that a RAAG is contained in $\mathcal{J}(\mathcal{D})$ is only necessary to obtain a complete quasi-isometry invariant.

If a RAAG splits over $\mathbb{Z}$ but is not contained in $\mathcal{J}(\mathcal{D})$, we expect our methods  can still be used to obtain a complete quasi-isometry invariant. One needs a complete understanding of just how rigid and flexible standard geodesics are; see Remark \ref{rem:ifunivfrfails}. However, since we have no examples of RAAGs that are not dovetail, we will not pursue this  matter further. 

Using Theorem \ref{thm:mainthmintro}, we are able to algorithmically decide if certain RAAGs are quasi-isometric once we  have a sufficiently detailed understanding of the relative quasi-isometries of the vertex spaces in their JSJ decomposition.
A sample theorem of this sort is the following:
\begin{restatable*}{thm}{algor}\label{thm:algor}
Let $\mathcal{C}$ be the class of RAAGs with JSJ decompositions whose rigid vertex groups are either free abelian or have finite outer automorphism group. There is an algorithm that, when given as input finite simplicial graphs $\Gamma$ and $\Lambda$ such that 
$A(\Gamma)\in \mathcal{C}$, determines whether $A(\Gamma)$ is quasi-isometric to $A(\Lambda)$.
\end{restatable*}

In contrast to the quasi-isometrically rigid classes of RAAGs considered in \cite{bestvina08raags}, \cite{huang2014quasi} and \cite{huang2015quasi}, there are several RAAGs to which Theorem \ref{thm:mainthmintro} can be applied that are quasi-isometric but not commensurable. In \cite{casals2019commensurability}, it is shown that there are infinitely many commensurability classes within the class of  tree RAAGs, all of which lie in a single quasi-isometry class.

The quasi-isometry classification of certain right-angled Coxeter groups by Dani--Thomas \cite{danithomas2017jsj} contains analogous results to this article, classifying certain right-anged Coxeter groups using Bowdtich's JSJ tree, which is just the JSJ tree of cylinders of a hyperbolic group.  We remark that although every right-angled Artin group is a finite index subgroup of a right-angled Coxeter group \cite{davis2000raagvsracg}, the groups considered in \cite{danithomas2017jsj} are one-ended and hyperbolic, so are not commensurable to any RAAG. Thus the main results of \cite{danithomas2017jsj} have no overlap with the results of this paper.

The methods developed in this article are of interest in their own right, and might be used in settings other than RAAGs. Studying the JSJ decomposition of finitely presented groups is a natural approach to their quasi-isometry classification, and the techniques developed here may be used when the conditions of \cite{cashenmartin2017quasi} are not known to hold. In particular, we do not assume that rigid  vertex groups (in the JSJ sense) are quasi-isometrically rigid, and we do not assume that cylinder stabilizers are two-ended.

We explain briefly why the quasi-isometry classification of the RAAGs that we consider present technical difficulties that do not occur in  \cite{bestvina08raags}, \cite{huang2014quasi} and \cite{huang2015quasi}. A typical strategy for classifying a class $\mathcal{G}$ of finitely generated groups up to quasi-isometry involves showing that they satisfy a form of (relative) quasi-isometic rigidity: two groups in $\mathcal{G}$ are quasi-isometric if and only if they are commensurable. When such a statement holds --- as is the case for the RAAGs considered in \cite{bestvina08raags}, \cite{huang2014quasi} and \cite{huang2015quasi} --- one does not need to explicitly construct quasi-isometries in order to demonstrate that two groups are quasi-isometric, since commensurable groups are necessarily quasi-isometric.
This strong form of relative quasi-isometric rigidity  doesn't hold (in general) for the RAAGs considered in this paper, as shown in \cite{casals2019commensurability}. We thus need to explicitly build quasi-isometries between RAAGs to demonstrate that they are quasi-isometric. This is also done for the RAAGs considered in \cite{behrstock2008qigraphman}  and \cite{behrstock2010highdimRAAG}.

\subsection*{Outline of the proof of Theorem \ref{thm:mainthmintro}} We briefly explain the main ingredients that we use to prove Theorem \ref{thm:mainthmintro}. In Definition \ref{def:gogtoc} we describe a graph of groups decomposition $C(\Gamma)$ of the RAAG $A(\Gamma)$, which can be easily read off from the defining graph $\Gamma$. Proposition \ref{prop:jsjtoc} shows  that $C(\Gamma)$ is the JSJ tree of cylinders decomposition of $A(\Gamma)$. Theorem \ref{thm:jsjtocqi} then states that two RAAGs that are quasi-isometric necessarily have weakly equivalent JSJ trees of cylinders. This motivates the following two questions, which we partially answer in the remainder of this article:
\begin{enumerate}
\item When do two RAAGs have weakly equivalent JSJ trees of cylinders?
\item Can we add extra data to the JSJ tree of cylinders to obtain a complete quasi-isometry invariant?
\end{enumerate}

By work of Cashen--Martin,  the first question is reduced to understanding  the relative quasi-isometry classification of vertex and edge groups in the JSJ tree of cylinders; see Section \ref{subsec:decoratedtrees} and \cite{cashenmartin2017quasi}. In Proposition \ref{prop:relqicyl} we give a necessary and sufficient condition for two cylindrical vertex groups in the JSJ tree of cylinders to be relatively quasi-isometric. The proof of Proposition \ref{prop:relqicyl} makes heavy use of the methods developed by Papasoglu--Whyte \cite{papasoglu2002quasi}, which we explain in Section \ref{sec:freeproduct}. Thus Proposition \ref{prop:relqicyl}  reduces the first question to understanding the relative quasi-isometry classification of  rigid vertex stabilizers in the JSJ tree of cylinders; see Proposition \ref{prop:solvqiprob}.

We now investigate the second question in Sections \ref{sec:stretch} and \ref{sec:constructqi}. As remarked earlier, the RAAGs whose defining graphs are shown in Figure \ref{fig:raagsnotqistretchintro} have weakly equivalent JSJ trees of cylinders but are not quasi-isometric. However, we can label certain edges in their JSJ trees of cylinders with extra data, namely the relative stretch factors, to obtain the \emph{embellished decoration}; see Definition \ref{defn:initialdecoration}. We show in Corollary \ref{cor:embeldecgeom}  that a quasi-isometry between two RAAGs preserves this embellished decoration. In particular, the two RAAGs whose defining graphs are shown in Figure \ref{fig:raagsnotqistretchintro} do \emph{not} have equivalent embellished  JSJ trees of cylinders.

We define the class of \emph{dovetail} RAAGs in Section \ref{sec:dovetail} and give many examples of RAAGs that are known to be dovetail. We use  the definition of  dovetail RAAG to choose quasi-isometries of vertex spaces satisfying the hypotheses of Proposition \ref{prop:qitree}. We then  apply Proposition \ref{prop:qitree} to  glue these quasi-isometries of vertex spaces together and construct a quasi-isometry between the ambient groups. The argument is quite delicate, and we make heavy use of the technology developed in Section \ref{sec:freeproduct}. The most involved part is Lemma \ref{lem:constructcylqi}, which   encodes quasi-isometries between incident edge spaces into a quasi-isometry between cylindrical vertex spaces.

\subsection*{Organization of the paper} Section \ref{sec:prelims} consists of preliminary results  on coarse geometry, JSJ trees of cylinders and RAAGs. We give an account of the parts of \cite{cashenmartin2017quasi} that we use.  
In Section \ref{sec:jsjtoc}  we give an explicit construction of the  JSJ tree of cylinders decomposition of a RAAG.
Section \ref{sec:freeproduct} contains technical results that allow one to construct quasi-isometries between infinite ended spaces, developing the methods of \cite{papasoglu2002quasi}. The technology in Section \ref{sec:freeproduct} is needed to prove Proposition \ref{prop:relqicyl} and Lemma \ref{lem:constructcylqi}.
In Section \ref{sec:cyl} we investigate the coarse geometry of cylindrical vertex stabilizers and give a necessary and sufficient for  cylindrical vertex groups to be relatively quasi-isometric. 

In Section \ref{sec:stretch} we elaborate on the above example of RAAGs that are  not quasi-isometric, yet have JSJ trees of cylinder that are weakly equivalent. We then introduce relative stretch factors, explaining  why they are well-defined and their geometric significance. We explain how to define the embellished decoration on the JSJ tree of cylinders, which  allows one to distinguish between certain RAAGs whose JSJ tree of cylinders are weakly equivalent, yet are not quasi-isometric.
In Section \ref{sec:constructqi}, we construct a quasi-isometry between RAAGs in $\mathcal{J}(\mathcal{D})$ that have equivalent JSJ trees of cylinders. This completes the proof of  Theorem \ref{thm:mainthmintro}.

In Section \ref{sec:algor} we investigate the algorithmic consequences of Theorem \ref{thm:mainthmintro}, and prove Theorem \ref{thm:algor}. In Section \ref{sec:ncliquetreeraags}, we prove Theorem \ref{thm:ncliquetreeraags}.
\subsection*{Acknowledgements}
I would like to thank Jingyin Huang for several helpful discussions, and my supervisor Panos Papasoglu for his advice and comments. I would also like to thank Chris Cashen for reading a preliminary version of this paper and offering suggestions.
\section{Preliminaries}\label{sec:prelims}
\subsection{Coarse Geometry}
Let $(X,d)$ be a metric space. For $x\in X$ and $\emptyset\neq A\subseteq X$, we let $d(x,A):=\inf \{d(a,x)\mid a\in A\}$.  
We define $N_r(A):=\{x\in X\mid d(x,A)\leq r\}$.
If  $A,B\subseteq X$ are non-empty, we define the \emph{Hausdorff distance} to be $$d_\mathrm{Haus}(A,B):=\inf\{r\geq 0\mid A\subseteq N_r(B) \textrm{ and } B\subseteq N_r(A)\}.$$ We are often only interested in whether $d_\mathrm{Haus}(A,B)$ is finite. In such a case, we simply say that $A$ and $B$ are \emph{coarsely equivalent}. Coarse equivalence is an equivalence relation among subsets of a given metric space. We let $[A]$ denote the coarse equivalence class containing $A$.
We say $N\subseteq X$ is an \emph{$(\epsilon,\delta)$-net} if $d(n,n')\geq \epsilon$ for all distinct $n,n'\in N$, and if for every $x\in X$, there exists some $n\in N$ such that $d(n,x)\leq \delta$. Such a net of $X$ is coarsely equivalent to $X$.

Suppose $(X,d_X)$ and $(Y,d_Y)$ are metric spaces and that $K\geq 1$ and $A\geq 0$. We say that $f:X\rightarrow Y$ is  \emph{$(K,A)$-coarse Lipschitz} if  for every $x,x'\in X$, $$d_Y(f(x),f(x'))\leq K d_X(x,x')+A.$$ We say that $f:X\rightarrow Y$ is a \emph{$(K,A)$-quasi-isometric embedding} if  for every $x,x'\in X$, $$\frac{1}{K}d_X(x,x')-A\leq d_Y(f(x),f(x'))\leq K d_X(x,x')+A.$$
If we also know that  for every $y\in Y$, there exists an $x\in X$ such that $d_Y(f(x),y)\leq A$, then we say that $f$ is a \emph{$(K,A)$-quasi-isometry}.  
Two metric spaces $X$ and $Y$  are \emph{quasi-isometric} if and only if there is a $(K,A)$-quasi-isometry $f:X\rightarrow Y$ for some $K\geq 1$ and $A\geq 0$. 
 
For $V\subseteq X$, we say that $f,g:V\rightarrow Y$   are  \emph{$A$-close} if $\sup_{v\in V}\{d_Y(g(v),f(v))\}\leq A$. We say that $f$ and $g$ are \emph{close} if they are $A$-close for some $A\geq 0$.

Every quasi-isometry $f:X\rightarrow Y$ has a \emph{coarse inverse} $g:Y\rightarrow X$ such that $g\circ f$ is close to $\mathrm{id}_X$ and $f\circ g$ is close to $\mathrm{id}_Y$. Such a $g$ is necessarily a quasi-isometry. Among the class of metric spaces, being quasi-isometric is an equivalence relation. We let $[[X]]$ denote the equivalence class consisting of all metric spaces quasi-isometric to $X$.
 
We remark that if $V,V'\subseteq X$ are coarsely equivalent and $f:X\rightarrow Y$ is a quasi-isometry, then $f(V)$ and $f(V')$ are coarsely equivalent.  Thus a quasi-isometry sends coarse equivalence classes to coarse equivalence classes.

We say that $X$ and $Y$ are \emph{bi-Lipschitz equivalent} if there exists a bijective $(K,0)$-quasi-isometric embedding from $X$ to $Y$ for some $K\geq 1$.  Any quasi-isometry $f:X\rightarrow Y$ can be factored through a bi-Lipschitz equivalence by passing to and from suitable nets. More precisely, for any quasi-isometry $f:X\rightarrow Y$, $f|_N$ is close to a bi-Lipschitz equivalence  $g:N\rightarrow M$ for suitably  chosen nets $N\subseteq X$ and $M\subseteq Y$.

We may think of a finitely generated group, equipped with the word metric with respect to some finite generating set, as a metric space in its own right. This is well defined up to quasi-isometry. See \cite{gromov1983infinite} for more details. The following elementary proposition illustrates this philosophy by  formulating an algebraic property of pairs of subgroups in terms of their geometry. We recall that two subgroups $K,H\leq G$ are said to be \emph{commensurable} if $K\cap H$ has finite index in both $K$ and $H$. 
\begin{prop}[{\cite[Corollary 2.14]{mosher2011quasiactions}}]\label{prop:sbgpcommen}
Let $G$ be a finitely generated group and  $H,K\leq G$ be finitely generated subgroups. Then $K$ and $H$ are coarsely equivalent  if and only if  $H$ and $K$  are commensurable.
\end{prop}

Given $A,B\subseteq X$, we say that $A$ is \emph{coarsely contained} in $B$ if for some $r\geq 0$, $A\subseteq N_r(B)$. Given two coarse equivalence classes $[A]$ and $[B]$ in $X$, the \emph{coarse intersection} $[A]\cap [B]$ is defined to be $[N_r(A)\cap N_r(B)]$ for all sufficiently large $r$. This is not always a well defined coarse equivalence class, but it will be if $X$ is a group and $A$ and $B$ are coarsely equivalent to subgroups \cite[Lemma 2.2]{mosher2011quasiactions}.

\subsection{Trees of Spaces and Quasi-Isometries}
We first fix notation and terminology for graphs.  A graph $\Gamma$ consists of the set $V\Gamma$ of  \emph{vertices}, the set  $E\Gamma$ of \emph{oriented edges} and maps $\iota,\tau:E\Gamma\rightarrow V\Gamma$ which map an oriented edge to its initial and terminal vertices respectively. For every $e\in E\Gamma$, we let $\bar e$ denote the edge with opposite orientation to $e$. We call the pair $\{e,\bar e\}$ an \emph{unoriented edge}. A \emph{graph isomorphism} $\chi:\Gamma\rightarrow \Gamma'$ is map which bijectively sends vertices to vertices, edges to edges, and commutes with the functions $\iota$ and $\tau$. A graph is  \emph{simplicial} if it has no one edge loops and at most one unoriented edge  between any pair  of vertices.

A subgraph $\Gamma'\subseteq \Gamma$  is \emph{induced} if whenever vertices $v,w\in V \Gamma'$ are joined by an edge in $\Gamma$, they are joined by an edge of $\Gamma'$. An induced subgraph is uniquely defined by its vertex set. The \emph{link} $\mathrm{lk}(v)$ of a vertex $v$ is defined to be the induced subgraph whose vertex set consists of vertices adjacent to $v$. The \emph{star}  $\mathrm{star}(v)$ of a vertex $v$ is the induced subgraph whose vertex set consists of $v$ and all adjacent vertices. 
When ambiguous, we write $\mathrm{lk}^\Gamma(v)$ and $\mathrm{star}^\Gamma(v)$ to denote the link and star of some $v\in V\Gamma$.
We remark that these definitions of link and star are not standard.

We use the following notation for trees of spaces as in \cite{cashenmartin2017quasi}.
\begin{defn}
A tree of spaces $X:=X(T,\{X_v\}_{v\in VT},\{X_e\}_{e\in ET},\{\alpha_e\}_{e\in ET})$ consists of:
\begin{enumerate}
\item a simplicial tree $T$ known as the \emph{base tree};
\item a metric space $X_v$ for each vertex $v\in VT$ known as a  \emph{vertex space};
\item a subspace $X_e\subseteq X_{\iota e}$ for each oriented edge $e$ of $T$ known as an \emph{edge space};
\item maps $\alpha_e:X_e\rightarrow X_{\overline{e}}$ for each edge $e\in ET$, such that $\alpha_{\bar e}\circ\alpha_e=\mathrm{id}_{X_e}$ and $\alpha_{e}\circ\alpha_{\bar e}=\mathrm{id}_{X_{\bar e}}$.
\end{enumerate}
We consider $X$ as a metric space as follows: we take the disjoint union of all  the $X_v$ and then, for all unoriented edges $\{e,\overline{e}\}$ and every $x\in X_e$, we attach a unit interval between  $x\in X_e$ and $\alpha_e(x)\in X_{\overline{e}}$.
\end{defn}
As noted in \cite{cashenmartin2017quasi}, it is easy to verify that $X$ is actually a metric space, i.e. that the quotient pseudometric on $X$ is actually a metric.
In \cite{cashenmartin2017quasi}, a more general definition is used in which  each $\alpha_{\bar e}\circ\alpha_e$ is only assumed to be close to the identity. That  will not be needed here, since all trees of spaces will arise `algebraically' using  Proposition  \ref{prop:qitreebs}.

We assume familiarity with Bass-Serre theory. See \cite{serre1977arbres} for details. Suppose $\mathcal{G}$ is a finite graph of finitely generated groups and $G:=\pi_1(\mathcal{G})$. Let $T$ be the Bass-Serre tree of $\mathcal{G}$. Every vertex or edge of $T$ is identified with a coset $gG_v$ or $gG_e$ of a vertex or edge group of  $\mathcal{G}$.

The following proposition, whose argument is standard, relates the algebra of the Bass-Serre tree to the geometry of the corresponding tree of spaces obtained by `blowing up' the Bass-Serre tree.

\begin{prop}[Section 2.5 of  \cite{cashenmartin2017quasi}] \label{prop:qitreebs}
Suppose $\mathcal{G}$, $G$ and $T$ are as above. Then there exists  a tree of spaces $X$ with base tree $T$ and a quasi-isometry $f:G\rightarrow X$. Moreover, there is a constant $A\geq 0$, such that $d_\mathrm{Haus}(f(g G_v),X_{\tilde{v}})\leq A$ and $d_\mathrm{Haus}(f(g G_e),X_{\tilde{e}})\leq A$ for all $\tilde v=gG_v \in VT$ and $\tilde e=gG_e\in ET$.
\end{prop}
\begin{proof}
We choose a generating set of $G$ that is a union of generating sets of each of the vertex groups of $\mathcal{G}$, as well as generators of the fundamental group of the underlying graph of $\mathcal{G}$. We define the vertex space of each $\tilde v=gG_v\in VT$ to be the coset $gG_v$, thought of as a subspace of $G$ equipped with the word metric. 
Suppose $\tilde e=gG_e$ is an edge of $T$. There are monomorphisms $\beta_0:G_e\rightarrow G_{\iota(e)}$ and $\beta_1:G_e\rightarrow G_{\tau(e)}$ associated with the edge $e$ of $\mathcal{G}$.
We define $X_{\tilde e}:=g\cdot\mathrm{im}(\beta_0)\subseteq X_{\iota (\tilde e)}$ and $\alpha_{\tilde e}$ by  $gx\mapsto g(\beta_1\circ \beta_0^{-1})(x)\subseteq X_{\tau (\tilde e)}$.

It is easy to verify that $X$ is a proper geodesic metric space (see Lemma 2.13 of \cite{cashenmartin2017quasi}) and that $G$ acts properly discontinuously and cocompactly on $X$. Hence for any $x\in X$, the orbit map $g\mapsto g\cdot x$ is a quasi-isometry. We let $\Omega\subseteq VT\cup ET$ be the finite set of edges and vertices corresponding to cosets of the form $G_v$ or $G_e$. We then define $A:=\max_{\omega\in \Omega}\{d_X(x,X_\omega)\}$, and it is straightforward to verify the final claim.
\end{proof}
Proposition \ref{prop:qitreebs} will be implicitly used throughout this paper, allowing us to switch between vertex and edge spaces in a tree of spaces, and the cosets of vertex and edge groups they correspond to in the Bass-Serre tree.

The following proposition  explains how to build a quasi-isometry between trees of spaces by patching together quasi-isometries of vertex spaces. This can be done if quasi-isometries on adjacent vertex spaces agree up to uniformly bounded error on their common edge space.

\begin{prop}[{\cite[Corollary 2.16]{cashenmartin2017quasi}}]\label{prop:qitree}
Let $K\geq 1$ and $A\geq 0$.
Suppose  that $X:=X(T,\{X_v\},\{X_e\},\{\alpha_e\})$ and $X':=X'(T',\{X'_v\},\{X'_e\},\{\alpha'_e\})$ are trees of spaces, and that there is a tree isomorphism $\chi:T\rightarrow T'$. Suppose for every $v\in VT$ and $e\in ET$ there is a $(K,A)$-quasi-isometry $\phi_v:X_v\rightarrow X'_{\chi(v)}$ and $\phi_e:X_e\rightarrow X'_{\chi(e)}$. 
Suppose also that for every  $e\in ET$,  the following diagram commutes up to uniformly bounded error $A$. Then there is a quasi-isometry $\phi: X\rightarrow X'$ such that $\phi|_{X_v}=\phi_v$ for every $v\in VT$.
\begin{figure}[h!]
\begin{tikzcd}  
X_{\iota e} \arrow{r}{\phi_{\iota e}} &   X'_{\chi(\iota e)}\\
X_{e} \arrow{r}{\phi_{ e}} \arrow[hookrightarrow]{u} \arrow{d}{\alpha_e}&   X'_{\chi( e)}\arrow[hookrightarrow]{u}\arrow{d}{\alpha'_{\chi(e)}}\\
X_{\tau e} \arrow{r}{\phi_{\tau e}} & X'_{\chi (\tau e)}
\end{tikzcd}
\end{figure}

\end{prop}

\subsection{JSJ trees of cylinders and relative quasi-isometries}

We suppose $T$ is the Bass-Serre tree associated to a finite graph of finitely generated groups with two-ended edge groups, and that $X$ is the associated tree of spaces as in Proposition \ref{prop:qitreebs}.  

 Guirardel and Levitt \cite{guirardel2011cylinders} associate the \emph{tree of cylinders} $T_\mathrm{cyl}$ to $T$. This is defined as follows: two edges of $T$  are \emph{equivalent} if their edge stabilizers are commensurable. The union of all edges in an equivalence class, which is shown to be a subtree of $T$, is called a \emph{cylinder}. 
The tree of cylinders is then defined to be a bipartite tree with vertex set $V_0\sqcup V_1$, where $V_0$ is the set of cylinders of $T$ and $V_1$ consists of vertices of $T$ that lie in at least two cylinders. There is an edge between $v\in V_0$ and $w\in V_1$ precisely when the vertex $w\in VT$ is contained in the cylinder $v$. We remark that even if edge stabilizers of $T$ are two-ended, edge stabilizers of $T_\mathrm{cyl}$ may not be.

 In light of Proposition \ref{prop:sbgpcommen}, there is an alternative characterisation of cylinders: two edges lie in the same cylinder if and only if the associated edge spaces of $X$ are coarsely equivalent. This is because every subgroup of the form $gHg^{-1}$ is coarsely equivalent to the coset $gH$. Cylinders thus correspond to coarse equivalence classes of edge spaces in the tree of spaces. Henceforth, we use this formulation of cylinders, which is more natural in the coarse geometric context.

A JSJ decomposition of a finitely presented group $G$ is a finite graph of groups that  encodes all possible splittings of $G$ over two-ended subgroups. (In general, they can be  defined over  other classes of subgoups, but we restrict to JSJ decompositions over two-ended subgroups here.)  They always exist for one-ended finitely presented groups.  Although JSJ decompositions are not generally unique, they all lie in the same deformation space. This ensures that their corresponding Bass-Serre trees all have the same tree of cylinders. We call this canonical tree the \emph{JSJ tree of cylinders}. See  \cite{forrester2002deforation} and \cite{guirardel2017jsj} for details and proofs of the above statements.

Every JSJ tree of cylinders has three \emph{types} of vertex: rigid, (quadratically) hanging or cylindrical. The $V_0$-vertices are  cylindrical and $V_1$-vertices are either rigid or hanging. We do not define what hanging and rigid vertices are here. One-ended right angled Artin groups have no quadratically hanging vertices, so all non-cylindrical vertices are  rigid.

In many interesting cases, e.g. hyperbolic groups, the JSJ tree of cylinders is itself a JSJ tree. In particular, the JSJ tree described by Bowditch \cite{bowditch1998cut} is a JSJ tree of cylinders.  For one-ended RAAGs this will never be the case --- edge stabilizers of the JSJ tree of cylinders are always one-ended, thus the JSJ tree of cylinders is never a JSJ tree.

In \cite{cashenmartin2017quasi}, relative quasi-isometries are defined. To do this, a \emph{peripheral structure} ---  a collection of coarse equivalence classes of subspaces --- is assigned to a space. A relative quasi-isometry is a quasi-isometry  that preserves the peripheral structure. This works well when the JSJ trees of cylinders has two-ended cylinder stabilizers, but needs to be modified when this is not the case. This is illustrated by the following example.

\begin{exmp}\label{exmp:peripheral}
Suppose a graph of groups  decomposition has  a vertex group $\mathbb{Z}^2=\langle a,b\mid [a,b]\rangle$, incident to exactly two edges whose edge groups are $\langle a \rangle$ and $\langle b \rangle$ respectively. In the Bass-Serre tree $T$, a lift $v$ of this vertex is contained in exactly two cylinders, one containing edges corresponding to cosets of the form $b^i\langle a \rangle$, and the other containing cosets of the form $a^i\langle b \rangle$. 
In the tree of cylinders $T_\mathrm{cyl}$, the $V_1$-vertex $v$ is incident to exactly two edges. Both of the  corresponding edge spaces are  coarsely equivalent to the vertex space $X_v$; this is because $\cup_{i\in \mathbb{Z}}(b^i\langle a \rangle)=\cup_{i\in \mathbb{Z}}(a^i\langle b \rangle)=\mathbb{Z}^2$. If we simply give $X_v$ a peripheral structure  consisting of coarse equivalence classes of adjacent edge spaces, then both elements in the peripheral structure are equal, even though they lie in different cylinders. 

The right thing to do is to think  of an element of the peripheral structure  as a family of parallel lines. A relative quasi-isometry ought to preserve coarse equivalence classes of lines themselves, not just the coarse equivalence class of their union. Doing this allows us to distinguish between the two incident edge spaces.
\end{exmp}
This example motivates the following definition. We caution the reader that as far as we are aware, this definition is not particularly intuitive or useful outside the specific context in which we use it: namely vertex spaces of the JSJ tree of cylinders. 

\begin{defn}\label{defn:relqi}

A \emph{peripheral structure} $\mathcal{P}_X$ on a metric space $X$  is a set containing elements of the form $\{A_i\}_{i\in I}$, where each $A_i\subseteq X$ and $[A_i]=[A_j]$ for all $i,j\in I$. A  \emph{$(K,A)$-relative quasi-isometry} $(f,f_*):(X,\mathcal{P}_X)\rightarrow (Y,\mathcal{P}_Y)$  consists of a  $(K,A)$-quasi-isometry $f:X\rightarrow Y$ and a bijection $f_*:\mathcal{P}_X\rightarrow\mathcal{P}_Y$ such that if $f_*(\{A_i\}_{i\in I})=\{B_j\}_{j\in J}$, then:
\begin{enumerate}
\item for every $i\in I$, there is some $j\in J$ such that $d_\mathrm{Haus}(f(A_i),B_j)\leq A$;
\item for every $j\in J$, there is some $i\in I$ such that $d_\mathrm{Haus}(f(A_i),B_j)\leq A$.
\end{enumerate}
We say that $(f,f_*)$ is a \emph{relative quasi-isometry} if it is a $(K,A)$-relative quasi-isometry for some $K\geq 1$ and $A\geq 0$.  To simplify notation, we often suppress the map between peripheral structures,  denoting $(f,f_*)$ by $f$.
We recover the standard notion of relative quasi-isometry (as used in \cite{cashenmartin2017quasi} for instance) if for every $\{A_i\}_{i\in I}\in \mathcal{P}_X$, $\lvert I \rvert=1$.
\end{defn}

\begin{exmp}\label{exmp:peripheralz2qi}
Let $G=\mathbb{Z}^2=\langle a,b\mid [a,b]\rangle$. Let $\mathcal{P}_1:=\{\{b^i\langle a\rangle\mid i\in\mathbb{Z}\}_{}\}$, $\mathcal{P}_2:=\{\{a^i\langle b\rangle\mid i\in\mathbb{Z}\}\}$ and $\mathcal{P}_3:=\mathcal{P}_1\cup \mathcal{P}_2$. The map $\mathrm{id}_G$ is not a relative quasi-isometry $(G,\mathcal{P}_1)\rightarrow (G,\mathcal{P}_2)$, however  the isomorphism $f:G\rightarrow G$ given by $a^ib^j\mapsto a^jb^i$ is a relative quasi-isometry $(G,\mathcal{P}_1)\rightarrow (G,\mathcal{P}_2)$. We note that  $(G,\mathcal{P}_3)$ is not relatively quasi-isometric to either $(G,\mathcal{P}_1)$ or $(G,\mathcal{P}_2)$.
Although $\mathcal{P}_4:=\cup_{i\in \mathbb{Z}}\{\{b^i\langle a\rangle\}\}$ is a  well-defined peripheral structure on $G$, it should not be confused with $\mathcal{P}_1$, since $\mathcal{P}_4$ contains infinitely many elements, whereas $\mathcal{P}_1$ contains a single element.
\end{exmp}

We now explain how to give vertex spaces of the JSJ trees of cylinders a peripheral structure. Let $T_\mathrm{JSJ}$ be a JSJ tree and let $T_\mathrm{cyl}$ be its tree of cylinders. Each edge $e\in ET_\mathrm{cyl}$ is incident to some $V_0$-vertex $v$ and  $V_1$-vertex $w$. We thus define
 $$\mathcal{S}_e:=\{X_f\mid f\in ET_\mathrm{JSJ},  f \textrm{ is incident to } w\textrm{ and } f \textrm{ is contained in the cylinder }v \}$$ and for each vertex $v\in T_\mathrm{cyl}$, we give $X_v$ the peripheral structure 
$$\mathcal{P}_v:=\{\mathcal{S}_e\mid \textrm{ for all } e\in ET_\mathrm{cyl} \textrm{ such that } e \text{ is adjacent to }v\}.$$
\begin{rem}\label{rem:defnofperiph}
It would be more accurate, though notationally inconvenient, to refer to elements of $\mathcal{P}_v$ as pairs $(\mathcal{S}_e,e)$. This is because it may be the case that $S_e=S_{e'}$ for some $e\neq e'$, in which case we have two distinct elements $S_e$ and $S_{e'}$ of $\mathcal{P}_v$.  This only occurs in the exceptional case where some edge $f=(w,w')\in ET_\mathrm{JSJ}$ is the only edge contained in its cylinder $v$. Then $e=(v,w)$ and $e'=(v,w')$ are distinct edges of $ET_\mathrm{cyl}$ with $\mathcal{S}_e=\mathcal{S}_{e'}=\{X_f\}$. This cannot occur for the JSJ tree of cylinders of a one-ended RAAG over two ended subgroups, as shown in Section \ref{sec:jsjtoc}.
\end{rem}
 In Remark \ref{rem:peripheralstrucutreraags}, we give a more refined description of the peripheral structure of vertex spaces in the JSJ tree of cylinders of a RAAG.

Being relatively quasi-isometric  is an equivalence relation among pairs $(X,\mathcal{P}_X)$. We define the \emph{relative quasi-isometry class} $[[(X,\mathcal{P}_X)]]$ to be the equivalence class consisting of all pairs $(Y,\mathcal{P}_Y)$ such that there exists a relative quasi-isometry $(X,\mathcal{P}_X)\rightarrow (Y,\mathcal{P}_Y)$. We let $\mathrm{QI}(X,\mathcal{P}_X)$ denote the set of all relative quasi-isometries $(X,\mathcal{P}_X)\rightarrow (X,\mathcal{P}_X)$.

The JSJ tree of cylinders over two-ended groups is a quasi-isometry invariant. More precisely, we have the following: 
\begin{thm}[\cite{papasoglu2005quasi},\cite{vavrichek2013commensurizer}]\label{thm:jsjtocqi}
Let $G$ and $G'$ be finitely presented groups and $T$ and $T'$ be their JSJ trees of cylinders. Let $X$ and $X'$ be the associated trees of spaces of $T$ and $T'$ respectively.  
Then for any quasi-isometry $f:G\rightarrow G'$, there is a constant $C$ and a unique tree isomorphism $f_*:T\rightarrow T'$ such that the following holds. 

For every vertex $v$ of $T$, $f(X_v)$ has Hausdorff distance at most $C$ from $X'_{f_*(v)}$, and $v$ and  $f_*(v)$ are of  the same vertex type (cylinder, hanging or rigid).
Moreover, $f$ restricted to each vertex space $X_v$  is close to a relative quasi-isometry $(X_v,\mathcal{P}_v)\rightarrow (X_{f_*(v)},\mathcal{P}_{f_*(v)})$.
\end{thm}
The result is essentially the same as Theorem 2.8 of \cite{cashenmartin2017quasi}, which can be deduced using work of  Papasoglu \cite{papasoglu2005quasi}  and Vavrichek \cite{vavrichek2013commensurizer}. However,  since we use a slightly different  definition of relative quasi-isometry to that found in \cite{cashenmartin2017quasi}, we give a careful proof of Theorem \ref{thm:jsjtocqi} below.
We first prove the following Lemma.

\begin{lem}\label{lem:treeedgespacesym}
Let $\mathcal{G}$ be a finite graph of groups with finitely generated edge and vertex groups. Let $T$ be the associated Bass-Serre tree and let $X$ be the tree of spaces. There is a proper non-decreasing function $\phi:\mathbb{R}_{\geq 0}\rightarrow \mathbb{R}_{\geq 0}$ such that if $X_e$ and $X_{e'}$ are edge spaces such that $[X_e]=[X_{e'}]$ and $X_e'\subseteq N_r(X_{e})$ for some $r\geq 0$, then $\mathrm{d}_\mathrm{Haus}(X_e',X_e)\leq \phi(r)$.
\end{lem}
\begin{proof}
For some fixed $r\geq 0$ and edge space $X_e$, there are only finitely many edge spaces such that $[X_e]=[X_{e'}]$ and $X_e'\subseteq N_r(X_{e})$. This follows from Proposition \ref{prop:qitreebs} and the observation that there are only finitely many cosets of edge groups of $\mathcal{G}$ that are contained in some $N_r(gG_e)$. We pick $R_e\geq 0$ such that $\mathrm{d}_\mathrm{Haus}(X_e',X_e)\leq R_e$ for all such edge spaces $X_{e'}$. The result holds as there are only finitely many $\pi_1(\mathcal{G})$-orbits of edge spaces in $T$.
\end{proof}

\begin{proof}[Proof of Theorem \ref{thm:jsjtocqi}]
Since every vertex space is stabilized by a subgroup, the coarse intersection of vertex spaces is always well-defined. We are thus able to deduce the following elementary facts about coarse intersection and containment of edge and vertex spaces in a JSJ tree (not the JSJ tree of cylinders).
\begin{enumerate}
\item \label{inthm:jsjtocqi2}If  $X_v$ and $X_w$ are distinct vertex spaces, then their coarse intersection is coarsely contained in the edge space of any edge in the segment $[v,w]$. In particular, the coarse intersection must thus either be bounded or two-ended; in the later case it is coarsely equivalent to an edge space $X_e$ for any $e\in[v,w]$.
\item \label{inthm:jsjtocqi3} If some edge space $X_e$ is coarsely contained in a vertex space $X_v$, 
 let $e'$ be the edge adjacent to $v$ contained in a geodesic joining $v$ to $e$. Then $[X_e]=[X_e']$. In particular, we deduce that $X_e$ is coarsely contained in $X_v$ if and only if $v$ is contained in the cylinder containing $e$. Moreover, if $X_e\subseteq N_r(X_v)$ for some $r$, then $X_e\subseteq N_r(X_{e'})$.
\end{enumerate}
We let $T_\mathrm{JSJ}$ and $T'_\mathrm{JSJ}$ be  JSJ trees of $G$ and $G'$ respectively, and let $T_\mathrm{cyl}$ and $T'_\mathrm{cyl}$ be their trees of cylinders.

By \ref{inthm:jsjtocqi3}, a vertex $v$ of  $T_\mathrm{JSJ}$ is a  $V_1$-vertex of $T_\mathrm{cyl}$ if and only if it coarsely contains  two edge spaces that  are not coarsely equivalent. In particular, this means $V_1$-vertex spaces cannot be bounded or two-ended. Thus  Theorem 7.1 of \cite{papasoglu2005quasi} tells us that  there exists a  constant $C_1$ such that for every $V_1$-vertex $v\in T_\mathrm{cyl}$, there is a $V_1$-vertex $v'\in T_\mathrm{JSJ}$ with $d_\mathrm{Haus}(f(X_v),X_{v'})\leq C_1$. 

We note that $v'$ is unique. Indeed, suppose $[X_{w}]=[X_{v'}]$, where $w\in VT'_\mathrm{cyl}$ is a $V_1$-vertex. If $w\neq v'$, then $[X_w]=[X_{v'}]=[X_w]\cap[X_{v'}]$. By \ref{inthm:jsjtocqi2}, $X_{v'}$ must therefore either be bounded or two-ended,  contradicting the fact that $v'$ is a $V_1$-vertex.

Now suppose $v\in VT_\mathrm{cyl}$ is a $V_0$-vertex corresponding to a cylinder containing some $e\in ET_\mathrm{JSJ}$. 
 By Theorem 7.1 of \cite{papasoglu2005quasi}, we know there is an  $e'\in ET'_\mathrm{JSJ}$ such that $[f(G_e)]=[G_{e'}]$. Let $v'\in VT'_\mathrm{cyl}$ be the cylinder containing $e'$. We claim that $[f(X_v)]=[X_{v'}]$. Note that  $v'$ is uniquely determined, independent of the choice of $e'$.

This claim follows from the discussion in Section 6 of \cite{vavrichek2013commensurizer} (see also \cite{guirardel2011cylinders}). There are two cases depending of the number of coends of $G_e$ in $G$: if $G_e$ has at least four coends, then the claim follows from Theorem 4.1 of \cite{vavrichek2013commensurizer}. Otherwise, the claim follows from \cite{papasoglu2005quasi}. This can be done uniformly, i.e.  there is a constant $C_2\geq 0$ such that for every $V_0$-vertex $v\in T$, there is a $V_0$-vertex $v'\in T$ such that $d_\mathrm{Haus}(f(X_v),X_{v'})\leq C_2$.

The above discussion tells us that there exists a unique map $f_*:VT_\mathrm{cyl}\rightarrow VT_\mathrm{cyl}'$.  By applying the above argument to the coarse inverse $g:G'\rightarrow G$ of $f$, we deduce that $f_*$ is in fact a bijection. We claim that a $V_0$-vertex $v\in VT_\mathrm{cyl}$ and a $V_1$-vertex $w\in VT_\mathrm{cyl}$ are joined by an edge if and only if $f_*(v)$ and $f_*(w)$ are also joined by an edge. This follows from \ref{inthm:jsjtocqi3} and the fact that for an edge $e\in T_\mathrm{JSJ}$, $X_e$ is coarsely contained in a $V_1$-vertex space $X_w$ if and only if $X_{f_*(e)}$ is coarsely contained in $X_{f_*(w)}$.

We now claim that for every vertex $v\in VT_\mathrm{cyl}$, the map $f|_{X_v}$ is close to a relative quasi-isometry $(X_v,\mathcal{P}_v)\rightarrow (X_{v'},\mathcal{P}_{v'})$. By Theorem 7.1 of \cite{papasoglu2005quasi} applied to $f$ and its coarse inverse, there is a $C_3\geq 0$ such that: \begin{enumerate}
\item \label{inthm:jsjtocqii} for all  $t\in ET_\mathrm{JSJ}$, there is a $t'\in ET'_\mathrm{JSJ}$ such that $\mathrm{d}_\mathrm{Haus}(f(X_t),X_{t'})\leq C_3$;
\item \label{inthm:jsjtocqiii}for all $t'\in ET'_\mathrm{JSJ}$, there is a $t\in ET_\mathrm{JSJ}$ such that $\mathrm{d}_\mathrm{Haus}(f(X_t),X_{t'})\leq C_3$.
\end{enumerate}
Pick some edge $e=(v,w)\in ET_\mathrm{cyl}$ and let $\mathcal{S}_e$ be as defined above. Let $e':=(v',w')=(f_*(v),f_*(w))$.  Without loss of generality, we may assume $v$ is a cylinder and $w$ is a $V_1$-vertex. Let $X_t\in \mathcal{S}_e$, where $t\in ET_\mathrm{JSJ}$ is an edge adjacent to $w$ and contained in the cylinder $v$. 

Condition \ref{inthm:jsjtocqii} tells us there is a $t'\in ET'_\mathrm{JSJ}$ such that $\mathrm{d}_\mathrm{Haus}(f(X_t),X_{t'})\leq C_3$. This is enough if $t'$ happens to be adjacent to $w'$, since then $X_{t'}\in \mathcal{S}_{e'}$; however, this is not necessarily the case. Since $X_t\subseteq N_1(X_w)$, we  deduce that $X_{t'}\subseteq N_D(X_{w'})$, where $D$ depends only on $C_1$, $C_3$ and the quasi-isometry constants of $f$.  By \ref{inthm:jsjtocqi3}, there is an edge $t''\in ET_\mathrm{JSJ}$, adjacent to $w'$, such that $X_{t'}\subseteq N_D(X_{t''})$. We note that $X_{t''}\in \mathcal{S}_{e'}$. By Lemma \ref{lem:treeedgespacesym}, we deduce that $d_\mathrm{Haus}(f(X_{t}),X_{t''})\leq C_3+\phi(D)$, which demonstrates the first condition in Definition \ref{defn:relqi}. The other condition  is proved similarly, using \ref{inthm:jsjtocqiii} and Lemma \ref{lem:treeedgespacesym}.
\end{proof}

It light of Theorem \ref{thm:jsjtocqi}, it is natural to ask whether one can determine the quasi-isometry type of a group from its JSJ decomposition. This question has been studied by \cite{cashenmartin2017quasi}, who have obtained positive results in many interesting cases where cylinder stabilisers are two-ended. We investigate this question for RAAGs, where cylinder stabilisers are one-ended.

\subsection{Decorated Trees}\label{subsec:decoratedtrees}
The material in this subsection is adapted from \cite{cashenmartin2017quasi}.
\begin{defn}
Let $G$ be a finitely presented group and let $T$ be its JSJ tree of cylinders. Let $\mathcal{O}$ be a set and  $\delta: VT\sqcup ET\rightarrow \mathcal{O}$ be a map such that $\delta(e)=\delta(\bar e)$ for every edge $e\in ET$. We call the map $\delta$ a \emph{decoration} and call $\mathcal{O}$ the set of \emph{ornaments}. We call the pair $(T,\delta)$ a \emph{decorated tree}.
\end{defn}
We may think of a decorated tree as a tree in which each vertex and unoriented edge is labelled with an element of $\mathcal{O}$. For instance, if $\mathcal{O}$ is a set of colours, then a decorated tree is just a coloured tree. We assume throughout that no edge has the same ornament as a vertex, thus $\mathcal{O}$ can be thought of as a disjoint union of vertex and edge ornaments.

\begin{defn}\label{defn:equiv}
Let $(T,\delta)$ and $(T,\delta')$ be two decorated trees with the same set of ornaments. Then we say that $(T,\delta)$ and $(T,\delta')$ are \emph{equivalent} if there exists a decoration preserving tree isomorphism $\chi:T\rightarrow T'$.
\end{defn}
Both weak equivalence and equivalence  as described in the introduction, are special cases of Definition \ref{defn:equiv} when $T$ is endowed with a specific decoration, namely the na{\"\i}ve and  embellished decorations, which we define below and in Section \ref{sec:embellished}  respectively.

Rather than using Definition \ref{defn:equiv} directly, we use Proposition \ref{prop:equivbijec}  to determine whether two decorated JSJ trees of cylinders are equivalent.
Since we are interested in determining the quasi-isometry type of a group, we will work with decorations that are invariant under quasi-isometry in the following sense:
\begin{defn}
Let $G$ be a finitely presented group and let $T$ be its JSJ tree of cylinders. A decoration $\delta:T\rightarrow \mathcal{O}$ is said to be \emph{geometric} if for any quasi-isometry $\phi:G\rightarrow G$ with associated tree isomorphism $\phi_*:T\rightarrow T$ as in Theorem \ref{thm:jsjtocqi}, we have $\delta\circ \phi_*=\delta$. 
\end{defn}

To any JSJ tree of cylinders, we can associate the following geometric decoration. We decorate each vertex $v$ with its type (cylinder, rigid or hanging) and its relative quasi-isometry class $[[X_v,\mathcal{P}_v]]$. We call this  the \emph{na{\"\i}ve decoration}. 
Theorem \ref{thm:jsjtocqi}  ensures that this is indeed a geometric decoration.

The vertices and edges of a decorated tree are naturally endowed with an equivalence relation as follows: two vertices or edges are equivalent if and only if they have the same ornament. We say a decoration $\delta':T\rightarrow \mathcal{O}'$ is a \emph{refinement} of $\delta:T\rightarrow \mathcal{O}$ if the partition of vertices and edges into $\mathcal{O}'$-classes is finer than the partition into $\mathcal{O}$-classes.

We now explain two procedures described in \cite{cashenmartin2017quasi} that, given any geometric decoration $\delta:T\rightarrow \mathcal{O}$, output a new geometric decoration $\delta'$ that refines $\delta$. The underlying principle is that there may be some vertices or edges $x,y$ such that $\delta(x)=\delta(y)$, but there is some obvious obstruction to the existence of a  quasi-isometry $\phi:G\rightarrow G$ such that $\phi_*(x)=y$. In such a case, we refine $\delta$ to eliminate this  obstruction.

The first of these procedures is called \emph{neighbour refinement}. The idea is as follows: 
we suppose $\mathcal{O}$ contains two elements that we call $\mathtt{red}$ and $\mathtt{blue}$, and that $T$ contains two \texttt{red} vertices $v$ and $w$. Suppose also that no \texttt{blue}  edge  is incident to $v$, but some \texttt{blue} edge is incident to $w$. Since the decoration is geometric, there can be no quasi-isometry $\phi:G\rightarrow G$ such that $\phi_*(v)=w$, since $\phi_*$ must send \texttt{blue} edges to \texttt{blue} edges. 

We refine $\delta$ by assigning different ornaments to \texttt{red} vertices incident to no \texttt{blue} edge, \texttt{red} vertices incident to countably infinitely many \texttt{blue} edges, \texttt{red} vertices incident to one \texttt{blue} edge, \texttt{red} vertices incident to two \texttt{blue} edges etc. 
Neighbour refinement applies the above procedure to all distinct pairs $\{\texttt{red},\texttt{blue}\}\subseteq\mathcal{O}$ simultaneously.   We refer the reader to \cite[Section 3.2]{cashenmartin2017quasi} for a formal exposition of this.

The second procedure  is called \emph{vertex refinement}. We first observe that a decorated tree of cylinders $(T,\delta)$ induces a decoration $\delta_v:\mathcal{P}_v\rightarrow \mathcal{O}$ on the peripheral structure  via the map $\mathcal{S}_e\mapsto \delta(e)$. This is always well-defined by Remark \ref{rem:defnofperiph}.  Suppose $(\phi_v,(\phi_v)_*):(X_v,\mathcal{P}_v)\rightarrow (X_w,\mathcal{P}_w)$ is a relative quasi-isometry. For ease of notation, we therefore identify $(\phi_v)_*$ with the corresponding map that sends an edge adjacent to $v$ to an edge adjacent to $w$.

Suppose that $v$ and $w$ are  vertices such that $\delta(v)=\delta(w)$.   Suppose also that there is no relative quasi-isometry $(X_v,\mathcal{P}_v)\rightarrow (X_w,\mathcal{P}_w)$ that preserves the decoration on the peripheral structure. Then there cannot be a quasi-isometry $\phi:G\rightarrow G$ such that $\phi_*(v)=w$. We thus refine $\delta$ by assigning different ornaments to $v$ and $w$.

Now suppose there are edges $e=(v,w)$ and $f=(v',w')$ such that $\delta(e)=\delta(f)$. Suppose also that either there is no relative  quasi-isometry $\phi_v:X_v\rightarrow X_{v'}$  such that $(\phi_v)_*(e)=f$, or there is no relative quasi-isometry $\phi_w:X_w\rightarrow X_{w'}$ such that $(\phi_w)_*(\overline{e})=\overline{f}$. In either case, there cannot be a quasi-isometry $\phi:G\rightarrow G$ such that $\phi_*(e)=f$. We thus refine $\delta$ by assigning different ornaments to $e$ and $f$.

Vertex refinement was originally defined in \cite[Section 5.3]{cashenmartin2017quasi}. Our approach needs minor modifications, so we explain the details here. We define  $\mathcal{Q}:=\{[[X_v,\mathcal{P}_v]]\mid v\in VT\}$. For each $Q\in \mathcal{Q}$, we pick some metric space $Z_Q$ with peripheral structure $\mathcal{P}_Q$ such that $[[Z_Q,\mathcal{P}_Q]]=Q$.

We recall every $\phi\in \mathrm{QI}(Z_Q,\mathcal{P}_Q)$ induces a bijection $\phi_*:\mathcal{P}_Q\rightarrow \mathcal{P}_Q$.  We define \begin{align*}
\mathcal{O}'_V&:=\mathcal{O}\times \bigg(\bigsqcup_{Q\in \mathcal{Q}} \mathrm{QI}(Z_Q,\mathcal{P}_Q)\backslash\mathcal{O}^{\mathcal{P}_Q}\bigg),
\end{align*}
where some $\phi\in\mathrm{QI}(Z_Q,\mathcal{P}_Q)$ acts on  $\rho\in\mathcal{O}^{\mathcal{P}_Q}$  via $\phi\cdot\rho:= \rho\circ \phi^{-1}_*$. Similarly, we define 
\begin{align*}
\mathcal{O}'_E&:=\mathcal{O}\times \bigg(\bigsqcup_{Q\in \mathcal{Q}}\mathrm{QI}(Z_Q,\mathcal{P}_Q)\backslash(\mathcal{O}^{\mathcal{P}_Q},\mathcal{P}_Q) \bigg)^2,
\end{align*}
where some $\phi\in\mathrm{QI}(Z_Q,\mathcal{P}_Q)$ acts on  $(\rho,\mathcal{S}_e)\in(\mathcal{O}^{\mathcal{P}_Q},\mathcal{P}_Q)$ via $$\phi\cdot (\rho,\mathcal{S}_e):= \big(\rho\circ \phi^{-1}_*,\phi_*(\mathcal{S}_{e})\big).$$

The vertex refinement $\delta':VT\sqcup ET\rightarrow \mathcal{O}'_V\sqcup \mathcal{O}'_E $ of $\delta$ is defined as follows. For every vertex $v\in VT$, we choose a relative quasi-isometry $(\mu_v,(\mu_v)_*):(X_v,\mathcal{P}_v)\rightarrow (Z_Q,\mathcal{P}_Q)$ for $Q=[[(X_v,\mathcal{P}_v)]]$. We recall $\delta_v$ is the decoration on $\mathcal{P}_v$ induced by $\delta$. We thus define $$\delta'(v):=\bigg(\delta(v),\mathrm{QI}(Z_Q,\mathcal{P}_Q)\cdot( \delta_v \circ (\mu_v)_*^{-1})\bigg)\in \mathcal{O}'_V.$$

We now define $\delta'(e)$ for each edge $e=(v,w)$ where $v$ is a cylinder.  This defines $\delta'$ on all edges, since we know that  $\delta'(e)$ and  $\delta'(\bar e)$ must be equal. 
We define \begin{align*}\delta'(e):=\bigg(\delta(e),\mathrm{QI}(Z_Q,\mathcal{P}_Q)\cdot\Big( \delta_v \circ (\mu_v)_*^{-1},(\mu_v)_*(\mathcal{S}_e)\Big),\\
\mathrm{QI}(Z_Q,\mathcal{P}_Q)\cdot\Big( \delta_w \circ (\mu_w)_*^{-1},(\mu_w)_*(\mathcal{S}_{\overline{e}})\Big)\bigg)\in \mathcal{O}'_E.\end{align*}
The decoration $\delta'$ is independent of the choice of maps $\{\mu_v\}_{v\in VT}$.

Unlike vertex refinement in \cite{cashenmartin2017quasi},  we don't keep track of partial orientations, which make no sense if edge spaces aren't two-ended. Although edge spaces do have canonical two-ended factors, the procedure described in Remark \ref{rem:flip} allows one to reverse the orientation of this  two-ended factor when necessary. In the above definition of $\delta'(e)$, we also needed to ensure that there are decoration preserving relative quasi-isometries at both the  initial and terminal vertices of $e$. This is not needed in \cite{cashenmartin2017quasi} because cylinders are assumed to be two-ended, thus every edge is incident to exactly one non-elementary  vertex space.

Starting with the na{\"\i}ve decoration, we iteratively perform neighbour and vertex refinement  till this process stabilizes, i.e. the partitions induced by successive refinements do not get strictly finer. This will always happen since the partition of vertices and edges induced by any geometric decoration will necessarily be coarser than the partition  into  $G$-orbits. 

We suppose $\delta:T\rightarrow \mathcal{O}$ is a decoration that is stable under vertex refinement.
Suppose there are edges $e,f\in ET$ such that $\delta(e)=\delta(f)$, and that $v=\iota e$ and $w=\iota f$ are both cylinders. Let $Q=[[X_v,\mathcal{P}_v]]$. Since the decoration is stable under vertex refinement, there is some relative quasi-isometry $\psi\in QI(Z_Q,\mathcal{P}_Q)$ such that
 $$\Big(\delta_w\circ (\mu_w)_*^{-1},(\mu_w)_*(f)\Big)=\Big(\delta_v\circ (\mu_v)_*^{-1}\circ\psi_*^{-1},(\psi_*\circ(\mu_v)_*)(e)\Big).$$ Let $\phi_v:=\mu_w^{-1}\circ\psi\circ\mu_v:(X_v,\mathcal{P}_v)\rightarrow (X_w,\mathcal{P}_w)$, where $\mu_w^{-1}$ is a coarse inverse to $\mu_w$. Then $f=(\phi_v)_*(e)$ and $\delta_w\circ (\phi_v)_*=\delta_v$. This observation illustrates part of the proof of the following proposition.

\begin{prop}[See Propositions 3.5, 5.9 and 5.12 of \cite{cashenmartin2017quasi}]\label{prop:selfqi}
Let $T$ be the JSJ tree of cylinders of a group $G$. Suppose $\delta:T\rightarrow \mathcal{O}$ is a geometric decoration that is stable under neighbour and vertex refinement. Then  the following hold:
\begin{enumerate}
\item Whenever $\delta(v)=\delta(w)$ for vertices of $T$, there is a decoration preserving tree automorphism $\chi$ such that $\chi(v)=w$. Moreover, there is a relative quasi-isometry $\phi_v:(X_v,\mathcal{P}_v)\rightarrow (X_w,\mathcal{P}_w)$ such that $(\phi_v)_*=\chi|_{lk(v)}$.
\item Suppose also that $\delta(e)=\delta(f)$ for edges $e,f$ of $T$, with $\iota e=v$ and $\iota f=w$. Then we  may choose $\chi$ as above such  that $\chi(e)=f$ and $(\phi_v)_*(e)=f$.
\end{enumerate}
\end{prop}

We wish to compare two JSJ trees of cylinders decorated with  different sets of ornaments. To do this, Cashen and Martin  define a matrix known as the \emph{structure invariant} which encodes the adjacency data of different ornaments \cite[Section 3.3]{cashenmartin2017quasi}. 

\begin{defn}
Given a decoration $\delta:T\rightarrow \mathcal{O}$ that is stable under neighbour refinement, the structure invariant $S(T,\delta,\mathcal{O})$ is an $|\mathrm{im}(\delta)| \times |\mathrm{im}(\delta)|$ matrix with entries in $\mathbb{N}\cup \{\infty\}$. Choosing an ordering of $\mathrm{im}(\delta)=\{o_1,\dots, o_{|\mathrm{im}(\delta)|}\}$, the  $(i,j)$-entry of $S(T,\delta,\mathcal{O})$ counts the number of copies of vertices or edges with ornament $o_j$ that are adjacent to a vertex or edge with ornament $o_i$. 
\end{defn}

This is well-defined because $\delta$ is stable under neighbour refinement. This is a special case of a more general definition  in \cite{cashenmartin2017quasi}, defined for a decoration that is not necessarily stable under neighbour refinement.

\begin{prop}[{\cite[Theorem 5.13]{cashenmartin2017quasi}}]\label{prop:equivbijec}
Suppose we have two JSJ trees of cylinders $T$ and $T'$ that are endowed with geometric decorations $\delta_0:T\rightarrow \mathcal{O}_0$ and $\delta'_0:T' \rightarrow \mathcal{O}_0$. We perform neighbour and vertex refinement iteratively until they stabilize, obtaining decorations $\delta:T\rightarrow \mathcal{O}$ and $\delta':T' \rightarrow \mathcal{O}'$. Then $(T,\delta_0)$ and $(T',\delta'_0)$  are equivalent if and only if:
\begin{enumerate}
\item there is a bijection $\beta:\mathrm{im}(\delta)\rightarrow \mathrm{im}(\delta')$ such that $$\delta_0\circ \delta^{-1}=\delta'_0\circ (\delta')^{-1}\circ \beta;$$
\item after ordering $\mathrm{im}(\delta)$ and $\mathrm{im}(\delta')$ so that $\beta$ is order preserving, the structure invariants $S(T,\delta,\mathcal{O})$ and $S(T',\delta',\mathcal{O}')$ are equal;
\item for every ornament $o\in \mathcal{O}$ with a vertex $v\in \delta^{-1}(o)$, there exists a vertex $v'\in (\delta')^{-1}(\beta (o))$ and a relative quasi-isometry $\phi_v:(X_v,\mathcal{P}_v)\rightarrow (X_{v'},\mathcal{P}_{v'})$ such that the decorations $\delta'_v\circ (\phi_v)_*$ and $\beta\circ \delta_v$ on the peripheral structure $\mathcal{P}_{v}$ are equal.
\end{enumerate}
\end{prop}

Using the notation of  Proposition \ref{prop:equivbijec}, we identify $\mathrm{im}(\delta)$ with $\mathrm{im}(\delta')$ via the bijection $\beta$. By combining Propositions \ref{prop:selfqi} and \ref{prop:equivbijec}, we thus deduce the following:

\begin{cor}\label{cor:dectreeisom}
Let $(T,\delta_0)$ and $(T',\delta'_0)$ be decorated JSJ trees of cylinders as above. Then there exist decorations $\delta:T\rightarrow \mathcal{O}$ and $\delta':T'\rightarrow \mathcal{O}$ such that $\mathrm{im}(\delta)=\mathrm{im}(\delta')$ and the following hold:
\begin{enumerate}
\item If $\delta(v)=\delta'(w)$ for vertices $v\in VT$ and $w\in VT'$, then there exists a decoration preserving tree isomorphism $\chi:T\rightarrow T'$ such that $\chi(v)=w$. Moreover, there is a relative quasi-isometry $\phi_v:(X_v,\mathcal{P}_v)\rightarrow (X'_w,\mathcal{P}_w)$ such that $(\phi_v)_*=\chi|_{lk(v)}$.
\item Suppose also that $\delta(e)=\delta(f)$ for edges $e\in ET$ and $f\in ET'$, with $\iota e=v$ and $\iota f=w$. Then we may choose $\chi$ as above such that $\chi(e)=f$ and  $(\phi_v)_*(e)=f$.
\end{enumerate}
\end{cor}
\subsection{RAAGs}
\label{sec:raags}
The Salvetti complex $S(\Gamma)$  of a RAAG  $A(\Gamma)$ is defined as follows: there is a single vertex $w$ and a one-edge loop $e_v$ is attached to $w$ for every $v\in V\Gamma$. For every maximal $n$-clique $\{v_1,\dots, v_n\}\subseteq V\Gamma$, we attach an $n$-torus to the edges $e_{v_1},\dots, e_{v_n}$.  The fundamental group of $S(\Gamma)$ is $A(\Gamma)$. In fact, $S(\Gamma)$ is a classifying space for $A(\Gamma)$.

\begin{conv}\label{conv:inclusionstandard} 
If $\Gamma'\subseteq \Gamma$ is an induced subgraph, we  identify  $A(\Gamma')$ with the subgroup of  $A(\Gamma)$ generated by vertices of $\Gamma'$, and identify $S(\Gamma')$ with the corresponding subcomplex of $S(\Gamma)$.
\end{conv}

Let $X(\Gamma)$ be the universal cover of $S(\Gamma)$ with covering map $p:X(\Gamma)\rightarrow S(\Gamma)$. Note that $X(\Gamma)$ is  a $CAT(0)$ cube complex.  
A \emph{standard $k$-flat} in $X(\Gamma)$ is a  $k$-flat that covers a standard $k$-torus in $S(\Gamma)$, i.e. a $k$-torus corresponding to a $k$-clique of $\Gamma$.  A \emph{standard geodesic} in $X(\Gamma)$ is a standard $1$-flat. A \emph{standard subcomplex} of $X(\Gamma)$ is a connected component of $p^{-1}(S({\Gamma'}))$ for some induced subgraph $\Gamma'\subseteq \Gamma$.

We may identify the 1-skeleton of $X(\Gamma)$ with the Cayley graph of $A(\Gamma)$ with respect to the generating set $V\Gamma$. Then edges of $X(\Gamma)$ are labelled by vertices of $\Gamma$. Let $V_e\in V\Gamma$ be the label associated to some edge $e\in X(\Gamma)$.  We define the \emph{support} $V_K$ of a subcomplex $K\subseteq X_\Gamma$  to be $\{V_e\mid e \textrm{ is an edge in } K\}$. If $K$ is a standard subcomplex associated to an induced subgraph $\Gamma'\subseteq \Gamma$, then $V_K=V\Gamma'$. Given $v\in V\Gamma$ and a standard geodesic $l\subseteq X(\Gamma)$, we say that $l$ is a $v$-geodesic if $V_l=\{v\}$. Given a subset $B\subseteq V\Gamma$, we say that a standard geodesic $l\subseteq X(\Gamma)$ is a $B$-geodesic if it is a $v$-geodesic for some $v\in B$.

It is shown in the proof of Proposition 3.1 of \cite{gandini16homstabraags} that every finite simplicial graph $\Gamma$ has maximal join decomposition $\Gamma_1\circ \dots \circ\Gamma_k$, where $\Gamma_1$ is a maximal clique factor and no $\Gamma_i$ splits as a join for $k\geq 2$. This gives direct product decompositions $X(\Gamma)=\mathbb{E}^n\times X(\Gamma_2)\times\dots \times X(\Gamma_k)$ and $A(\Gamma)=\mathbb{Z}^n\times A(\Gamma_2)\dots \times A(\Gamma_n)$. We call this the \emph{de Rahm decomposition} of a RAAG. As noted by Huang \cite{huang2014quasi}, work of  \cite{behrstockcharney12divergence},\cite{abrams13fillings} and \cite{kapovich1998derahm} imply the de Rahm decomposition is stable under quasi-isometry in the following sense:

\begin{thm}[{\cite[Theorem 2.9]{huang2014quasi}}]\label{thm:derahm} Let $\Gamma$ and $\Lambda$ be finite simplicial graphs, and let $X(\Gamma)=\mathbb{E}^n\times X(\Gamma_2)\times \dots \times X(\Gamma_k)$ and $X(\Lambda)=\mathbb{E}^m\times X(\Lambda_2)\times\dots \times X(\Lambda_l)$ be the de Rahm decompositions. Suppose there is a quasi-isometry $f:X(\Gamma)\rightarrow X(\Lambda)$. 
Then $n=m$, $k=l$, and after reordering $\Lambda_2,\dots, \Lambda_l$, there are quasi-isometries $f_i:X(\Gamma_i)\rightarrow X(\Lambda_i)$ such that the following diagram, 
whose vertical maps are projections, commutes up to some uniform error $D$.
\begin{figure}[H]
\begin{tikzcd}  
X(\Gamma)\arrow{r}{f}\arrow[d]& X(\Lambda)\arrow[d]\\
X(\Gamma_i)\arrow{r}{f_i}& X(\Lambda_i)
\end{tikzcd}
\end{figure}
Moreover, if $n=m=0$, then $f$ is uniformly close to the product of quasi-isometries $f_2\times \dots \times f_k$.

\end{thm}
We now explain how to naturally endow  a RAAG with a peripheral structure. All peripheral structures of RAAGs considered in this paper will arise in this way. 

\begin{defn}\label{defn:pstrucraag}
Let $\Gamma$ be a finite simplicial graph. For each standard geodesic $l$ in $X(\Gamma)$, we  define  $$\mathcal{S}_l=\{l'\subseteq X(\Gamma)\mid \textrm{$l'$ is a standard geodesic parallel to $l$}\}.$$ Let $B\subseteq V\Gamma$ be an arbitrary set of vertices. We define a peripheral structure $\mathcal{P}_{B}$ of $X(\Gamma)$ to be  $$\mathcal{P}_{B}:=\{\mathcal{S}_l\mid \textrm{ $l$ is a $B$-geodesic}\}.$$
In other words, $\mathcal{P}_{B}$ consists of the set of parallel classes of standard geodesics labelled by elements of $B$. We also let $\mathcal{P}_{B}$ denote the corresponding peripheral structure on $A(\Gamma)$ obtained by restricting to vertex sets of standard geodesics as above.

\end{defn}

\begin{exmp}
Let $\Gamma$ be the graph consisting of two vertices $a$ and $b$ joined by a single edge. Then $A(\Gamma)=\langle a,b\mid [a,b]\rangle\cong \mathbb{Z}^2$. Let $\mathcal{P}_1$, $\mathcal{P}_2$ and $\mathcal{P}_3$ be the peripheral structures  defined in Example \ref{exmp:peripheralz2qi}. Then $\mathcal{P}_1=\mathcal{P}_{\{a\}}$, $\mathcal{P}_2=\mathcal{P}_{\{b\}}$ and $\mathcal{P}_3=\mathcal{P}_{\{a,b\}}$
\end{exmp}

By Theorem \ref{thm:derahm}, the quasi-isometry classification of RAAGs that split as a direct product reduces to the quasi-isometry classification of these direct product factors. The following proposition shows that the same is true for relative quasi-isometries.

\begin{prop}\label{prop:prodrelqi}
Suppose $\Gamma$ and $\Lambda$ are finite simplicial graphs and that $(f,f_*):(X(\Gamma),\mathcal{P}_{B_\Gamma})\rightarrow (X(\Lambda),\mathcal{P}_{B_\Lambda})$ is a relative quasi-isometry. Then there exists a relative quasi-isometry $(g,g_*):(X(\Gamma),\mathcal{P}_{B_\Gamma})\rightarrow (X(\Lambda),\mathcal{P}_{B_\Lambda})$ such that $f_*=g_*$ and $g$  splits as a product of relative quasi-isometries of factors in the de Rahm decomposition.
\end{prop}
\begin{proof}
Let $X(\Gamma)=X(\Gamma_1)\times \dots \times X(\Gamma_k)$ and $X(\Lambda)=X(\Lambda_1)\times \dots \times X(\Lambda_l)$ be de Rahm decompositions. By Theorem \ref{thm:derahm}, we see that $l=k$ and (after reordering) we have quasi-isometries $f_i:X(\Gamma_i)\rightarrow X(\Lambda_i)$, as defined in Theorem \ref{thm:derahm}.
We will show that $f_1\times \dots \times f_k$ is the required relative quasi-isometry. 

We first observe that any standard geodesic $l$ in $X(\Gamma)$ is parallel to a standard geodesic in some $X(\Gamma_i)$. Indeed, we see that the vertex set of $l$ is $g_1\dots g_n\langle v\rangle$, where $v\in \Gamma_i$ for some $i$ and each $g_j\in A(\Gamma_j)$. As $g_1\dots g_n\langle v\rangle=g_i\langle v\rangle g_1\dots g_{i-1}g_{i+1}\dots g_n$, we see that $l$ is parallel to the standard geodesic with vertex set $g_i\langle v\rangle\subseteq X(\Gamma_i)$. Thus there is a natural correspondence between elements of $\mathcal{P}_{B_\Gamma}$ and of $\mathcal{P}_{\Gamma_1\cap B_\Gamma}\sqcup\dots \sqcup \mathcal{P}_{\Gamma_k\cap B_\Gamma }$. A similar argument holds for standard geodesics in $X(\Lambda)$.

Let $l\subseteq X(\Gamma)$ be a $B_\Gamma$-geodesic, and let $q$ be a $B_\Lambda$-geodesic in $X(\Lambda)$ such that $[q]=[f(l)]$. Say $l'\subseteq X(\Gamma_i)$ and $q'\subseteq X(\Lambda_j)$ are standard geodesics parallel to $l$ and $q$ respectively. Let $p_i:X(\Gamma)\rightarrow X(\Gamma_i)$ and $p'_i:X(\Lambda)\rightarrow X(\Lambda_i)$ be  CAT(0) projections. By the definition of $f_i$, we observe that  $[f_i(l')]=[(f_i\circ p_i)(l)]=[(p'_i\circ f)(l)]=[p'_i(q)]$. If $i\neq j$, then $p'_i(q)$ is a point and hence $f_i$ cannot be a quasi-isometry. Thus $i=j$ and so $[f_i(l')]=[q']$. This demonstrates that $f_1\times \dots \times f_k:(X(\Gamma),\mathcal{P}_{B_\Gamma})\rightarrow (X(\Lambda),\mathcal{P}_{B_\Lambda})$ is indeed a relative quasi-isometry  and $f_*=(f_1\times \dots \times f_k)_*$.
\end{proof}

In \cite{huang2014quasi}, Huang  showed that two  RAAGs with finite outer automorphism group are quasi-isometric if and only if their defining graphs are isomorphic. These methods were generalised in \cite{huang2015quasi}, where Huang defined the following class of RAAGs.

\begin{defn}
A RAAG $A(\Gamma)$ is of \emph{type II} if:
\begin{enumerate}
\item $\Gamma$ is connected;
\item  there do not exist distinct vertices $v,w\in V\Gamma$ such that $\mathrm{lk}(v)\cap \mathrm{lk}(w)$ separates $\Gamma$.
\end{enumerate}
\end{defn}
In \cite{huang2015quasi}, Huang shows that two RAAGs of type II are quasi-isometric if and only if they are commensurable. 
  We work with the extension complex of a RAAG, initially defined by Kim-Koberda \cite{kim2013embedability}. The following description is due to Huang \cite{huang2014quasi}. 

\begin{defn}
Given a finite simplicial graph $\Gamma$, the \emph{extension complex} $\mathcal{R}(\Gamma)$ is a simplicial complex defined as follows:
\begin{enumerate}
\item the vertex set consists of parallel classes of standard geodesics in $X(\Gamma)$;
\item vertices $v_1$ and $v_2$ are joined by an edge if and only if there are standard geodesics $l_i$ in the parallel classes $v_i$ for $i=1,2$, such that $l_1$ and $l_2$ span a standard $2$-flat;
\item $\mathcal{R}(\Gamma)$ is the flag complex defined by its 1-skeleton.
\end{enumerate}
If $F$ is a standard $k$-flat, let $\Delta(F)$ denote the corresponding $(k-1)$-simplex of $\mathcal{R}(\Gamma)$. More generally, if $K\subseteq X(\Gamma)$ is any standard subcomplex and $\{F_i\}_{i\in I}$ is the collection of standard flats contained in $K$, then  $\Delta(K)$ is the union $\cup_{i\in I}\Delta(F_i)$.
\end{defn}

In Theorem 2.25 and Lemma 3.21 of \cite{huang2015quasi}, it is shown that a quasi-isometry $q$ between RAAGs of type II induces a simplicial isomorphism $q_*$ between their extension complexes. One drawback of the way this is formulated in \cite{huang2015quasi} is that a `nice' quasi-isometry may not induce a `nice' simplicial isomorphism. For example, if $q:X(\Gamma)\rightarrow X(\Gamma)$ is the identity map, the induced map $q_*$ is not necessarily the identity map. This can remedied by the following proposition.

\begin{prop}\label{prop:inducedmapextcmplx}
Let $A(\Gamma)$ and $A(\Lambda)$ be RAAGs of type II and suppose $q:X(\Gamma)\rightarrow X(\Lambda)$ is a quasi-isometry. Then there exists a simplicial isomorphism $q_*:\mathcal{R}(\Gamma)\rightarrow \mathcal{R}(\Lambda)$ such that for all standard geodesics $l\subseteq X(\Gamma)$ and $l'\subseteq X(\Lambda)$ such that $[q(l)]=[l']$, we have $q_*(\Delta(l))=\Delta(l')$.
\end{prop}
\begin{proof}
We refer to the proof of Theorem 2.20 of \cite{huang2015quasi}, which we will not reproduce here. The key point is that whilst proving Theorem 2.20 of \cite{huang2015quasi}, some arbitrary choices were made when defining  $q_*$ over certain vertices. Suppose that such a choice is to be made for some vertex $v=\Delta(l)$, and  there exists a standard geodesic $l'\subseteq X(\Lambda)$ such that $[q(l)]=[l']$. Then $q_*(v)$ can be chosen to be equal to $\Delta(l')$. Lemma 3.21 of \cite{huang2015quasi} shows that this map is indeed a simplicial isomorphism.
\end{proof}

\section{The JSJ tree of cylinders of a RAAG}\label{sec:jsjtoc}
The purpose of this section is to explicitly describe the JSJ tree of a cylinders of a RAAG in terms of its defining graph. Combined with Theorem \ref{thm:jsjtocqi} and the results of Section \ref{subsec:decoratedtrees}, this provides us with several quasi-isometry invariants of a RAAG that can be seen in its defining graph.

We assume that $\Gamma$ is a finite, connected, simplicial graph that contains at least three vertices.
It is shown in \cite{clay14raagsplit}  that  $A(\Gamma)$ splits over a 2-ended subgroup precisely when $\Gamma$ has a cut vertex. Moreover, \cite{clay14raagsplit} constructs a \emph{visual} JSJ decomposition of  $A(\Gamma)$, i.e. a graph of groups decomposition expressed in terms of the defining graph $\Gamma$. This is  generalised to splittings over any abelian subgroup in \cite{groves2017abelian}.
We now describe three  graph of group decompositions of $A(\Gamma)$: $\mathcal{G}(\Gamma)$, $\mathcal{G}'(\Gamma)$ and $C(\Gamma)$, the latter two of which are a JSJ decomposition  and the JSJ tree of cylinders  decomposition respectively. Figure \ref{fig:gogexamples} gives examples of these graphs of groups. 

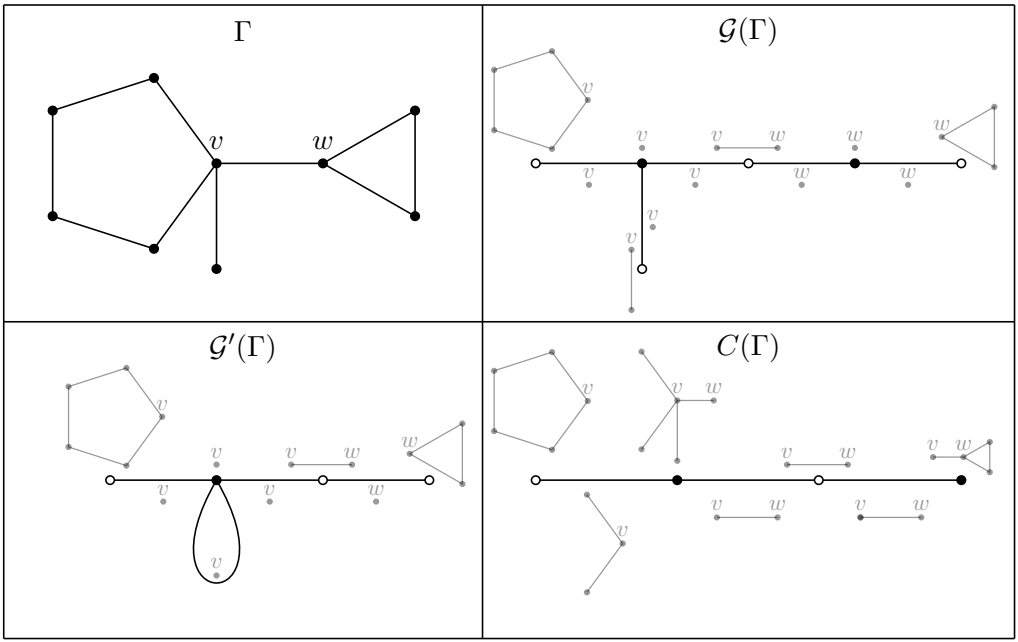
\begin{figure}[b]
    \begin{tikzpicture}[semithick,scale=1.4]

      \draw (0, 0) -- (-0.588, 0.809) -- (-1.54,  0.500) -- (-1.54, -0.500) -- (-0.588, -0.809) -- (0,0);
      \filldraw[fill=black] (0,0) circle [radius=0.04];
      \filldraw[fill=black] (-0.588, 0.809) circle [radius=0.04];
      \filldraw[fill=black] (-1.54, 0.500) circle [radius=0.04];
      \filldraw[fill=black] (-1.54, -0.500) circle [radius=0.04];
      \filldraw[fill=black] (-0.588, -0.809) circle [radius=0.04];
      \draw (0,0) -- (0,-1);
	\draw (0,0) -- (1,0);
      	\filldraw[fill=black] (1,0) circle [radius=0.04]; 
    \draw (1,0) -- (1.866,0.5) -- (1.866,-0.5) -- (1,0);
	\filldraw[fill=black] (1.866,0.5) circle [radius=0.04];
	\filldraw[fill=black] (1.866,-0.5) circle [radius=0.04];
      \filldraw[fill=black] (0,-1) circle [radius=0.04];
      \draw (-0,0.2) node {$v$};
      \draw (1,0.2) node {$w$};
      
      \draw (3,0) -- (4,0) -- (5,0) -- (6,0) -- (7,0);
      \draw (4,0) -- (4,-1);
      \filldraw[fill=white] (3,0) circle [radius=0.04];
      \filldraw[fill=black] (4,0) circle [radius=0.04];
      \filldraw[fill=white] (5,0) circle [radius=0.04];
      \filldraw[fill=black] (6,0) circle [radius=0.04];
      \filldraw[fill=white] (7,0) circle [radius=0.04];
      \filldraw[fill=white] (4,-1) circle [radius=0.04];

      \draw (3.1,0.6) node[scale=0.8] {\begin{tikzpicture}[semithick,scale=1,fill opacity=0.4,draw opacity=0.4] 
      \draw (0, 0) -- (-0.588, 0.809) -- (-1.54,  0.500) -- (-1.54, -0.500) -- (-0.588, -0.809) -- (0,0);
      \filldraw[fill=black] (0,0) circle [radius=0.04];
      \filldraw[fill=black] (-0.588, 0.809) circle [radius=0.04];
      \filldraw[fill=black] (-1.54, 0.500) circle [radius=0.04];
      \filldraw[fill=black] (-1.54, -0.500) circle [radius=0.04];
      \filldraw[fill=black] (-0.588, -0.809) circle [radius=0.04]; 
      \draw (-0,0.2) node[scale=1] {$v$};      
      \end{tikzpicture}};
      
      \draw (4,0.25) node[scale=0.8] {\begin{tikzpicture}[semithick,scale=1,fill opacity=0.4,draw opacity=0.4] 
      \filldraw[fill=black] (0,0) circle [radius=0.04]; 
      \draw (-0,0.2) node[scale=1] {$v$};      
      \end{tikzpicture}};
      \draw (3.5,-0.1) node[scale=0.8] {\begin{tikzpicture}[semithick,scale=1,fill opacity=0.4,draw opacity=0.4] 
      \filldraw[fill=black] (0,0) circle [radius=0.04]; 
      \draw (-0,0.2) node[scale=1] {$v$};      
      \end{tikzpicture}};
      \draw (4.5,-0.1) node[scale=0.8] {\begin{tikzpicture}[semithick,scale=1,fill opacity=0.4,draw opacity=0.4] 
      \filldraw[fill=black] (0,0) circle [radius=0.04]; 
      \draw (-0,0.2) node[scale=1] {$v$};      
      \end{tikzpicture}};
      \draw (4.1,-0.5) node[scale=0.8] {\begin{tikzpicture}[semithick,scale=1,fill opacity=0.4,draw opacity=0.4] 
      \filldraw[fill=black] (0,0) circle [radius=0.04]; 
      \draw (-0,0.2) node[scale=1] {$v$};      
      \end{tikzpicture}};
      \draw (3.9,-1) node[scale=0.8] {\begin{tikzpicture}[semithick,scale=1,fill opacity=0.4,draw opacity=0.4] 
      \filldraw[fill=black] (0,0) circle [radius=0.04]; 
      \draw (-0,0.2) node[scale=1] {$v$};
      \draw (0,0) -- (0,-1);
      \filldraw[fill=black] (0,-1) circle [radius=0.04];
      \end{tikzpicture}};
      \draw (5,0.25) node[scale=0.8] {\begin{tikzpicture}[semithick,scale=1,fill opacity=0.4,draw opacity=0.4] 
      \filldraw[fill=black] (0,0) circle [radius=0.04];
      \filldraw[fill=black] (1,0) circle [radius=0.04]; 
      \draw (-0,0.2) node[scale=1] {$v$};
      \draw (0,0) -- (1,0);
      \draw (1,0.2) node {$w$};      
      \end{tikzpicture}};
      \draw (5.5,-0.1) node[scale=0.8] {\begin{tikzpicture}[semithick,scale=1,fill opacity=0.4,draw opacity=0.4] 
      \filldraw[fill=black] (0,0) circle [radius=0.04]; 
      \draw (-0,0.2) node[scale=1] {$w$};    
      \end{tikzpicture}};
      \draw (6,0.25) node[scale=0.8] {\begin{tikzpicture}[semithick,scale=1,fill opacity=0.4,draw opacity=0.4] 
      \filldraw[fill=black] (0,0) circle [radius=0.04]; 
      \draw (-0,0.2) node[scale=1] {$w$};      
      \end{tikzpicture}};
      \draw (6.5,-0.1) node[scale=0.8] {\begin{tikzpicture}[semithick,scale=1,fill opacity=0.4,draw opacity=0.4] 
      \filldraw[fill=black] (0,0) circle [radius=0.04]; 
      \draw (-0,0.2) node[scale=1] {$w$};      
      \end{tikzpicture}};
      \draw (7,0.25) node[scale=0.8] {\begin{tikzpicture}[semithick,scale=1,fill opacity=0.4,draw opacity=0.4] 
      \filldraw[fill=black] (0,0) circle [radius=0.04];
      \draw (0,0) -- (0.866,0.5) -- (0.866,-0.5) -- (0,0);
	  \filldraw[fill=black] (0.866,0.5) circle [radius=0.04];
	  \filldraw[fill=black] (0.866,-0.5) circle [radius=0.04];
      \draw (-0,0.2) node[scale=1] {$w$};    
      \end{tikzpicture}};

      \draw (2.5,-4.5) -- (2.5,1.5);
      \draw (-2,1.5) -- (7.5,1.5);
      \draw (-2,1.5) -- (-2,-4.5);
      \draw (-2,-1.5) -- (7.5,-1.5);
      \draw (7.5,-4.5) -- (7.5,1.5);
      \draw (-2,-4.5) -- (7.5,-4.5);     	
      \draw (0.25,1.25) node {$\Gamma$};
      \draw (0.25,-1.75) node {$\mathcal{G}'(\Gamma)$};
      \draw (5,1.25) node {$\mathcal{G}(\Gamma)$};
      \draw (5,-1.75) node {$C(\Gamma)$};

      \draw (-1,-3)--(0,-3)--(1,-3)--(2,-3);
      \filldraw[fill=black] (0, -3) circle [radius=0.04];
      \filldraw[fill=white] (-1, -3) circle [radius=0.04];
      \filldraw[fill=white] (1, -3) circle [radius=0.04]; 
      \filldraw[fill=white] (2, -3) circle [radius=0.04];
      \draw (0,-3) to[out=-120,in=-60,distance=1.5cm] (0,-3);  
      
      \draw (-0.9,-2.4) node[scale=0.8] {\begin{tikzpicture}[semithick,scale=1,fill opacity=0.4,draw opacity=0.4] 
      \draw (0, 0) -- (-0.588, 0.809) -- (-1.54,  0.500) -- (-1.54, -0.500) -- (-0.588, -0.809) -- (0,0);
      \filldraw[fill=black] (0,0) circle [radius=0.04];
      \filldraw[fill=black] (-0.588, 0.809) circle [radius=0.04];
      \filldraw[fill=black] (-1.54, 0.500) circle [radius=0.04];
      \filldraw[fill=black] (-1.54, -0.500) circle [radius=0.04];
      \filldraw[fill=black] (-0.588, -0.809) circle [radius=0.04]; 
      \draw (-0,0.2) node[scale=1] {$v$};      
      \end{tikzpicture}};
      
      \draw (1,-2.75) node[scale=0.8] 
      {\begin{tikzpicture}[semithick,scale=1,fill opacity=0.4,draw opacity=0.4] 
      \filldraw[fill=black] (0,0) circle [radius=0.04];
      \filldraw[fill=black] (1,0) circle [radius=0.04]; 
      \draw (-0,0.2) node[scale=1] {$v$};
      \draw (0,0) -- (1,0);
      \draw (1,0.2) node {$w$};           
      \end{tikzpicture}};
      
      \draw (0,-2.75) node[scale=0.8] 
      {\begin{tikzpicture}[semithick,scale=1,fill opacity=0.4,draw opacity=0.4] 
      \filldraw[fill=black] (0,0) circle [radius=0.04]; 
      \draw (-0,0.2) node[scale=1] {$v$};      
      \end{tikzpicture}};
      
      \draw (2,-2.75) node[scale=0.8] 
      {\begin{tikzpicture}[semithick,scale=1,fill opacity=0.4,draw opacity=0.4] 
      \filldraw[fill=black] (0,0) circle [radius=0.04];
      \draw (0,0) -- (0.866,0.5) -- (0.866,-0.5) -- (0,0);
	  \filldraw[fill=black] (0.866,0.5) circle [radius=0.04];
	  \filldraw[fill=black] (0.866,-0.5) circle [radius=0.04];
      \draw (-0,0.2) node[scale=1] {$w$};       
      \end{tikzpicture}};
      
      \draw (1.5,-3.1) node[scale=0.8] 
      {\begin{tikzpicture}[semithick,scale=1,fill opacity=0.4,draw opacity=0.4] 
      \filldraw[fill=black] (0,0) circle [radius=0.04]; 
      \draw (-0,0.2) node[scale=1] {$w$};      
      \end{tikzpicture}};
      
      \draw (0.5,-3.1) node[scale=0.8] 
      {\begin{tikzpicture}[semithick,scale=1,fill opacity=0.4,draw opacity=0.4] 
      \filldraw[fill=black] (0,0) circle [radius=0.04]; 
      \draw (-0,0.2) node[scale=1] {$v$};      
      \end{tikzpicture}};
      
      \draw (-0.5,-3.1) node[scale=0.8] 
      {\begin{tikzpicture}[semithick,scale=1,fill opacity=0.4,draw opacity=0.4] 
      \filldraw[fill=black] (0,0) circle [radius=0.04]; 
      \draw (-0,0.2) node[scale=1] {$v$};      
      \end{tikzpicture}};
      
      \draw (0,-3.8) node[scale=0.8] 
      {\begin{tikzpicture}[semithick,scale=1,fill opacity=0.4,draw opacity=0.4] 
      \filldraw[fill=black] (0,0) circle [radius=0.04]; 
      \draw (-0,0.2) node[scale=1] {$v$};      
      \end{tikzpicture}};

      \draw (3,-3) -- (4.33,-3) -- (5.66,-3) -- (7,-3);
      \filldraw[fill=white] (3,-3) circle [radius=0.04];
      \filldraw[fill=black] (4.33,-3) circle [radius=0.04];
      \filldraw[fill=white] (5.66,-3) circle [radius=0.04];
      \filldraw[fill=black] (7,-3) circle [radius=0.04];

      \draw (3.1,-2.25) node[scale=0.8] {
      \begin{tikzpicture}[semithick,scale=1,fill opacity=0.4,draw opacity=0.4] 
      \draw (0, 0) -- (-0.588, 0.809) -- (-1.54,  0.500) -- (-1.54, -0.500) -- (-0.588, -0.809) -- (0,0);
      \filldraw[fill=black] (0,0) circle [radius=0.04];
      \filldraw[fill=black] (-0.588, 0.809) circle [radius=0.04];
      \filldraw[fill=black] (-1.54, 0.500) circle [radius=0.04];
      \filldraw[fill=black] (-1.54, -0.500) circle [radius=0.04];
      \filldraw[fill=black] (-0.588, -0.809) circle [radius=0.04]; 
      \draw (-0,0.2) node[scale=1] {$v$};      
      \end{tikzpicture}};
      
      \draw (4.397,-2.3) node[scale=0.8] {\begin{tikzpicture}[semithick,scale=1,fill opacity=0.4,draw opacity=0.4] 
       \draw (0, 0) -- (-0.588, 0.809);
       \draw (-0.588, -0.809) -- (0,0);
      \filldraw[fill=black] (0,0) circle [radius=0.04];
      \filldraw[fill=black] (-0.588, 0.809) circle [radius=0.04];
      \filldraw[fill=black] (-0.588, -0.809) circle [radius=0.04];

	\draw (0,0) -- (0.6,0);
	\draw (0,0) -- (0,-1);
	\filldraw[fill=black] (0.6,0) circle [radius=0.04];
      \filldraw[fill=black] (0,-1) circle [radius=0.04];
      \draw (-0,0.2) node {$v$};
      \draw (0.6,0.2) node {$w$};
      \end{tikzpicture}};
      
      \draw (5.66,-2.75) node[scale=0.8] {
      \begin{tikzpicture}[semithick,scale=1,fill opacity=0.4,draw opacity=0.4] 
	  \draw (0,0) -- (1,0);
	  \filldraw[fill=black] (1,0) circle [radius=0.04];
      \filldraw[fill=black] (0,0) circle [radius=0.04];
      \draw (-0,0.2) node {$v$};
      \draw (1,0.2) node {$w$};
      \end{tikzpicture}};
      
      \draw (6.95,-2.75) node[scale=0.8] {
      \begin{tikzpicture}[semithick,scale=1,fill opacity=0.4,draw opacity=0.4] 
	  \draw (0,0) -- (0.5,0);
	  \filldraw[fill=black] (0.5,0) circle [radius=0.04];
      \filldraw[fill=black] (0,0) circle [radius=0.04];
      \draw (-0,0.2) node {$v$};
      \draw (0.5,0.2) node {$w$};
      \draw (0.5,0) -- (0.5*1.866,0.25) -- (0.5*1.866,-0.25) -- (0.5,0);
	  \filldraw[fill=black] (0.5*1.866,0.25) circle [radius=0.04];
	  \filldraw[fill=black] (0.5*1.866,-0.25) circle [radius=0.04];
      \end{tikzpicture}};
      
      \draw (3.7,-3.6) node[scale=0.8] {\begin{tikzpicture}[semithick,scale=1,fill opacity=0.4,draw opacity=0.4] 
       \draw (0, 0) -- (-0.588, 0.809);
       \draw (-0.588, -0.809) -- (0,0);
      \filldraw[fill=black] (0,0) circle [radius=0.04];
      \filldraw[fill=black] (-0.588, 0.809) circle [radius=0.04];
      \filldraw[fill=black] (-0.588, -0.809) circle [radius=0.04];
      \draw (-0,0.2) node {$v$};
      \end{tikzpicture}};
      
      \draw (5.0,-3.25) node[scale=0.8] {\begin{tikzpicture}[semithick,scale=1,fill opacity=0.4,draw opacity=0.4] 
      \filldraw[fill=black] (0,0) circle [radius=0.04];
	\draw (0,0) -- (1,0);
	\filldraw[fill=black] (1,0) circle [radius=0.04];
      \draw (-0,0.2) node {$v$};
      \draw (1,0.2) node {$w$};
      \end{tikzpicture}};
      
      \draw (6.35,-3.25) node[scale=0.8] {
      \begin{tikzpicture}[semithick,scale=1,fill opacity=0.4,draw opacity=0.4] 
      \filldraw[fill=black] (0,0) circle [radius=0.04];
	  \draw (0,0) -- (1,0);
	  \filldraw[fill=black] (1,0) circle [radius=0.04];
      \draw (-0,0.2) node {$v$};
      \draw (1,0.2) node {$w$};
      \filldraw[fill=black] (0,0) circle [radius=0.04]; 
      \end{tikzpicture}};

   \end{tikzpicture}
    \caption{A graph $\Gamma$ and the corresponding visual graph of groups decompositions $\mathcal{G}(\Gamma)$, $\mathcal{G}'(\Gamma)$ and $C(\Gamma)$. We label each vertex and edge with the induced subgraph of $\Gamma$, shown in grey, that generates the corresponding vertex or edge group. }\label{fig:gogexamples}
    \end{figure}

\begin{defn}[\cite{clay14raagsplit}]\label{def:gogjsj}
Let $C_0$  denote the set of cut vertices of $\Gamma$. We say a subgraph of $\Gamma$ is \emph{biconnected} if it is connected and has no cut vertex. (Of course, a biconnected subgraph $\Gamma'$ may contain a cut vertex of $\Gamma$  provided it is not also a cut vertex of $\Gamma'$.) Let $B_1$ denote the set of maximal biconnected subgraphs of $\Gamma$.
 We define a bipartite graph of groups $\mathcal{G}(\Gamma)$ on the vertex set $C_0\sqcup B_1$ as follows:
\begin{itemize}
\item black vertices are vertices $v\in C_0$  with vertex  group $A({\{v\}})\cong \mathbb{Z}$;
\item white vertices are subgraphs $\Gamma'\in B_1$ with vertex group $A({\Gamma'})$;
\item $v\in C_0$ and $\Gamma'\in B_1$ are joined by an edge if and only if $v\in \Gamma'$. The edge group of this edge is $A({\{v\}})$.
\end{itemize}
The inclusions of edge groups into adjacent vertex groups are as described in Convention \ref{conv:inclusionstandard}.
\end{defn}

It is easy to verify that $\pi_1(\mathcal{G}(\Gamma))=A(\Gamma)$. 
We now modify $\mathcal{G}(\Gamma)$  to obtain a JSJ decomposition $\mathcal{G}'(\Gamma)$ of $A(\Gamma)$.  As Lemma \ref{lem:collapsecyl} shows, these modifications don't alter the associated tree of cylinders.

Firstly, $\mathcal{G}(\Gamma)$ may contain what is called in \cite{clay14raagsplit} a  (non-quadratically) \emph{hanging vertex} --- a valence one vertex whose  vertex group is $\mathbb{Z}^2$. Every such vertex group splits as an  HNN extension relative to its incident edge group. We refine $\mathcal{G}(\Gamma)$ at every hanging vertex by attaching a one-edge loop. This resulting graph of groups  may not be reduced in the sense of \cite{bestvinafeighn91accessibility};  we therefore collapse an edge adjacent to each valence two black vertex. 
After doing this, the resulting graph of groups  thus obtained is denoted  $\mathcal{G}'(\Gamma)$.

\begin{thm}[\cite{clay14raagsplit}]
If $\Gamma$ is a finite, connected, simplicial graph with at least three vertices, the graph of groups $\mathcal{G}'(\Gamma)$ is a JSJ decomposition of $A(\Gamma)$.
\end{thm}

\begin{lem}\label{lem:collapsecyl}
Let $T(\Gamma)$ and $T'(\Gamma)$ denote the Bass-Serre trees associated to $\mathcal{G}(\Gamma)$ and $\mathcal{G}'(\Gamma)$. Then
$T'(\Gamma)_\mathrm{cyl}=T(\Gamma)_\mathrm{cyl}$. In particular, $T(\Gamma)_\mathrm{cyl}$ is the JSJ tree of cylinders of $A(\Gamma)$.
\end{lem}
\begin{proof}
We obtain $T'(\Gamma)$ from $T(\Gamma)$ by applying the following two procedures. Firstly, we blow up a vertex corresponding to a hanging $\mathbb{Z}^2$ subgroup. Secondly, we collapse edges adjacent to valence two black vertices. Each of these procedures take place within a cylinder, and neither alters vertices at which distinct cylinders intersect. Thus $T'(\Gamma)_\mathrm{cyl}=T(\Gamma)_\mathrm{cyl}$. \end{proof}

The following lemma allows us to to describe cylinder stabilizers in the JSJ tree of cylinders.

\begin{lem}\label{lem:cylinder}
Let $v\in V\Gamma$ be a cut vertex and let $e$ be an edge of $T(\Gamma)$ with stabilizer $A(\{v\})\cong \mathbb{Z}$ for some $v\in V\Gamma$. Then an edge $f$ of $T(\Gamma)$ lies in the same cylinder as $e$  if and only if $f=ge$ for some $g\in A(\mathrm{star}(v))$.
\end{lem}
\begin{proof}
By the definition of $\mathcal{G}(\Gamma)$ and Proposition \ref{prop:qitreebs}, edge spaces  of $T(\Gamma)$ may be identified with cosets of cyclic subgroups $\langle w \rangle_{w\in V\Gamma}$.  Identifying the 1-skeleton of $X(\Gamma)$ with the Cayley graph of $G$, we identify edge spaces with standard geodesics. Two edges lie in the same cylinder if and only if the associated standard geodesics are coarsely equivalent.

Let $K$ be the standard geodesic corresponding to the cyclic subgroup $\langle v \rangle $.
It can easily be deduced  from Lemma 2.10 of \cite{huang2017quasiflat}  that two standard geodesics are coarsely equivalent precisely when they are parallel.
 Moreover,  Corollary 3.2 of \cite{huang2014quasi} tells us that two standard geodesics are parallel if and only if they have the same support. We therefore deduce that any standard geodesic parallel to $K$ is a translate of $K$.
  Lemma 3.4 of \cite{huang2014quasi} then says the parallel set of $K$ canonically splits as $K\times X({\mathrm{lk}(v)})$. We thus deduce that $gK$ is parallel to $K$ if and only if $g\in A({\mathrm{star}(v)})$.
\end{proof}

We now describe the visual construction of the JSJ  tree of cylinders of $A(\Gamma)$, mirroring the definition of $\mathcal{G}(\Gamma)$. 

\begin{defn}[cf. Definition \ref{def:gogjsj}] \label{def:gogtoc}
Let $C_0$  denote the set of cut vertices of $\Gamma$.  Let $C_1\subseteq B_1$ denote the set of maximal biconnected subgraphs of $\Gamma$  that either contain at least two elements of $C_0$, or are not contained in $\mathrm{star}(v)$ for any $v\in C_0$. We define a bipartite tree of groups $C(\Gamma)$ as follows. 
\begin{enumerate}
\item Black vertices are vertices $v\in C_0$ with vertex group $A({\mathrm{star}^\Gamma(v)})$.
\item White vertices are subgraphs $\Gamma'\in C_1$ with vertex group $A({\Gamma'})$.
\item A black vertex $v\in C_0$ and a white vertex $\Gamma'\in C_1$ are joined by an edge if and only if $v\in \Gamma'$. The corresponding edge group is $A({\mathrm{star}^{\Gamma'}(v)})=A({\Gamma'})\cap A({\mathrm{star}^\Gamma(v)}) =A(\{v\})\times A(\mathrm{lk}^{\Gamma'}(v))$. 
\end{enumerate}
The inclusions of edge groups into adjacent vertex groups are as described in Convention \ref{conv:inclusionstandard}. 
\end{defn}

It is easy to verify that $\pi_1(\mathcal{G}(\Gamma))=A(\Gamma)$. 
The definition of  $C_1$ is precisely the condition needed to ensure a vertex of the JSJ tree is contained in at least two cylinders, and thus corresponds to a $V_1$-vertex in its tree of cylinders.

\begin{prop}\label{prop:jsjtoc}
Let $\Gamma$ be a finite, connected simplicial graph.
The Bass-Serre tree of $C(\Gamma)$ is equivariantly isomorphic to the JSJ tree of cylinders of  $A(\Gamma)$. Black and white vertices  of $C(\Gamma)$ correspond to cylindrical and rigid vertices respectively.
\end{prop}

\begin{proof}
By Lemma \ref{lem:collapsecyl},  it is sufficient to show that $T(\Gamma)_\mathrm{cyl}$ is equivariantly isomorphic to the Bass-Serre tree of $C(\Gamma)$. 
We first remark that the colouring of  $C(\Gamma)$ induces a colouring of its Bass-Serre tree $\tilde C(\Gamma)$ into black and white vertices.  We define $e_v$ to be the edge of $T(\Gamma)$ corresponding to the coset $\langle v \rangle$.
We may assume that $\Gamma$ contains a cut vertex (hence at least three vertices), otherwise both $\tilde C(\Gamma)$ and the JSJ tree of cylinders are points. In particular, since $\Gamma$ is assumed to be connected, every graph $\Gamma'\in B_1$ contains at least one vertex in $C_0$, where $B_1$ is as in Definition \ref{def:gogjsj}.

We define an $A(\Gamma)$-equivariant graph isomorphism $\Omega:\tilde{C}(\Gamma)\rightarrow T(\Gamma)_\mathrm{cyl}$ as follows. Let $gA({\mathrm{star}^\Gamma(v)})$ be a black vertex of $\tilde C(\Gamma)$. We define $\Omega(gA({\mathrm{star}^\Gamma(v)})):=\mathrm{cyl}(ge_v)$, where $\mathrm{cyl}(ge_v)$ is the vertex of $T(\Gamma)_\mathrm{cyl}$ corresponding to the cylinder of $T(\Gamma)$ containing $ge_v$.
Lemma \ref{lem:cylinder} tells us  that $\Omega$ is well-defined and injective. Moreover, every edge $f$ of $T(\Gamma)$ is of the form $ge_v$ for some $g\in A(\Gamma)$ and $v\in V\Gamma$. Thus $\Omega(gA(\mathrm{star}(v)))=\mathrm{cyl}(f)$, so $\Omega$ defines a bijection between black vertices of $\tilde C(\Gamma)$ and cylinders of $T(\Gamma)$.

Now suppose $w=gA({\Gamma'})$ is a white vertex of $\tilde{C}(\Gamma)$. Note that $\Gamma'\in C_1\subseteq B_1$, so it follows from the definition of $\mathcal{G}(\Gamma)$ that $w$ corresponds to the vertex $w'=gA({\Gamma'})$ of $T(\Gamma)$. We claim that $w'$ is contained in at least two cylinders, hence is a  $V_1$-vertex of  $T(\Gamma)_\mathrm{cyl}$; we thus define $\Omega(w)=w'$. 
As remarked earlier, every $\Gamma'\in B_1$ contains at least one vertex in $C_0$. We claim that a white vertex $gA(\Gamma')$ of $T(\Gamma)$ is contained in at least two cylinders if and only if $\Gamma'\in C_1$.
If $\Gamma'\in C_1$, it  either contains at least two vertices $v,v'\in C_0$, or is not contained in $\mathrm{star}(v)$ for the unique $v\in C_0\cap \Gamma'$.  In the former case, $w'$ is contained in the two cylinders $\mathrm{cyl}(ge_v)$ and $\mathrm{cyl}(ge_{v'})$. In the latter case, there exists a $k\in A({\Gamma'})\backslash A({\mathrm{star}(v)})$, such that $\mathrm{cyl}(ge_v)$ and $\mathrm{cyl}(gke_v)$ are distinct cylinders containing $w'=gA({\Gamma'})$.

Now suppose $w'=gA(\Gamma')\in VT(\Gamma)$ is adjacent to edges $g'e_v$ and $g''e_{v'}$ that are contained in distinct cylinders. As these cylinders are distinct, it follows from Lemma \ref{lem:cylinder} that either $v\neq v'$, or $(g')^{-1}g''\in A(\Gamma')\backslash A(\mathrm{star}^{\Gamma'}(v))$. In either case, we conclude that $\Gamma\in C_1$. Thus $\Omega$ defines a equivariant bijection from white vertices of $\tilde{C}(\Gamma)$ to $V_1$-vertices of $T(\Gamma)_\mathrm{cyl}$.

Since $\tilde{C}(\Gamma)$ and $T(\Gamma)_\mathrm{cyl}$ are both bipartite, it is sufficient to show a black vertex $w=gA({\mathrm{star}(v)})$ and a white vertex $w'=g'A({\Gamma'})$ are joined by an edge in $\tilde C(\Gamma)$ if and only if $\Omega(w)$ and $\Omega(w')$ are  joined by an edge in $T(\Gamma)_\mathrm{cyl}$. 
If $w$ and $w'$ are joined by an edge, then  $v\in \Gamma'$ and there exist $k\in A({\mathrm{star}(v)})$ and $k'\in A({\Gamma'})$ such that $gk=g'k'$. Therefore, the edge $f=gke_v$ of $T(\Gamma)$ is adjacent to the vertex $g'k'A({\Gamma'})=g'A({\Gamma'})$. By Lemma \ref{lem:cylinder}, $\mathrm{cyl}(f)=\mathrm{cyl}(ge_v)$, thus $g'A({\Gamma'})$  lies in the cylinder $\mathrm{cyl}(ge_v)$ of $T(\Gamma)$, and so $\Omega(w)$ and $\Omega(w')$ are joined by an edge in $T(\Gamma)_\mathrm{cyl}$.

Conversely, if $\Omega(w)$ and $\Omega(w')$ are joined by an edge, the vertex $g'A({\Gamma'})$  is contained in $\mathrm{cyl}(ge_v)$. Thus some edge $f$ of $\mathrm{cyl}(ge_v)$ is incident to $g'A({\Gamma'})$ in $T(\Gamma)$. By Lemma \ref{lem:cylinder}, $f=gke_v$ for some $k\in A({\mathrm{star}(v)})$. Hence $v\in \Gamma'$, and there is some $k'\in A({\Gamma'})$ such that $g'k'=gk$. Thus $w=gkA({\mathrm{star}(v)})$ and $w'=gkA({\Gamma'})$, so $w$ and $w'$ are indeed joined by an edge in $\tilde C(\Gamma)$.
\end{proof}

\begin{conv}
Given a RAAG $A(\Gamma)$, we work with the Bass-Serre tree of spaces of $C(\Gamma)$. However, rather than equipping each vertex and edge space with the word metric as in the proof of Proposition \ref{prop:qitreebs}, we equip it with the CAT(0) metric. Indeed, since each vertex stabilizer $A(\Gamma_v)$ of the Bass-Serre tree is itself a RAAG, we let the corresponding vertex space be $X(\Gamma_v)$ equipped with the CAT(0) metric.  We equip edge spaces with the CAT(0) metric in the same way.
\end{conv}

\begin{defn}\label{defn:cyldecomp}
Each vertex stabiliser associated to the cylinder $gA(\mathrm{star}(v))$ splits as a  product $gA(\{v\})g^{-1}\times gA(\mathrm{lk}(v))g^{-1}\cong \mathbb{Z}\times gA(\mathrm{lk}(v))g^{-1}$. Similarly, each edge stabilizer $gA(\mathrm{star}^{\Gamma'}(v))g^{-1}$ splits as $gA(\{v\})g^{-1}\times gA(\mathrm{lk}^{\Gamma'}(v))g^{-1}\cong \mathbb{Z}\times gA(\mathrm{lk}^{\Gamma'}(v))g^{-1}$. There is an associated  decomposition of the vertex and edge spaces: cylindrical vertex spaces split as $X_v=\mathbb{E}^1\times Y_v=gX(\{v\})\times gX(\mathrm{lk}^{}(v))$ and edge spaces split as $X_e=\mathbb{E}^1\times Y_e=gX(\{v\})\times gX(\mathrm{lk}^{\Gamma'}(v))$. We call this the \emph{cylindrical decomposition} of a cylindrical vertex space or an edge space. 
\end{defn}

A consequence of working with the CAT(0) metric is that each edge map $\alpha_e:X_e\rightarrow X_{\overline e}$  splits as a product of cubical isomorphisms $(\mathbb{E}^1,Y_e)\rightarrow (\mathbb{E}^1,Y_{\overline{e}})$.

\begin{rem}\label{rem:peripheralstrucutreraags}
Using Definitions \ref{def:gogjsj} and \ref{def:gogtoc}, we  can give an explicit description of elements in the peripheral structure of a vertex space in the JSJ tree of cylinders. Let $e=gA(\mathrm{star}^{\Gamma'}(v))$ be an edge in the Bass-Serre tree of $C(\Gamma)$ adjacent to a rigid vertex $gA(\Gamma')$ and a cylinder $g(\mathrm{star}^{\Gamma}(v))=\mathrm{cyl}(ge_v)$, where $e_v$ is as in the proof of Proposition \ref{prop:jsjtoc}. The edges of $T(\Gamma)$ that are contained in $\mathrm{cyl}(ge_v)$ and adjacent to $gA(\Gamma')$ are precisely those of the form $gg'e_v$ for some $g'\in A(\mathrm{lk}^{\Gamma'}(v))$. Notice that the vertex set of $Y_e$ is $gA(\mathrm{lk}^{\Gamma'}(v))$ --- see Definition \ref{defn:cyldecomp}. Thus elements of  $\mathcal{S}_e$ correspond precisely to subsets of $X_e$ of the form $\mathbb{E}^1\times \{y\}$ for some vertex $y\in  Y_e$.

\end{rem}

\section{Quasi-isometries of  free products}\label{sec:freeproduct}
In this section we develop techniques to  determine  when cylindrical vertex spaces in a  JSJ tree of cylinders are relatively quasi-isometric 
by adapting work of  Papasoglu--Whyte \cite{papasoglu2002quasi}. These techniques are heavily used  in the proofs of both Proposition \ref{prop:relqicyl} and Theorem \ref{thm:mainthm}. This section can be skipped on first reading and referred back to when needed in the proofs of these results.

To motivate the problem, we consider the RAAGs $A(\Gamma)$ and $A(\Lambda)$, where $\Gamma$ and $\Lambda$ are as in Figure \ref{fig:raagsnotqi}. The cylindrical vertex group in $C(\Gamma)$ is isomorphic to  $\mathbb{Z}\times (F_2 * F_2 * \mathbb{Z} * \mathbb{Z}^2)$, whereas the  cylindrical vertex group  in $C(\Lambda)$ is isomorphic to  $\mathbb{Z}\times (F_2 * F_2)$. It follows from Theorems \ref{thm:jsjtocqi} and \ref{thm:derahm} that if $A(\Gamma)$ and $A(\Lambda)$ are quasi-isometric, then  $F_2 * F_2$ must be quasi-isometric to $F_2 * F_2 * \mathbb{Z} * \mathbb{Z}^2$. However, the work of Papasoglu--Whyte ensures these infinite-ended groups are not quasi-isometric \cite{papasoglu2002quasi}. We thus deduce that $A(\Gamma)$ and $A(\Lambda)$ are not quasi-isometric.

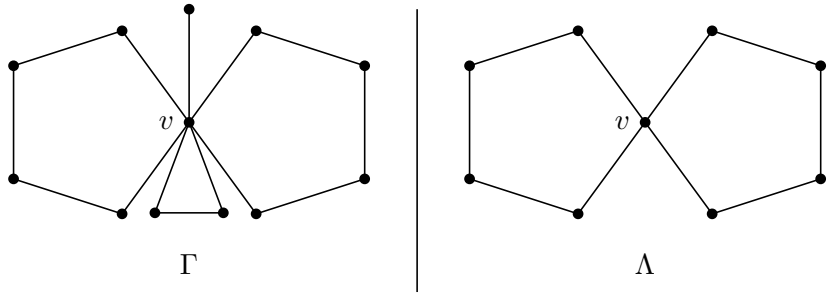
\begin{figure}
    \begin{tikzpicture}[semithick,scale=1.5]

      \draw (0, 0) -- (-0.588, 0.809) -- (-1.54,  0.500) -- (-1.54, -0.500) -- (-0.588, -0.809) -- (0,0);
      \filldraw[fill=black] (0,0) circle [radius=0.04];
      \filldraw[fill=black] (-0.588, 0.809) circle [radius=0.04];
      \filldraw[fill=black] (-1.54, 0.500) circle [radius=0.04];
      \filldraw[fill=black] (-1.54, -0.500) circle [radius=0.04];
      \filldraw[fill=black] (-0.588, -0.809) circle [radius=0.04];
      \draw (0, 0) -- (0.588, 0.809) -- (1.54,  0.500) -- (1.54, -0.500) -- (0.588, -0.809) -- (0,0);
      \filldraw[fill=black] (0.588, 0.809) circle [radius=0.04];
      \filldraw[fill=black] (1.54, 0.500) circle [radius=0.04];
      \filldraw[fill=black] (1.54, -0.500) circle [radius=0.04];
      \filldraw[fill=black] (0.588, -0.809) circle [radius=0.04];
	\draw (0,0) -- (0,1);
	\draw (0,0) -- (-.3,-0.8) -- (0.3,-0.8) -- (0,0);
	\filldraw[fill=black] (0,1) circle [radius=0.04];
      \filldraw[fill=black] (-.3,-0.8) circle [radius=0.04];
      \filldraw[fill=black] (.3,-0.8) circle [radius=0.04];      
      \draw (-0.2,0) node {$v$};
      
     \draw (4+0, 0) -- (4+-0.588, 0.809) -- (4+-1.54,  0.500) -- (4+-1.54, -0.500) -- (4+-0.588, -0.809) -- (4+0,0);
      \filldraw[fill=black] (4+0,0) circle [radius=0.04];
      \filldraw[fill=black] (4+-0.588, 0.809) circle [radius=0.04];
      \filldraw[fill=black] (4+-1.54, 0.500) circle [radius=0.04];
      \filldraw[fill=black] (4+-1.54, -0.500) circle [radius=0.04];
      \filldraw[fill=black] (4+-0.588, -0.809) circle [radius=0.04];
      \draw (4+0, 0) -- (4+0.588, 0.809) -- (4+1.54,  0.500) -- (4+1.54, -0.500) -- (4+0.588, -0.809) -- (4+0,0);
      \filldraw[fill=black] (4+0.588, 0.809) circle [radius=0.04];
      \filldraw[fill=black] (4+1.54, 0.500) circle [radius=0.04];
      \filldraw[fill=black] (4+1.54, -0.500) circle [radius=0.04];
      \filldraw[fill=black] (4+0.588, -0.809) circle [radius=0.04];     
      \draw (4-0.2,0) node {$v$};
\draw (2,-1.5) -- (2,1);      	
      \draw (0,-1.25) node {$\Gamma$};
      \draw (4,-1.25) node {$\Lambda$};      
   \end{tikzpicture}
    \caption{The defining graphs of RAAGs that are not quasi-isometric}\label{fig:raagsnotqi}
    \end{figure}

To obtain more refined quasi-isometry invariants, we need to determine whether two cylindrical vertex spaces are  \emph{relatively} quasi-isometric with respect to some peripheral structure. This is done in Proposition \ref{prop:relqicyl}. The argument required here is more technical, and we  need to adapt the methods used in \cite{papasoglu2002quasi} to build quasi-isometries that preserve a given peripheral structure.  
We thus define spaces called \emph{trees of graphs} and explain how to build quasi-isometries between trees of graphs using \emph{moves} of type I, II and III. 
Specifically, we need an understanding of why, using work of Papasoglu--Whyte, the free product $G_1 * G_2$ of two finitely generated, infinite groups  is quasi-isometric to:
\begin{enumerate}[label={(\Roman*)}]
\item $G_1 * G_1 * G_2$;
\item $G'_1 * G_2$, whenever $G'_1$ is quasi-isometric to $G_1$;
\item $G_1 * G_2 * \mathbb{Z}$.
\end{enumerate}
This can be done by applying moves of type I, II and III respectively.
Although the main arguments are based on  \cite{papasoglu2002quasi}, the terminology used here (moves, protographs, special subgraphs, isolated vertices, depth etc.) is not found in \cite{papasoglu2002quasi}.

\begin{figure}[b]
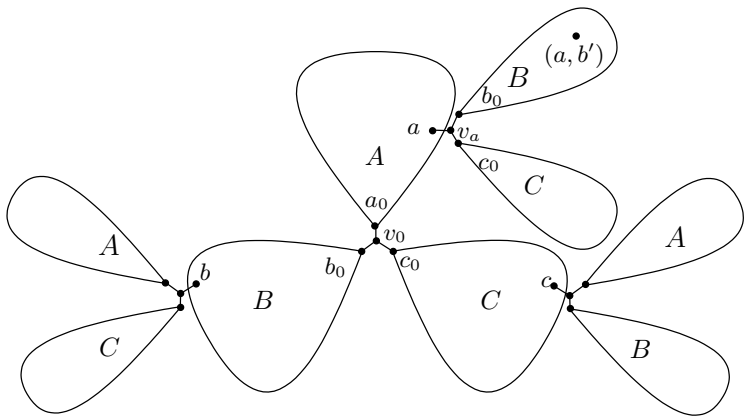
\caption[Schematic diagram of a free product of graphs]{Schematic diagram of a part of the graph $A*B*C$. If we remove the component of $(A*B*C)\backslash  \{v_a\}$ containing $v_0$, we obtain $\mathrm{branch}(v_a)$.}\label{fig:freeproductscheme}
\end{figure}

\begin{defn}
A \emph{based graph} $(A,a_0)$ is a graph $A$ with a distinguished vertex $a_0\in A$ called the \emph{basepoint}.
Let $\mathcal{R}:=\{(A,a_0), (B,b_0), \dots, (E,e_0)\}$ be a set of finitely many connected, infinite based graphs  of uniformly bounded valance, which we call \emph{protographs}.  
The free product is built out of copies of the protographs called \emph{special subgraphs}. Vertices that are not contained in a special subgraph are called \emph{isolated vertices}. Given $\mathcal{R}$ as above, we define the \emph{free product}\index{free product of graphs} $\Gamma= A*\dots *E$ of graphs inductively. 

We define $\Gamma^1$ by taking the disjoint union of a vertex $v_0$ and copies of each based protograph $(A,a_0)\in \mathcal{R}$.  We then join $v_0$ to each of the basepoints $a_0$ of the protographs via an edge. We call each copy of a protograph  	a special subgraph, and call the vertex $v_0$  an isolated vertex.
We define $\Gamma^{i+1}$ as follows. For every vertex of $\Gamma^i$, say $a\in A$, that is neither an isolated vertex nor  incident to an isolated vertex,    we attach a copy of each of the based protographs in $\mathcal{R}\backslash \{(A,a_0)\}$  via their basepoints to an isolated vertex, say $v_a$, which is itself attached via an edge to  $a$.  This is done so that for every isolated vertex $v_a$ of $\Gamma^{i+1}$, there is exactly one copy of every protograph attached to $v_a$. 

Noting that $\Gamma^1\subseteq \Gamma^2\subseteq\dots$ is a nested sequence of graphs, we define the graph $\Gamma:=\cup_{i\in \mathbb{N}} \Gamma^{i}$. We call each $\Gamma^i$ the \emph{depth $i$ subgraph of $\Gamma$}.
We endow $\Gamma$ with the path metric in which every edge has length one.
\end{defn}

The free product $A*B*C$ is illustrated schematically in Figure \ref{fig:freeproductscheme}. 
This definition differs very slightly from the procedure given in \cite{papasoglu2002quasi}, which doesn't include isolated vertices.  It is shown in \cite{papasoglu2002quasi} that the choice of basepoints doesn't affect the construction (up to bi-Lipschitz equivalence), provided the special subgraphs satisfy a mild quasi-homogeneous property; this will be evident from the description of type II moves, which we will describe shortly. 

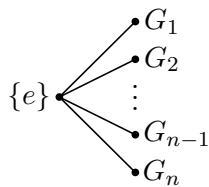
\begin{figure}[b]
    \begin{tikzpicture}[semithick,scale=1.0]

      \draw (0, 0) -- (1, 1.0);
      \draw (0, 0) -- (1, 0.5);
      \draw (0, 0) -- (1, -0.5);
      \draw (0, 0) -- (1, -1.0);
      \filldraw[fill=black] (0,0) circle [radius=0.04];
      \filldraw[fill=black] (1, 1.0) circle [radius=0.04];
      \filldraw[fill=black] (1, 0.5) circle [radius=0.04];
      \filldraw[fill=black] (1, -0.5) circle [radius=0.04];
      \filldraw[fill=black] (1, -1.0) circle [radius=0.04];
      \draw (1,0.1) node {$\vdots$};
      \draw (1.35,1.0) node {$G_1$};
      \draw (1.35,0.5) node {$G_2$};
      \draw (1.55,-0.5) node {$G_{n-1}$};
      \draw (1.35,-1.0) node {$G_n$};
      \draw (-0.4,0) node {$\{ e\}$};

   \end{tikzpicture}
    \caption{The graph of groups $\mathcal{G}$}\label{fig:treeofspacesgog}
  \end{figure}

\begin{rem}\label{rem:freeproductqitreeofgraphs}
The above construction is motivated by the following. Suppose $G=G_1*\dots *G_n$ is the free product of finitely many finitely generated infinite groups. For each $i$, let $A_i$ denote the Cayley graph of $G_i$ with respect to some finite generating set $S_i\subseteq G_i\backslash \{e\}$. Let $\mathcal{G}$ be a finite connected graph of groups as shown in Figure \ref{fig:treeofspacesgog}. Clearly $\pi_1(\mathcal{G})=G$. We equip $G$ with the word metric with respect to $S_1\sqcup \dots \sqcup S_n$. 

Let $X$ be the tree of spaces built by blowing up the Bass-Serre tree of $\mathcal{G}$ as in Proposition \ref{prop:qitreebs}. We observe that $X$ is isomorphic (as a graph) to the free product of graphs $A_1*\dots * A_n$. Each isolated vertex corresponds to a coset of the form $g\{e\}$. Each special subgraph that is a copy of the protograph $A_i$ corresponds to a coset of the form $g G_i$. 
By Proposition \ref{prop:qitreebs}, there is an equivariant quasi-isometry $f:G\rightarrow X$ such that $f(g G_i)$ gets mapped to within uniform finite Hausdorff distance of the corresponding special subgraph obtained from the protograph $A_i$.
\end{rem}

We may represent each non-isolated vertex $x\in A*\dots *E$ by an ordered tuple $(v_1,\dots, v_n)$ of vertices of protographs, chosen such that every adjacent pair $v_i$ and $v_{i+1}$ lie in different protographs, and  that  $v_i$ is not a basepoint for $i<n$. The first $(n-1)$ elements of the tuple encodes the unique sequence of points at which a path from $v_0$ to $x$ with no backtracking  exits a special subgraph. For example, $(a,b')$ denotes the vertex $b'$ contained in the copy of the protograph  $B$ attached to $(a)$  as shown in Figure \ref{fig:freeproductscheme}. We denote this special subgraph by $(a,B)$. 

Using this notation, we use $*$ as a wildcard character to define maps between free products of graphs. For instance, $(a,*)\mapsto (a',*)$  implies that $(a,b)\mapsto (a',b)$, that $(a,b,c)\mapsto (a',b,c)$, and so on. We write $(a,B,*)\mapsto (a',B,*)$ to denote $(a,b,*)\mapsto (a',b,*)$ for all $b\in B$. We may also write $(-,a,*)\mapsto (-,a',*)$, where the $-$ symbol denotes some fixed unspecified initial string.

The \emph{depth} of a vertex or special subgraph is defined to be the least $i$ such that $\Gamma^i$ contains that vertex or subgraph. We define the \emph{ascending link}  $\mathrm{lk}^+(v)$ of an isolated vertex $v$ to be the set of special subgraphs attached to $v$, whose depth is equal to that of $v$. We define $\mathrm{\emph{branch}}(v)$ to be the complement in $\Gamma$ of the component of $\Gamma\backslash \{v\}$ containing $v_0$.

A free product of graphs is a specific example of a more general construction known as a \emph{tree of graphs}\index{tree of graphs}. This is defined like a free product of graphs, except we no longer assume there  is  exactly one copy of each protograph  attached to each isolated vertex, and we no longer require the set $\mathcal{R}$ of protographs to be finite.
For instance, an isolated vertex could have two copies of some protograph $(A,a_0)\in \mathcal{R}$ attached, or could have no copies attached. We do however, assume that there is a uniform bound to the number of special subgraphs attached  to each isolated vertex. We extend the definition of depth and $\mathrm{lk}^+$  to trees of graphs. However, we ensure that $\mathrm{lk}^+$ includes multiplicities. For example, if two copies of a protograph $A$ are attached to $v$, these represent distinct elements of $\mathrm{lk}^+(v)$.
 
We also use the tuple notation to describe points on a tree of graphs. However, care must be taken when doing this, since it is not immediate when a tuple gives a well-defined vertex. For instance, the notation $(a,b)$ is ambiguous if there are two copies of the protograph $B$ attached to $a$.

We now describe three moves, which we call moves of type I, II and III, that map a tree of graphs to another tree of graphs.
 Each move will be a quasi-isometry  of the underlying space. Moves of type I and II will induce a bijection between special subgraphs in the domain and target trees of graphs. By applying these moves iteratively, we can construct   quasi-isometries of the form I, II and III described above.

Each of these moves is said to \emph{take place at an isolated vertex $v$}. A move that takes place at $v$ changes at least one special subgraph of $\mathrm{lk}^+(v)$, but acts trivially outside  $\mathrm{branch}(v)$. In particular, if $v$ is an isolated vertex of depth $i$, then a move that takes place at $v$ will fix all vertices of depth at most $i-1$.

\begin{figure}
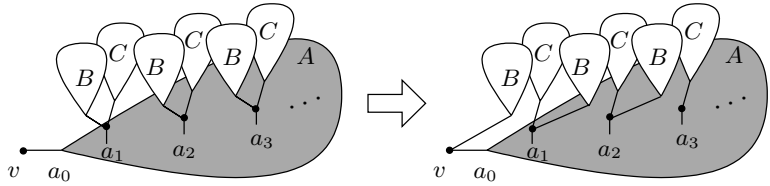
\caption[An example of a type I move]{A type I move that takes place at $v$: sliding $B$ onto $v$.}
\label{fig:sliding}
\end{figure}
\subsection*{Type I moves}\index{move on a tree of graphs!type I}\index{move on a tree of graphs!sliding}
A move of type I, also known as a \emph{sliding move}, is   shown pictorially in Figure \ref{fig:sliding}. We describe a type I move that takes place at an isolated vertex $v$. We fix some special subgraph $(A,a_0)\in \mathrm{lk}^+(v)$ and choose an infinite geodesic ray $(a_0,a_1,\dots)$ in $A$, starting at the basepoint $a_0\in A$.

 Let $v_{a_i}$ be the isolated vertex attached to $a_i$ for each $i>0$. We fix a based protograph $(B,b_0)$. A type I move involves `sliding' the edge attaching  $b_0$ from the isolated vertex $v_{a_{i+1}}$ to $v_{a_i}$ if $i>1$, and to  $v$ if $i=1$.  This can be described in the tuple notation  by the following transformations:   \begin{align*}
(-,a_1,B,*)&\mapsto (-,\bar B,*)\\
(-,a_2,B,*)&\mapsto (-,a_1,B,*)\\
(-,a_3,B,*)&\mapsto (-,a_2,B,*)\\
&\hspace{7pt}\vdots
\end{align*} 
We write $B$ and $\bar B$  to distinguish between the two copies of the photograph $B$ attached to the isolated vertex $v$.

 We denote the procedure described above as \emph{sliding $B$ onto $v$}. We can also perform the inverse move, which we call \emph{sliding $B$ off $v$} (providing there are copies of $B$ attached to $v$ and each $v_{a_i}$ for $i>0$). Either of these is called a type I move taking place at $v$. As shown in Lemma 1.2 of  \cite{papasoglu2002quasi}, this is a  bi-Lipschitz equivalence and induces a bijection of special subgraphs.

Using sliding moves, one can show that $A*A*B$ is bi-Lipschitz equivalent to $A*B$. This is done by sliding a copy of $A$ onto every isolated vertex. 
Provided we only perform a uniformly finite number of slides at each isolated vertex, the composition of infinitely many  sliding moves will always be a bi-Lipschitz equivalence. This will be discussed in more detail at the end of this section.

Sliding moves are much more flexible than one might first think. For instance, it would appear  in Figure \ref{fig:sliding} that  one cannot slide $A$ onto $v$. Although this cannot be done directly, it can be done stages: we  first slide $A$ onto $v_{a_i}$ for every $i>0$, then slide $A$ onto $v$, and then slide $A$ off $v_{a_i}$ for every $i>0.$\footnote{We are implicitly assuming that the initial tree of graphs in Figure \ref{fig:sliding} is the free product of graphs $A*B*C$.} Although this trick will not work for an arbitrary tree of graphs, it will work for all trees of graphs that arise in practice. We will therefore be able to slide any protograph onto or off any isolated vertex.

\subsection*{Type II moves}\index{move on a tree of graphs!type II}
We first observe the following: suppose $(A,a_0)$ and $(A',a'_0)$ are based protographs and $\phi:A\rightarrow A'$ is a bi-Lipschitz equivalence with $\phi(a_0)=a'_0$. Then the transformation $$(-,a,*)\mapsto (-,\phi(a),*)$$ is  a bi-Lipschitz equivalence. 
Quasi-isometries are significantly less rigid than bi-Lipschitz equivalences; for instance, they are not necessarily bijections. Consequently, if $\phi$ is a quasi-isometry, the above transformation is not necessarily a quasi-isometry.  However, by combining with type I moves, we can nonetheless replace the special subgraph $A$ in a tree of spaces with any  graph $A'$ quasi-isometric to $A$.

\begin{figure}
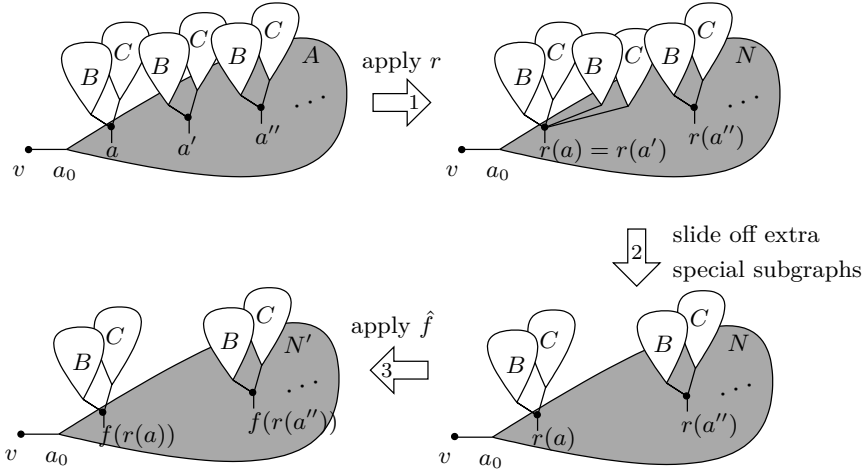
\caption[An example of a type II move]{
A type II move that takes place at $v$: 1) We apply $r$ to the special subgraph $A$ identifying isolated vertices $v_{a}$ and $v_{a'}$ if $r(a)=r(a')$. 2) We slide special subgraphs off  so that every point of $N$ not equal to the basepoint has exactly one copy of $B$ and $C$ attached. 3) We apply $\hat{f}$ to the special subgraph $N$. 4) We apply steps 1) and 2) in reverse.}
\label{fig:quasiisomtree}
\end{figure}

We now describe a type II moves that takes place at an isolated vertex $v$. Suppose $(A,a_0)\in \mathrm{lk}^+(v)$ and  that $f:A\rightarrow A'$ is a quasi-isometry with $f(a_0)=a'_0$. We may choose nets  $N\subseteq A$ and $N'\subseteq A'$ such that $f(N)=N'$ and $\hat f:=f|N:N\rightarrow N'$ is a bi-Lipschitz equivalence. Let $r:A\rightarrow N$ and $s:A'\rightarrow N'$ be surjective maps that are close to the identity, chosen such that $r^{-1}(r(a_0))$ and $s^{-1}(s(a'_0))$ each contain exactly one point, namely $a_0$ and $a'_0$ respectively.

We follow the procedure shown in Figure \ref{fig:quasiisomtree}, which is also described in the proof of Theorem 0.1 of \cite{papasoglu2002quasi}. We first apply the transformations $(-,a,*)\mapsto (-,r(a),*)$.
We do this so that whenever there are two points $a,a'\in A\backslash \{a_0\}$ such that $r(a)=r(a')$, special subgraphs in $\mathrm{lk}^+(v_a)$ and $\mathrm{lk}^+(v_{a'})$  get sent to special subgraphs in $\mathrm{lk}^+(v_{r(a)})$.

At each such isolated vertex $v_{r(a)}$, we slide off extra special subgraphs as necessary, ensuring that $\mathrm{lk}^+(v_{r(a)})=\mathrm{lk}^+(v_{a})$ for all $a\in A$. By our choice of $r$, we do not need to apply these sliding moves to the basepoint $a_0\in A$. After doing this, we apply the transformation $(-,a,*)\mapsto (-,\hat f (a),*)$. This does give a well-defined map as $\hat f$ is a bi-Lipschitz equivalence. We have thus `replaced' the special subgraph $A$ with $N'$.

Finally, we perform the inverse of the operations we applied when replacing $A$ with $N$, allowing us to replace $N'$ with $A'$. For each $n'\in N$, we slide extra special subgraphs onto $v_{n'}$, allowing us to identify $\mathrm{lk}^+(v_{n'})$ with $\sqcup_{x\in s^{-1}(n')}\mathrm{lk}^+(v_{x})$. For instance, suppose  $s^{-1}(n')=\{x,\overline{x}\}$ with $\mathrm{lk}^+(v_{x})=\{B,C\}$ and $\mathrm{lk}^+(v_y)=\{B,C\}$. 
Then we slide each of  $B$ and $C$ onto   $v_{n'}$  so that $\mathrm{lk}^+(v_{n'})=\{B,C,\overline{B},\overline{C}\}$, where  we write $\overline{B}$ and $\overline{C}$ to distinguish between the two copies of the protographs $B$ and $C$ in $\mathrm{lk}^+(v_{n'})$. We now apply the following transformations:
\begin{align*}
(-,n',B,*)&\mapsto (-,x, B,*)\\
(-,n',C,*)&\mapsto (-,x,C,*)\\
(-,n',\overline{B},*)&\mapsto (-,\overline{x},B,*)\\
(-,n',\overline{C},*)&\mapsto (-,\overline{x},C,*)
\end{align*}
We send $n'$ and $v_n'$ to $x$ and $v_x$ respectively for some arbitrarily chosen $x\in s^{-1}(n)$.
We apply this procedure for each $n'\in N$.

  We call this a \emph{type II move}. Type II moves allow one  to show that $A*B$ is quasi-isometric to $A'*B$ if $A$ is quasi-isometric to $A'$. 
We note the number of sliding moves we have to perform in this procedure is determined by the quasi-isometry constants of $f$.
  
This procedure is quite powerful, because it not only tells us that two trees of graphs are quasi-isometric, but allows us build a prescribed quasi-isometry of protographs into the quasi-isometry of the tree of graphs. More precisely, we have shown above that if $f:A\rightarrow A'$ is a quasi-isometry and $A$ is a special subgraph of a tree of graphs $\Gamma$, we may apply a type II move to $\Gamma$, obtaining  a tree of graphs $\Gamma'$ and a
quasi-isometry $\phi:\Gamma\rightarrow \Gamma'$ such that $\phi|_A$ is close to $f$. 

This observation is essential in proving Theorem \ref{thm:mainthmintro}, since it allows us to define quasi-isometries such that the diagram of Proposition \ref{prop:qitree} commutes up to uniformly bounded error. We summarise with the following Lemma, which can be proved by applying type II moves to all vertices in $\mathrm{lk}^+(v)$.

\begin{figure}
\resizebox{1.27\totalheight}{!}{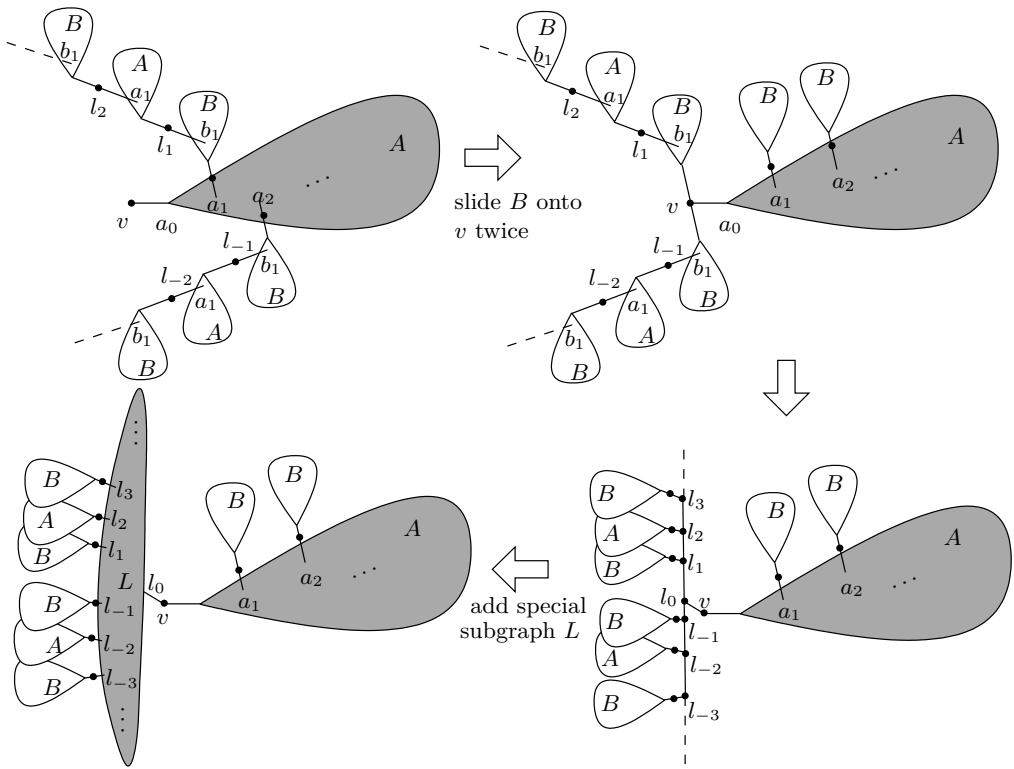}
\caption[An example of a type III move]{A type III move that takes place at $v$: we add the two-ended special subgraph $L$.}
\label{fig:addline}
\end{figure}
\begin{lem}\label{lem:qilink}
Let $\Gamma$ be a tree of spaces and $v\in \Gamma$ be an isolated vertex.  Suppose there is a collection $\mathcal{R}$ of based protographs and a bijection $\chi:\mathrm{lk}^+(v)\rightarrow \mathcal{R}$ such that for every $A\in \mathrm{lk}^+(v)$, there is a basepoint preserving quasi-isometry $f_A:A\rightarrow \chi(A)$. Then we may apply moves of type II at $v$ to obtain a quasi-isometry $f:\Gamma\rightarrow \Gamma'$ such that:
\begin{enumerate}
\item $f$ is a graph isomorphism outside $\mathrm{branch}(v)$;
\item $\mathrm{lk}^+(f(v))=\mathcal{R}$;
\item for every $A\in \mathrm{lk}^+(v)$,  $f|_A$ is close to $f_A$.
\end{enumerate}  
\end{lem}

\subsection*{Type III moves}\index{move on a tree of graphs!type III} A \emph{type III   move}  allows us to add or remove special subgraphs  to trees of graphs. More specifically, only bi-infinite lines can be added or removed.

  An example of a type III move that takes place at $v$ is shown in Figure \ref{fig:addline}. We choose vertices $a_1,a_2\in A$ and $b_1\in B$ such that $d(a_0,a_1)=d(a_1,a_2)=d(b_0,b_1)=1$. As shown in Figure \ref{fig:addline}, we define each $l_i$ to be the isolated vertex adjacent to either $(-,a_1,b_1,a_1,b_1,\dots)$ or $(-,a_2,b_1,a_1,b_1,\dots)$. We slide $B$ onto $v$ twice, and then observe that \begin{align}
 d(l_i,l_j) =    3\lvert i-j\rvert.
\end{align}
  
As shown in the third step of Figure \ref{fig:addline}, we can rearrange the $l_i$ vertices so they  lie along a single line $\{\dots,l_{-2},l_{-1},l_0,l_1,l_2,\dots \}$, adding vertices as in Figure \ref{fig:addline} by subdividing edges. We now artificially declare $L$ to be a special subgraph with basepoint $l_0$. This gives a tree of graphs with a two-ended special subgraph that doesn't correspond to a special subgraph of the starting tree of graphs. A type III move is defined to be either this procedure or its inverse.

A combination of type I and type III moves allows one to see that $A*B$ is quasi-isometric to  $A*B*\mathbb{Z}$ (or more generally, to  $A*B*F_n$). This is shown in Theorem 2.1 of \cite{papasoglu2002quasi}.  It follows from Lemma 3.2 of \cite{papasoglu2002quasi} that one-ended special subgraphs can never be added or removed.

\subsection*{Combining moves}
Each application of a move of type I, II or III provides us with a quasi-isometry $\gamma:\Gamma\rightarrow \Gamma'$ between trees of graphs. If for example,  we want to show that $A*B$ is quasi-isometric to  $A*B*B$, we need to apply infinitely many moves of type I, which means composing infinitely many quasi-isometries. We explain why this can be done.

We first consider the slide move shown in Figure \ref{fig:sliding}, which we denote $\gamma:\Gamma\rightarrow \Gamma'$. We have already remarked that $\gamma$ acts as an isometry outside $\mathrm{brach}(v)$. In fact, we observe that in some sense $\gamma$ only alters distances near the special subgraph $A$.  Although $\gamma$ is not an isometry when restricted to $\Gamma\backslash A$, it is an isometry when restricted to components of $\Gamma\backslash A$. Similarly, type II and III moves also only alter distances  `locally'. This allows us to compose infinitely many moves, provided they satisfy the following condition:

\begin{defn}
A sequence of moves applied to a tree of graphs is said to be \emph{uniformly locally finite}\index{move on a tree of graphs!uniformly locally finite sequence} if there is a uniform bound on the number of moves that take place at an isolated vertex, and there is a uniform bound on the  quasi-isometry constants of quasi-isometries used to define moves of type II.
\end{defn}
\begin{prop}\label{prop:compmoves}
The composition of a uniformly locally finite sequence of moves is a quasi-isometry.
\end{prop}
\begin{proof}
 By the uniformly locally finite condition, we can assume each move is a $(K,A)$-quasi-isometry between trees of graphs for some uniform constants $K$ and $A$. We claim that the composition of moves is then a $(K',A')$-quasi-isometry for suitable constants $K'$ and $A'$.
Suppose $x$ and $y$ are vertices in the tree of graphs that lie in special subgraphs (other cases can be dealt with similarly). Then there exists a unique sequence of points \begin{align}\label{eqn:diststring} x=x_0,y_0,v_1,x_1,y_1,v_2,x_2,\dots, v_n,x_n,y_n=y\end{align} such that $x_i$ and $y_i$ are vertices in the same special subgraph, $v_i$ is the isolated vertex joining $y_{i-1}$ and $x_i$, and there is no backtracking, i.e. no $v_i$ appears twice. The distance  $d(x,y)$ can be computed by summing the distance between successive points of (\ref{eqn:diststring}).

We first suppose that every move acts as a $K$-bi-Lipschitz equivalence on the tree of spaces. Indeed, suppose that   for all $x,y$, each move only alters distances in a finite subsequence of (\ref{eqn:diststring}). As only uniformly finitely many moves take place at each isolated vertex, we see that for any fixed finite substring of (\ref{eqn:diststring}) of length $R$, there are at most $S=S(R)$ moves that alter distances in that substring. This is  independent of $x$ and $y$. Thus when we compose infinitely many moves, the resulting map is  a $K^N$-bi-Lipschitz equivalence  for some sufficiently large $N$. 

We now explain how to deal with the general case where moves act as $(K,A)$-quasi-isometries. Quasi-isometries have additive errors  (the `$A$' factors). We need to show these cannot accumulate when applying a uniformly locally finite sequence of moves, e.g. if we apply 1000 moves, we need to show there is no $1000A$ factor. This does not follow directly from the above argument, which only shows multiplicative errors (the `$K$' factors) don't accumulate.

 It is sufficient to show this for maps such that $$\dots,v_i,x_i,y_i,v_{i+1},\dots$$ maps to $$ \dots,v_i,f(x_i),f(y_i),v_{i+1},\dots,$$ where $f$ is a $(K,A)$-quasi-isometry of the special subgraph containing $x_i$ and $y_i$. 
 Such maps occur in  type II moves, e.g. step 1 of Figure \ref{fig:quasiisomtree}. We can recover moves of type I, II and  III by composing such maps with bi-Lipschitz equivalences.

There are two cases. First suppose $d(x_i,y_i)\geq 2KA$. Then \begin{align*}\frac{1}{2K}d(x_i,y_i)\leq \frac{1}{K}d(x_i,y_i)-A&\leq d(f(x_i),f(y_i))\\&\leq Kd(x_i,y_i)+A\leq (K+\frac{1}{2K})d(x_i,y_i).\end{align*} 

Now suppose $d(x_i,y_i)\leq 2KA$. Let $i>0$. Since $d(v_i,x_i)=1$, we see that $$d(v_i,f(y_i))=1+d(f(x_i),f(y_i))\geq 1\geq \frac{1}{2KA+1}d(v_i,y_i).$$  

If $x_i\neq y_i$,  then $d(x_i,y_i)\geq 1$ (as $x_i$ and $y_i$ are vertices), so that $d(f(x_i),f(y_i))\leq (K+A)d(x_i,y_i)$. The case $x_i=y_i$ is clear, since $0=d(f(x_i),f(y_i))=d(x_i,y_i)=0$.
Consequently, we see that additive errors cannot accumulate when we apply a uniformly locally finite sequence of moves.    We conclude by applying the above argument in the case where all moves are bi-Lipschitz equivalences.
\end{proof}

\section{Relative quasi-isometry type of cylinders}\label{sec:cyl}
In this section, we analyse the relative quasi-isometry type of  cylinder stabilizers in the JSJ tree of cylinders of a RAAG. We recall from Section \ref{sec:jsjtoc}  that the cylindrical  vertex groups in the JSJ tree of cylinders of $A(\Gamma)$ are of the form $A({\mathrm{star}(v)})$ for some cut vertex  $v\in V\Gamma$. Proposition \ref{prop:relqicyl} gives necessary and sufficient for two cylindrical vertex groups to be relatively quasi-isometric.

Throughout this section, we assume that $\Gamma$ is connected and isn't of the form $\mathrm{star}(v)$ for some $v\in V\Gamma$, otherwise the JSJ tree of cylinders consists  of a single vertex.
Recall the definition of $B_1$ and $C_1$ from Definitions \ref{def:gogjsj} and \ref{def:gogtoc}. We let $\Gamma_1,\dots, \Gamma_n\in C_1$ and $\Gamma'_1,\dots \Gamma'_m\in B_1\backslash C_1$ be all the subgraphs in $B_1$ that contain $v$. We define $\emph{blocks}$ $A_i:=\mathrm{lk}^{\Gamma_i}(v)$ and $B_i:=\mathrm{lk}^{\Gamma'_i}(v)$ for each $i$.
 We say the $A_i$ are \emph{peripheral blocks} and the $B_i$ are \emph{non-peripheral blocks}. A \emph{coned-off block} is the join of a block with the cut vertex $v$. We call each $A(A_i)$ a \emph{peripheral factor} and each  $A(B_i)$ a \emph{non-peripheral factor} of the cylinder.

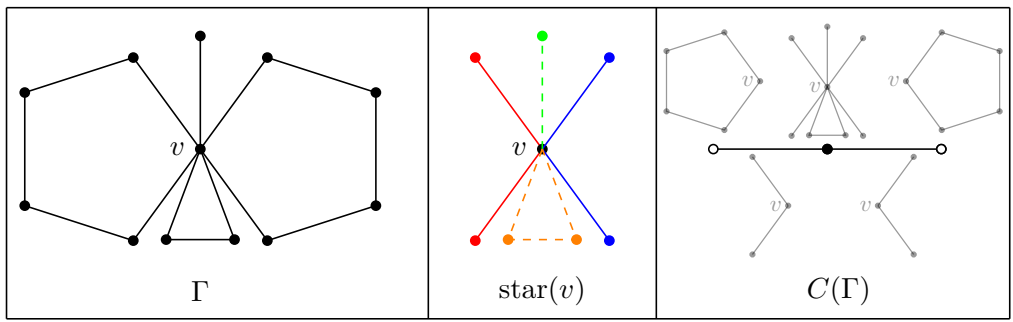
\begin{figure}[bp]
    \begin{tikzpicture}[semithick,scale=1.5]

      \draw (0, 0) -- (-0.588, 0.809) -- (-1.54,  0.500) -- (-1.54, -0.500) -- (-0.588, -0.809) -- (0,0);
      \filldraw[fill=black] (0,0) circle [radius=0.04];
      \filldraw[fill=black] (-0.588, 0.809) circle [radius=0.04];
      \filldraw[fill=black] (-1.54, 0.500) circle [radius=0.04];
      \filldraw[fill=black] (-1.54, -0.500) circle [radius=0.04];
      \filldraw[fill=black] (-0.588, -0.809) circle [radius=0.04];
      \draw (0, 0) -- (0.588, 0.809) -- (1.54,  0.500) -- (1.54, -0.500) -- (0.588, -0.809) -- (0,0);
      \filldraw[fill=black] (0.588, 0.809) circle [radius=0.04];
      \filldraw[fill=black] (1.54, 0.500) circle [radius=0.04];
      \filldraw[fill=black] (1.54, -0.500) circle [radius=0.04];
      \filldraw[fill=black] (0.588, -0.809) circle [radius=0.04];
	\draw (0,0) -- (0,1);
	\draw (0,0) -- (-.3,-0.8) -- (0.3,-0.8) -- (0,0);
	\filldraw[fill=black] (0,1) circle [radius=0.04];
      \filldraw[fill=black] (-.3,-0.8) circle [radius=0.04];
      \filldraw[fill=black] (.3,-0.8) circle [radius=0.04];      
      \draw (-0.2,0) node {$v$};
      
  \draw[red] (3, 0) -- (3-0.588, 0.809);
  \draw[red] (3,0) -- (3-0.588, -0.809);
      \filldraw[fill=black] (3,0) circle [radius=0.04];
      \filldraw[red,fill=red] (3-0.588, 0.809) circle [radius=0.04];
      \filldraw[red,fill=red] (3-0.588, -0.809) circle [radius=0.04];
  \draw[blue] (3, 0) -- (3+0.588, 0.809);
  \draw[blue] (3,0) -- (3+0.588, -0.809);
      \filldraw[blue] (3+0.588, 0.809) circle [radius=0.04];
      \filldraw[blue] (3+0.588, -0.809) circle [radius=0.04];
	\draw[dashed,green] (3,0) -- (3,1);
	\draw[dashed, orange] (3,0) -- (3-.3,-0.8) -- (3+0.3,-0.8) -- (3,0);
	\filldraw[green] (3,1) circle [radius=0.04];
      \filldraw[orange] (3-.3,-0.8) circle [radius=0.04];
      \filldraw[orange] (3+.3,-0.8) circle [radius=0.04];      
      \draw (3-0.2,0) node {$v$};      
      \draw (0,-1.25) node {$\Gamma$};
      \draw (3,-1.25) node {$\mathrm{star}(v)$};
      \draw (5.6,-1.25) node {$C(\Gamma)$};

      \draw (2,-1.5) -- (2,1.25);
      \draw (4,-1.5) -- (4,1.25);
      \draw (-1.7,1.25) -- (7.1,1.25);
      \draw (-1.7,-1.5) -- (7.1,-1.5);
      \draw (4.5,0) -- (5.5,0) -- (6.5,0); 
      \draw (-1.7,1.25) -- (-1.7,-1.5);
      \draw (7.1,1.25) -- (7.1,-1.5);
      \filldraw[fill=white] (4.5, 0) circle [radius=0.04];
      
      \draw (5.5,0.6) node[scale=0.8] {\begin{tikzpicture}[semithick,scale=1,fill opacity=0.4,draw opacity=0.4]
      \draw (0, 0) -- (-0.588, 0.809);
      \draw (-0.588, -0.809) -- (0,0);
      \filldraw[fill=black] (0,0) circle [radius=0.04];
      \filldraw[fill=black] (-0.588, 0.809) circle [radius=0.04];
      \filldraw[fill=black] (-0.588, -0.809) circle [radius=0.04];
      \draw (0, 0) -- (0.588, 0.809);
      \draw (0.588, -0.809) -- (0,0);
      \filldraw[fill=black] (0.588, 0.809) circle [radius=0.04];
      \filldraw[fill=black] (0.588, -0.809) circle [radius=0.04];
	\draw (0,0) -- (0,1);
	\draw (0,0) -- (-.3,-0.8) -- (0.3,-0.8) -- (0,0);
	\filldraw[fill=black] (0,1) circle [radius=0.04];
      \filldraw[fill=black] (-.3,-0.8) circle [radius=0.04];
      \filldraw[fill=black] (.3,-0.8) circle [radius=0.04];      
      \draw (-0.2,0) node {$v$};
      \end{tikzpicture}};
      
      \draw (4.5,0.6) node[scale=0.8] {\begin{tikzpicture}[semithick,scale=1,fill opacity=0.4,draw opacity=0.4]
       \draw (0, 0) -- (-0.588, 0.809) -- (-1.54,  0.500) -- (-1.54, -0.500) -- (-0.588, -0.809) -- (0,0);
      \filldraw[fill=black] (0,0) circle [radius=0.04];
      \filldraw[fill=black] (-0.588, 0.809) circle [radius=0.04];
      \filldraw[fill=black] (-1.54, 0.500) circle [radius=0.04];
      \filldraw[fill=black] (-1.54, -0.500) circle [radius=0.04];
      \filldraw[fill=black] (-0.588, -0.809) circle [radius=0.04];    
      \draw (-0.2,0) node {$v$};
      \end{tikzpicture}};
      \filldraw[fill=black] (5.5, 0) circle [radius=0.04];
      \filldraw[fill=white] (6.5, 0) circle [radius=0.04];
      
      \draw (6.5,0.6) node[scale=0.8] {\begin{tikzpicture}[semithick,scale=1,fill opacity=0.4,draw opacity=0.4]
      \draw (0, 0) -- (0.588, 0.809) -- (1.54,  0.500) -- (1.54, -0.500) -- (0.588, -0.809) -- (0,0);
      \filldraw[fill=black] (0,0) circle [radius=0.04];
      \filldraw[fill=black] (0.588, 0.809) circle [radius=0.04];
      \filldraw[fill=black] (1.54, 0.500) circle [radius=0.04];
      \filldraw[fill=black] (1.54, -0.500) circle [radius=0.04];
      \filldraw[fill=black] (0.588, -0.809) circle [radius=0.04];
      \draw (-0.2,0) node {$v$};
      \end{tikzpicture}};
      
      \draw (5,-0.5) node[scale=0.8] {\begin{tikzpicture}[semithick,scale=1,fill opacity=0.4,draw opacity=0.4]
      \draw (0, 0) -- (-0.588, 0.809);
      \draw (-0.588, -0.809) -- (0,0);
      \filldraw[fill=black] (0,0) circle [radius=0.04];
      \filldraw[fill=black] (-0.588, 0.809) circle [radius=0.04];
      \filldraw[fill=black] (-0.588, -0.809) circle [radius=0.04];     
      \draw (-0.2,0) node {$v$};
      \end{tikzpicture}};
      
      \draw (6,-0.5) node[scale=0.8] {\begin{tikzpicture}[semithick,scale=1,fill opacity=0.4,draw opacity=0.4]
      \draw (0, 0) -- (0.588, 0.809);
      \draw (0.588, -0.809) -- (0,0);
      \filldraw[fill=black] (0,0) circle [radius=0.04];
      \filldraw[fill=black] (0.588, 0.809) circle [radius=0.04];
      \filldraw[fill=black] (0.588, -0.809) circle [radius=0.04];     
      \draw (-0.2,0) node {$v$};
      \end{tikzpicture}};
      
      \filldraw[fill=black] (5.5, 0) circle [radius=0.04];
      \filldraw[fill=white] (6.5, 0) circle [radius=0.04];
      
   \end{tikzpicture}
    \caption{The decomposition of $\mathrm{star}(v)$ into coned-off blocks}\label{fig:cylinderdecomposition}
  \end{figure}
 
 The black (cylindrical) vertex groups of $C(\Gamma)$ are thus of the form \begin{align}\label{eqn:cyl} A({\mathrm{star}(v)})\cong\mathbb{Z}\times (A({A_1})*\dots * A({A_n})* A({B_1}) * \dots * A({B_m})).\end{align} Adjacent edge groups are precisely the $\mathbb{Z}\times A({A_i})$ terms for each peripheral factor $A(A_i)$.
  We note that $n\geq 1$ and $n+m\geq 2$. 
  \begin{rem}\label{rem:npfac}
As each non-peripheral block is connected, every $A(B_i)$ is either infinite cyclic or is one-ended.
  \end{rem}

We illustrate these ideas with Figure \ref{fig:cylinderdecomposition}. There is a single cut vertex of $\Gamma$, labelled $v$. The graph of groups $C(\Gamma)$  has three vertices: two rigid (white) vertices that correspond to the left and right pentagons, and one cylindrical (black) vertex. 
The cylinder has four coned-off blocks illustrated in different colours. The two coned-off blocks shown using filled lines are  peripheral. These correspond to incident edge groups in $C(\Gamma)$. The  two  coned-off  blocks shown using dashed lines are non-peripheral. These do not correspond to incident edge groups in $C(\Gamma)$.

Theorem \ref{thm:jsjtocqi} tells us that a quasi-isometry between two RAAGs induces relative quasi-isometries between their cylinder stabilizers. Cylinders have a direct product decomposition corresponding to the join decomposition $\{v\} \circ \mathrm{lk}(v)$. As $v$ is a cut vertex, $\mathrm{lk}(v)$ is disconnected and so doesn't split as a join.
Therefore,  Theorem \ref{thm:derahm} tells us that   a quasi-isometry between RAAGs induces quasi-isometries between the corresponding direct product factors.  By applying \cite{papasoglu2002quasi}, we thus find obstructions to RAAGs being quasi-isometric.

At the beginning of Section \ref{sec:freeproduct}, we demonstrated that the RAAGs with defining graphs in  Figure \ref{fig:raagsnotqi} can  be shown not to be quasi-isometric via such an argument.
We contrast this with Example 3.22 of \cite{huang2014quasi}, which describes two RAAGs that are commensurable, hence quasi-isometric. The cylinders in this example are of the form $\mathbb{Z}\times (F_2*F_2)$ and $\mathbb{Z}\times(F_3*F_3* \mathbb{Z})$, and these are quasi-isometric.
As remarked in the introduction, the RAAGs defined by graphs shown in Figure 1 of \cite{behrstock2017quasiflats} can also be shown not to be quasi-isometric by analysing the quasi-isometry types of cylinders.

We can do much better than this by observing that a quasi-isometry  doesn't just preserve the quasi-isometry type of cylinders but also their relative quasi-isometry type. The following proposition tells us when cylinders are relatively quasi-isometric.

\begin{prop}\label{prop:relqicyl}
Let $\Gamma$ and $\Lambda$ be connected subgraphs containing cut vertices $v\in V\Gamma$ and $w\in V\Lambda$. The cylindrical vertex groups $A(\mathrm{star}^\Gamma(v))$ and $A(\mathrm{star}^\Lambda(w))$ are relatively quasi-isometric if and only if they have the same quasi-isometry types  of peripheral factors (one-ended or otherwise), and the same  quasi-isometry types of one-ended non-peripheral factors.
\end{prop}
Before proving this, we need to modify the tree of spaces slightly so we are able to apply the machinery of Section \ref{sec:freeproduct}. The idea is as follows: suppose that $G_1$ and $G_2$ are finitely generated groups with finite generating sets $S_1$ and $S_2$. We can \emph{blow-up} the Cayley graph of  $G_1* G_2$ with respect to the generating set $S_1\sqcup S_2$ by adding edges and isolated vertices, thus obtaining a free product of the Cayley graphs of $G_1$ and $G_2$. See Remark \ref{rem:freeproductqitreeofgraphs}.

We now observe that a cylindrical vertex space  $X(\mathrm{star}^\Gamma(v))$ has a cylindrical decomposition $\mathbb{E}^1\times X(\mathrm{link}^\Gamma(v))$. 
As $v$ is a cut vertex, $A(\mathrm{link}^\Gamma(v))$ decomposes non-trivially as a free product $A({A_1})*\dots * A({A_n})* A({B_1}) * \dots * A({B_m})$ as shown above. We apply the above blowing up procedure so that the cylindrical vertex space is of the form $\mathbb{E}^1\times \Sigma_v$, where $\Sigma_v$ is a free product of graphs.

\begin{defn}
Let $G$ be a RAAG with JSJ tree of cylinders $T$ and tree of spaces $X$. Applying the above blowing-up procedure yields the \emph{blown-up tree of spaces} $\hat X$. It is clear that $X$ is equivariantly quasi-isometric to $\hat X$.
\end{defn}

For clarity, we recall the peripheral structure of the cylindrical vertex space $\mathbb{E}^1\times \Sigma_v$. Each incident edge $e$ of $T$ is identified with a coset $gA(A_i)$, and hence with a special subgraph $A_e$ in $\Sigma_v$. The associated edge space is then $\mathbb{E}^1\times A_e$. Elements of $\mathcal{S}_e$ are precisely subsets of the form $\mathbb{E}^1\times \{a\}$, where $a$ is a vertex of the special subgraph $A_e$. This follows from Remark \ref{rem:peripheralstrucutreraags}. Each special subgraph  of $\Sigma_v$ is either \emph{peripheral} if it is equal to $A_e$ for some  $e\in ET$, or \emph{non-peripheral} otherwise.

\begin{proof}[Proof of Proposition \ref{prop:relqicyl}]
 Suppose that two cylindrical vertex groups  $A(\mathrm{star}^\Gamma(v))$ and $A(\mathrm{star}^\Lambda(w))$  have the same quasi-isometry types of peripheral factors and the same quasi-isometry types of one-ended non-peripheral factors. We claim that their  vertex spaces in the blown-up tree of spaces are relatively quasi-isometric.  We do this by constructing a quasi-isometry $f:\Sigma_v\rightarrow \Sigma_w$ that bijectively sends peripheral special subgraphs to within uniform finite Hausdorff distance of peripheral special subgraphs. This implies the claim.  We proceed as follows.

 Let $v_0$ and $w_0$ be the depth one isolated vertices of $\Sigma_v$ and $\Sigma_w$ respectively. By Remark \ref{rem:npfac}, non-peripheral special subgraphs  are either lines  or are one-ended.  If necessary, we  apply a type III move at $v_0$ so that $\mathrm{lk}^+(v_0)$ and $\mathrm{lk}^+(w_0)$ consist of  the same quasi-isometry types of non-peripheral special subgraphs. 
 
We now apply type I moves at $v_0$ to obtain a bijection $\chi:\mathrm{lk}^+(v_0)\rightarrow\mathrm{lk}^+(w_0)$ that preserves quasi-isometry class and such that $\chi(A)$ is peripheral if and only if $A$ is.  We  apply Lemma \ref{lem:qilink} at $v_0$, where for each $A\in\mathrm{lk}^+(v_0)$, we choose a quasi-isometry $f_A:A\rightarrow \chi(A)$ that preserves the  base vertex; this can be done since each special subgraph is homogeneous. The depth one part of the  resulting tree of spaces  now agrees with  $\Sigma_w$.
 
 We proceed inductively, first applying the necessary type I and type III moves at each isolated vertex of depth $i$ and then  applying Lemma \ref{lem:qilink}. The tree of graphs we end up with is  $\Sigma_w$. As $\Sigma_v$ and $\Sigma_w$ each consist of finitely many isometry classes of special subgraphs, the quasi-isometry constants of all the $f_A$ can be assumed to be uniformly bounded.  This composition of moves is uniformly locally finite, so by Proposition \ref{prop:compmoves} we have the required quasi-isometry $\Sigma_v\rightarrow \Sigma_w$.
 
For the converse, we suppose there is a relative quasi-isometry $\phi:A(\mathrm{star}^\Gamma(v))\rightarrow A(\mathrm{star}^\Lambda(w))$. Theorem \ref{thm:derahm} then tells us that $\phi$ induces a quasi-isometry $f:\Sigma_v\rightarrow \Sigma_w$. As $\phi$  is a relative quasi-isometry, it must  bijectively map peripheral special subgraphs to within uniformly finite Hausdorff distance of peripheral special subgraphs. Lemma 3.2 of \cite{papasoglu2002quasi} also says that $f$ preserves one-ended special subgraphs up to uniformly finite Hausdorff distance. It follows that $\Sigma_v$ and $\Sigma_w$ must have the same quasi-isometry types of peripheral special subgraphs, and the same quasi-isometry types of one-ended non-peripheral special subgraphs.
\end{proof}

\section{Relative stretch factors}\label{sec:stretch}
So far, we have given a necessary condition for any two RAAGs to be quasi-isometric: their JSJ trees of cylinders, each of which is endowed with the na{\"\i}ve decoration, must be equivalent. We will   prove in Section \ref{sec:constructqi} that for a large class of RAAGs, the JSJ tree of cylinders is a complete quasi-isometry invariant.
This is done by finding some geometric decoration $\delta:T\rightarrow \mathcal{O}$ that encodes more data than the na{\"\i}ve decoration. We  want such a decoration $\delta$ to include relative stretch factors, which we define shortly. This mirrors Section 4 of \cite{cashenmartin2017quasi}, in which stretch factors are assigned to edges adjacent to \emph{quasi-isometrically rigid} vertices.

To motivate this, we give an example of two RAAGs that are not quasi-isometric but have JSJ trees that, when endowed with the na{\"\i}ve decoration, are equivalent.
\begin{exmp}\label{exmp:stretch}
Consider the RAAGs $A(\Gamma)$ and $A(\Lambda)$, where $\Gamma$ and $\Lambda$ are as shown in Figure \ref{fig:raagsnotqistretch}.  Let $P$ be the graph of the right-hand pentagon of $\Gamma$, and let $\Lambda'$ be the right-hand maximal biconnected subgraph of $\Lambda$. We observe that $\Lambda'$ is the graph obtained by doubling a pentagon along the star of a vertex.
In fact, $A(\Lambda')$ is isomorphic to an index 2 subgroup of the RAAG $A(P)$; see \cite[Example 1.4 and Section 11]{bestvina08raags}.  It is isomorphic to  the kernel of the map $\psi:A(P)\rightarrow\mathbb{Z}_2$ that sends the generator $v$ to 1 and all other generators to 0. In particular, $A(\Lambda')$ is quasi-isometric to $A(P)$. 

It is not hard to see that if we decorate the  JSJ trees of cylinders of $A(\Gamma)$ and $A(\Lambda)$ with the na{\"\i}ve decoration, then this decoration is stable under neighbour and vertex refinement. Thus the decorated trees are equivalent.

\begin{figure}[ht!]
    \begin{tikzpicture}[semithick,scale=1.5]

\draw (0+0, 0) -- (0+-0.588, 0.809) -- (0+-1.54,  0.500) -- (0+-1.54, -0.500) -- (0+-0.588, -0.809) -- (0+0,0);
      \filldraw[fill=black] (0+0,0) circle [radius=0.04];
      \filldraw[fill=black] (0+-0.588, 0.809) circle [radius=0.04];
      \filldraw[fill=black] (0+-1.54, 0.500) circle [radius=0.04];
      \filldraw[fill=black] (0+-1.54, -0.500) circle [radius=0.04];
      \filldraw[fill=black] (0+-0.588, -0.809) circle [radius=0.04];
      \draw (0+0, 0) -- (0+0.588, 0.809) -- (0+1.54,  0.500) -- (0+1.54, -0.500) -- (0+0.588, -0.809) -- (0+0,0);
      \filldraw[fill=black] (0+0.588, 0.809) circle [radius=0.04];
      \filldraw[fill=black] (0+1.54, 0.500) circle [radius=0.04];
      \filldraw[fill=black] (0+1.54, -0.500) circle [radius=0.04];
      \filldraw[fill=black] (0+0.588, -0.809) circle [radius=0.04];     
	\draw (-0.2,0) node {$v$};
      
     \draw (4+0, 0) -- (4+-0.588, 0.809) -- (4+-1.54,  0.500) -- (4+-1.54, -0.500) -- (4+-0.588, -0.809) -- (4+0,0);
      \filldraw[fill=black] (4+0,0) circle [radius=0.04];
      \filldraw[fill=black] (4+-0.588, 0.809) circle [radius=0.04];
      \filldraw[fill=black] (4+-1.54, 0.500) circle [radius=0.04];
      \filldraw[fill=black] (4+-1.54, -0.500) circle [radius=0.04];
      \filldraw[fill=black] (4+-0.588, -0.809) circle [radius=0.04];
      \draw (4+0, 0) -- (4+0.588, 0.809) -- (4+1.54,  0.500) -- (4+1.54, -0.500) -- (4+0.588, -0.809) -- (4+0,0);
      \filldraw[fill=black] (4+0.588, 0.809) circle [radius=0.04];
      \filldraw[fill=black] (4+1.54, 0.500) circle [radius=0.04];
      \filldraw[fill=black] (4+1.54, -0.500) circle [radius=0.04];
      \filldraw[fill=black] (4+0.588, -0.809) circle [radius=0.04];     
      \draw (4+0.588, 0.809) -- (4+1.54+0.5,  0.600) -- (4+1.54+0.5, -0.600) -- (4+0.588, -0.809);
      \filldraw[fill=black] (4+1.54+0.5,  0.600) circle [radius=0.04];
      \filldraw[fill=black] (4+1.54+0.5, -0.600) circle [radius=0.04];   
\draw (2,-1.5) -- (2,1);      	
	\draw (4-0.2,0) node {$v$};
      \draw (0,-1.25) node {$\Gamma$};
      \draw (4,-1.25) node {$\Lambda$};      
   \end{tikzpicture}
    \caption{The defining graphs of RAAGs that are not quasi-isometric, yet have equivalent JSJ trees of cylinders when equipped with the na\"ive decoration.}\label{fig:raagsnotqistretch}
    \end{figure}
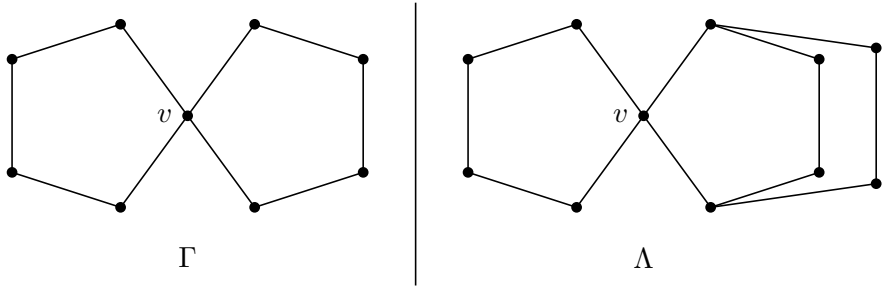

However, the two RAAGs are not quasi-isometric. This is because a quasi-isometry $A(P)\rightarrow A(\Lambda')$ necessarily shrinks distances by a factor of two along $v$-geodesics --- this will be shown in Lemma \ref{lem:stretcoarse}, and is witnessed by the fact that $A(\Lambda')$ is isomorphic to an index two subgroup of $A(P)$.   However, this shrinking cannot occur for a quasi-isometry $A(P)\rightarrow A(P)$. 

If  $A(\Gamma)$ and $A(\Lambda)$ are quasi-isometric, then it follows from Theorem \ref{thm:jsjtocqi} that there are suitable relative quasi-isometries $f_1:A(P)\rightarrow A(P)$ and $f_2:A(P)\rightarrow A(\Lambda')$ that agree along a $v$-geodesic. This cannot occur, since then $f_2$ would shrink distances on $l$ by a factor of two and $f_1$ would not. 

\end{exmp}

\subsection{Rigidity of standard geodesics}\label{sec:rigflexstretch}
We recall that if $X$ is a metric space and $A\subseteq X$ is a subspace, $[A]$ is the \emph{coarse equivalence class of $A$}, i.e. the set of all subspaces of $X$ that are at finite Hausdorff distance from $A$.

\begin{defn}
If $l\subseteq X(\Gamma)$ and $l'\subseteq X(\Lambda)$ are standard geodesics, then a \emph{based quasi-isometry} $f:(X(\Gamma),l)\rightarrow (X(\Lambda),l')$ is a quasi-isometry $f:X(\Gamma)\rightarrow X(\Lambda)$ such that $[f(l)]=[l']$.
Let $f:(X(\Gamma),l)\rightarrow (X(\Lambda),l')$ be a based quasi-isometry. We isometrically identify  $v(l)$ and $v(l')$ with copies of $\mathbb{Z}$. We define the \emph{stretch factor} $\mathrm{str}(f,l,C):=r$ if  $f|_{v(l)}$ is $C$-close to a map $v(l)\rightarrow v(l')$ of the form $n\mapsto \pm\lfloor r n\rfloor+b$ for some $r,b\in \mathbb{Q}$. If $\mathrm{str}(f,l,C)$ is defined for some sufficiently large $C$, we denote it simply by $\mathrm{str}(f,l)$.
\end{defn}

We are interested in what possible values  $\mathrm{str}(f,l)$ can take as we vary $f$. 
On the one hand, if $\Gamma$ and $\Lambda$ are complete graphs on $n$ vertices, then $f$ can be chosen such that $\mathrm{str}(f,l)$ takes any positive rational number, and may not even have a well-defined stretch factor.
In contrast, if a RAAG $A(\Gamma)$ has finite outer automorphism group, then $\mathrm{str}(f,l)$ is always well-defined and is independent of the choice of $f$, provided the coarse equivalence class $[f(l)]$ is fixed --- this will be shown in Lemma \ref{lem:stretcoarse}.  We thus say that the stretch factors are \emph{rigid}.

\begin{defn}
Let $A(\Gamma)$ be a RAAG and $l\subseteq X(\Gamma)$ be a standard geodesic.
We say that $l$ is an \emph{$\mathcal{R}$-geodesic} if for any RAAG $A(\Lambda)$, any standard geodesic $l'\subseteq X(\Lambda)$, and any pair of based quasi-isometries $f,g:(X(\Gamma),l)\rightarrow (X(\Lambda),l')$, then $\mathrm{str}(f,l)$ and $\mathrm{str}(g,l)$ are well-defined and equal.
\end{defn}
The following proposition is straightforward to prove.
\begin{prop}\label{prop:frgeodprop}
Let $A(\Gamma)$ be a RAAG and let $l\subseteq X(\Gamma)$ be a standard geodesic. Then:
\begin{enumerate}
\item for every $g\in A(\Gamma)$, the standard geodesic $gl$ is an $\mathcal{R}$-geodesic if and only if $l$ is;
\item if $f:(X(\Gamma),l)\rightarrow (X(\Lambda),l')$ is a  based quasi-isometry, then $l'$ is an $\mathcal{R}$-geodesic 
 if and only if $l$ is.
\end{enumerate}
\end{prop}

The following lemma, coupled with combination theorems  (Propositions   \ref{prop:derahmsystem}, \ref{prop:systemfreeprod} and Proposition \ref{prop:frsysjsj})  furnishes us with many  RAAGs containing $\mathcal{R}$-geodesics. 
\begin{lem}\label{lem:stretcoarse}
Let $A(\Gamma)$ be a RAAG of type II with trivial centre. Then every standard geodesic in $X(\Gamma)$ is an $\mathcal{R}$-geodesic.
More precisely, suppose that $f:X(\Gamma) \rightarrow X(\Lambda)$ is a $(K,A)$-quasi-isometry such that $d_\mathrm{Haus}(f(l),l')\leq A$ for some standard geodesics $l$ and $l'$. Then there exists a $C=C(K,A,\Gamma,\Lambda)$ such that $\mathrm{str}(f,l,C)=\mathrm{str}(f,l)$ is well-defined.

Moreover, let $f_*:\mathcal{R}(\Gamma)\rightarrow \mathcal{R}(\Lambda)$ be the simplicial isomorphism as in Proposition \ref{prop:inducedmapextcmplx}. 
Then  $\mathrm{str}(f,l)$ is the \emph{stretch factor of $f_*$ at $\Delta(l)$}, as defined in Lemma 5.4 of \cite{huang2015quasi}. 
\end{lem}

We first sketch  how the stretch factor   is defined in  \cite{huang2015quasi}.  This material will not be used elsewhere in this paper. Recall that $\mathcal{R}(\Gamma)$ is the extension complex associated to $A(\Gamma)$ and that $\Delta$ is the map that takes a  standard flat in $X(\Gamma)$ to the corresponding simplex of $\mathcal{R}(\Gamma)$.

Let $A(\Gamma)$ be a RAAG of type II with trivial centre and let $l\subseteq X(\Gamma)$ be a standard geodesic with $v:=\Delta(l)\in \mathcal{R}(\Gamma)$. Suppose that $\overline{v}\in V\Gamma$ is the label of $l$. A \emph{$v$-branch} is the induced subcomplex spanned by vertices in a component of $\mathcal{R}(\Gamma)\backslash \mathrm{Star}(v)$. Note that a $v$-branch always exists, since we have assumed that $\langle \overline{v}\rangle$ is not central in $A(\Gamma)$.
Lemma 6.2 of \cite{huang2014quasi} implies that the CAT(0) projection $X(\Gamma)\rightarrow l$ induces a well defined map $\pi_l$ from $v$-branches to vertices of $l$. Given a vertex $v$ of $l$, we say that a $v$-branch $B$ has \emph{height} $x$ if $\pi_l(B)=x$. 

Given a component $C$ of $\Gamma\backslash \mathrm{Star}(\overline v)$, we define $\partial C$ to be the induced subgraph spanned by vertices $\{x\in V\Gamma \backslash C\mid \textrm{there exists $y\in C$ such that $d(x,y)=1$}\}$. We define a \emph{$v$-peripheral subcomplex of $X(\Gamma)$} to be a standard subcomplex $K\subseteq X(\Gamma)$ such that $\Gamma_K=\partial C$ for some component  $C$ of $\Gamma\backslash \mathrm{Star}(\overline v)$.

Given a $v$-branch $B$, we define $\partial B$ to be the induced subcomplex of $\mathcal{R}(\Gamma)$ spanned by vertices $\{x\in \mathcal{R}(\Gamma)\backslash B\mid \textrm{there exists a $u\in B$ such that $d(u,x)=1$ }\}$. We say that $\partial B$ is a \emph{$v$-peripheral subcomplex} of $\mathcal{R}(\Gamma)$. We define  $\overline{B}$ to be the induced subcomplex of $\mathcal{R}(\Gamma)$ spanned by vertices in $B$ and $\partial B$. We remark that $\overline{B}$ is not necessarily equal to the topological closure of $B$.

Let $L$ be a component of $X(\Gamma)\backslash P_l$, where $P_l$ is the parallel set of $l$. We  define $\partial L$ to be the induced subcomplex of $X(\Gamma)$ spanned by the vertices  $$\{x\in X(\Gamma)\backslash L\mid \textrm{there exists a $y\in L$ such}  \textrm{ that $d(y,x)=1$}\}.$$ We define  $\overline{L}$ to be the induced subcomplex of $X(\Gamma)$ spanned by vertices in $L$ and $\partial L$. As above, we remark that $\overline{L}$ is not necessarily equal to the topological closure of $L$.

\begin{lem}[see Lemma 3.14 and Lemma 3.26 of \cite{huang2015quasi}]\label{lem:vbranchparallelset}
Let $X(\Gamma)$ be a RAAG of type II. Suppose that $l$ is a standard geodesic in $X(\Gamma)$ and $v=\Delta(l)$. Let $\overline{v}\in V\Gamma$ be the label of  $l$.
\begin{enumerate}
\item \label{vbranchparallelsetcorrespond} There is a 1-1 correspondence between $v$-branches and pairs $(C,K)$, where $C$ is a component of $\Gamma\backslash \mathrm{Star}(\overline{v})$ and $K$ is a $v$-peripheral subcomplex of $X(\Gamma)$ with $\Gamma_K=\partial C$.
\item There is a 1-1 correspondence between $v$-branches and components of $X(\Gamma)\backslash P_l$ defined as follows: for every $v$-branch $B$ there is a component $L$ of $X(\Gamma)\backslash P_l$ such that $\Delta(\overline{L})=\overline{B}$ and $\Delta(\partial L)=\partial B$. Moreover, $\partial L$ is a $v$-peripheral subcomplex and if $B$ corresponds to the pair $(C,K)$ as in \ref{vbranchparallelsetcorrespond}, then $\partial L=K$.
\end{enumerate}
Suppose $q:(X(\Gamma),\mathcal{P}_{\partial\Gamma})\rightarrow (X(\Lambda),\mathcal{P}_{\partial\Lambda})$ is a $(K,A)$-relative quasi-isometry and $q_*$ is a map as in Proposition \ref{prop:inducedmapextcmplx}. Let $l'$ be a standard geodesic such that $\Delta(l')=q_*(v)$.
\begin{enumerate}
\item[(3)]  There is a $D=D(K,A,\Gamma,\Lambda)$ such that for every component $L$ of $X(\Gamma)\backslash P_l$ there is a component $L'$ of $X(\Lambda)\backslash P_{l'}$ such that $d_\mathrm{Haus}(q(L),L')\leq D$, where $\Delta(\overline{L'})=q_*(\Delta(\overline{L}))$. Moreover $d_{\mathrm{Haus}}( q(\partial L),\partial L')\leq D$.
\end{enumerate}
\end{lem}
\begin{rem}\label{rem:projcompcomp}
If $L$ is a component of $X(\Gamma)\backslash P_l$ corresponding to a $v$-branch $B$, then the image of $\overline{L}$ under the projection $\pi_l:X(\Gamma)\rightarrow l$ is a single point and is equal to the height of $B$. This follows from  Lemma 6.2 of \cite{huang2014quasi} and  Lemma 3.26 of \cite{huang2015quasi}.
\end{rem}

Let $v$, $l$ and $l'$ be as in Lemma \ref{lem:vbranchparallelset}. We identify $v(l)$ and $v(l')$ with $\mathbb{Z}$. Lemma 5.4 of \cite{huang2015quasi} defines a quasi-isometry $h:\mathbb{Z}\rightarrow \mathbb{Z}$ as follows. 
Let $x\in v(l)$ and let $B$ be a $v$ branch of height $x$. Then $h(x)$ is defined to be the height of the $q_*(v)$-branch $q_*(B)$. (In fact, Lemma 5.4 of \cite{huang2015quasi} actually defines a collections of maps $h_i$, all of which are equal up to bounded distance.) 

It is shown in the proof of Lemma 5.4 of \cite{huang2015quasi} that $h$ is indeed a quasi-isometry and  that up to bounded distance, $h$ is independent of the choice of $B$. It is also shown that $h$ is bounded distance $E=E(K,A,\Gamma,\Lambda)$ from a map of the form $a\mapsto \pm \lfloor r a \rfloor +b$, where $b\in \mathbb{Z}$ and  $r$ is  a positive rational number that  is defined to be the \emph{stretch factor} of $q_*$ at $v$.

\begin{proof}[Proof of Lemma \ref{lem:stretcoarse}]

We fix a component $C$ of $\Gamma\backslash \mathrm{Star}(\overline{v})$. For each $x\in v(l)$, let $K_x$ denote the $v$-peripheral subcomplex such that $x\in K_x$ and $\Gamma_{K_x}=\partial C$. Using Lemma \ref{lem:vbranchparallelset}, let $B_x$ be the $v$-branch corresponding to $(C,K_x)$ and let $L_x$ denote the corresponding component of $X(\Gamma)\backslash P_l$. Let $L'_x$ be the component of $X(\Lambda)\backslash P_{l'}$ corresponding to the $q_*(v)$-branch $q_*(B_x)$. As above, we set $h(x)$ to be the height of $q_*(B_x)$. We  show that $h$, thought of as a map $v(l)\rightarrow v(l')$, is close to $q|_{v(l)}$. 

We note that $x\in K_x=\partial L_x$ so we see that $q(x)\in N_D(\partial L'_x)$, where $D$ is a constant as in Lemma \ref{lem:vbranchparallelset}. By the choice of $l'$, there is a constant $R$ such that $d_\mathrm{Haus}(q(l),l')\leq R$. Hence $q(x)\in N_D(\partial L'_x)\cap N_R(l')$. 
By Remark \ref{rem:projcompcomp} and Lemma 2.10 of \cite{huang2017quasiflat}, it follows that $d(h(x),q(x))\leq R'$, where $R'$ depends only on $R,D$ and the dimension of $X(\Lambda)$.
\end{proof}
\begin{rem}
It should be noted that Lemma \ref{lem:stretcoarse} is not necessarily true if we drop the assumption that $[f(l)]=[l']$ for some standard geodesic $l'$. Even if such an $l'$ does exist, we still need to ensure that $q_*(\Delta(l))=\Delta(l')$. This is  the case if $q_*$ is obtained via Proposition \ref{prop:inducedmapextcmplx}, but not necessarily true  for an arbitrary $q_*$ obtained using Theorem 2.5 of \cite{huang2015quasi}. 
\end{rem}

\subsection{Embellished decorations}
\label{sec:embellished}
We now use  $\mathcal{R}$-geodesics to refine the na{\"\i}ve decoration assigned to the JSJ tree of cylinder of a RAAG. We call this new decoration the \emph{embellished decoration}. Crucially, the RAAGs described in Example \ref{exmp:stretch} do not have equivalent JSJ trees of cylinders with respect to the embellished decoration. 

The ideas in this subsection are heavily influenced by Section 4 of \cite{cashenmartin2017quasi}.  However, the notion of quasi-isometric rigidity assumed here is weaker than that assumed in \cite{cashenmartin2017quasi}. For instance, in a RAAG with finite outer automorphism group, not every quasi-isometry $X(\Gamma)\rightarrow X(\Gamma)$ is close to an isometry. See Figure 1 in \cite{huang2014quasi} and the accompanying discussion for an example of such a quasi-isometry. 

For the remainder of this section, we let $\mathcal{Q}:=\{[[X(\Gamma),l]]\}$ denote the set of based quasi-isometry classes of RAAGs. For each $Q\in \mathcal{Q}$, we fix some  $(X(\Gamma_Q),l_Q)$  such that $[[X(\Gamma_Q),l_Q]]=Q$. The space $(X(\Gamma_Q),l_Q)$ will play a similar role to that of the rigid model spaces in Section 4 of \cite{cashenmartin2017quasi}.
\begin{defn}\label{defn:stretch}
Suppose that  $l\subseteq X(\Gamma')$ is an $\mathcal{R}$-geodesic. If $Q=[[(X(\Gamma'),l)]]$, we pick some based quasi-isometry $f:(X(\Gamma'),l)\rightarrow (X(\Gamma_Q),l_Q)$. We then define $\mathrm{str}(\Gamma', l):=\mathrm{str}(f,l)$. 
\end{defn}
It follows from the definition of an $\mathcal{R}$-geodesic that $\mathrm{str}(\Gamma', l)$  is well-defined and independent of the choice of $f$. However, $\mathrm{str}(\Gamma', l)$ does depend on the choice of $(X(\Gamma_Q),l_Q)$. If we use a different space $(X(\Gamma_Q),l_Q)$ to define stretch factors, then for every $(X(\Gamma'),l)\in Q$, the stretch factor $\mathrm{str}(\Gamma', l)$ will be scaled by some $\lambda_Q$.

The following definition makes use of the graph of groups $C(\Gamma)$ and subgraphs $C_0,C_1$ of $\Gamma$ as in  Definition \ref{def:gogtoc}.
\begin{defn}\label{defn:r edge and cyclinder}
Let $A(\Gamma)$ be a one-ended RAAG. Every edge in $C(\Gamma)$ is of the form $e=A(\mathrm{star}^{\Gamma'}(v))$, where  $\Gamma'\in C_1$ and $v\in \Gamma'\cap C_0$.  Let $l$ be the standard geodesic in $X(\Gamma')$ corresponding to the coset $\langle v \rangle$. 
We say that $e$ is an  \emph{$\mathcal{R}$-edge} if $l$ is an $\mathcal{R}$-geodesic. 
For every such  $\mathcal{R}$-edge $e$, we define $\mathrm{str}(e):=\mathrm{str}(\Gamma', l)$.
We then say that a cylinder is an \emph{$\mathcal{R}$-cylinder} if it has at least one incident $\mathcal{R}$-edge.
\end{defn}

These notions can be lifted to the JSJ tree of cylinders $T(\Gamma)$ in the obvious way, i.e. an edge $f=gA(\mathrm{star}^{\Gamma'}(v))$ of $T(\Gamma)$ is an $\mathcal{R}$-edge precisely when $e=A(\mathrm{star}^{\Gamma'}(v))$ is an $\mathcal{R}$-edge of $C(\Gamma)$, in which case $\mathrm{str}(f)=\mathrm{str}(e)$.

\begin{exmp}
Consider the RAAG $A(\Gamma)$ defined in Figure \ref{fig:gogexamples}. There are two  cylindrical vertex groups in $C(\Gamma)$. Since the RAAG whose defining graph is a pentagon is of type II with trivial centre, the left cylinder  of $C(\Gamma)$, as shown in Figure \ref{fig:gogexamples}, is an $\mathcal{R}$-cylinder. It is straightforward to see that no standard geodesic in a free abelian group is an $\mathcal{R}$-geodesic. Thus the right  cylinder in Figure \ref{fig:gogexamples} is not an $\mathcal{R}$-cylinder.
\end{exmp}

\begin{lem}\label{lem:redgeqi}
If $\phi:A(\Gamma)\rightarrow A(\Lambda)$ is a quasi-isometry, then the induced tree isomorphism between JSJ trees of cylinders   preserves $\mathcal{R}$-edges and hence $\mathcal{R}$-cylinders.
\end{lem}
\begin{proof}
This follows from Theorem \ref{thm:jsjtocqi} and Proposition \ref{prop:frgeodprop}.
\end{proof}

Suppose a cylindrical vertex $v$ in the JSJ tree of cylinders is incident to two $\mathcal{R}$-edges $e$ and $f$. Then we define the \emph{relative stretch factor} $$\mathrm{relstr}(e,f):=\frac{\mathrm{str}(f)}{\mathrm{str(e)}}.$$

\begin{lem}\label{lem:relstrqi}
Relative stretch factors are preserved under quasi-isometries.
\end{lem}
\begin{proof}
Let $\phi:A(\Gamma)\rightarrow A(\Lambda)$ be a quasi-isometry. Suppose $e$ and $f$ are $\mathcal{R}$-edges in the JSJ tree of cylinders of a RAAG $A(\Gamma)\in \mathcal{C}$, and that $\iota e=\iota f$ is a cylinder. We want to show that $\mathrm{relstr}(e,f)=\mathrm{relstr}(\phi_*e,\phi_*f)$.  Let $l$ be a standard geodesic in $X_{\tau e}$  corresponding to $e$, and let $l'$ be a standard geodesic in $X_{\phi_*(\tau e)}$ such that $[\phi(l)]=[l']$. 

 We know from Theorem \ref{thm:jsjtocqi} that $\phi$ induces a relative quasi-isometry $\phi_{\tau_e}:(X_{\tau e},\mathcal{P}_{\tau e})\rightarrow (X_{\phi_*(\tau e)},\mathcal{P}_{\phi_*(\tau e)})$. We let $Q=[[X_{\tau(e)},l]]=[[X_{\phi_*(\tau e)},l']]$. We define maps  $$(X_{\tau(e)},l)\xrightarrow{f_1} (X(\Gamma_Q),l_Q) \xrightarrow{f_2} (X_{\phi_*(\tau e)},l')$$ such that the composition $f_2\circ f_1$ is close to $\phi_{\tau_e}$. 
 Using Definitions \ref{defn:stretch} and \ref{defn:r edge and cyclinder}, we see that $\mathrm{str}(\phi,l)=\mathrm{str}(\phi_{\tau_e},l)=\mathrm{str}(f_1,l)\mathrm{str}(f_2,l_Q)=\frac{\mathrm{str}(e)}{\mathrm{str}(\phi_*(e))}$.

We apply the same argument with a standard geodesic $k$ in $X_{\tau f}$ representing the edge $f$, and a standard geodesic $k'$ in $X_{\phi_*(\tau f)}$ such that $[\phi(k)]=[k']$. We remark that $k$ and $l$ are parallel, since they both correspond to subsets of the form $\mathbb{E}^1\times \{x\}$ in the  vertex space of $\iota e=\iota f$ (see Definition \ref{defn:cyldecomp} and the following discussion).
 Thus $$\frac{\mathrm{str}(e)}{\mathrm{str}(\phi_*(e))}=\mathrm{str}(\phi,l)=\mathrm{str}(\phi,k)=\frac{\mathrm{str}(f)}{\mathrm{str}(\phi_*(f))},$$ which  implies that $\mathrm{relstr}(e,f)=\mathrm{relstr}(\phi_*e,\phi_*f)$.
\end{proof}

In light of Lemma \ref{lem:relstrqi}, it is natural to incorporate relative stretch factors into decorations. To do this, we first \emph{normalise} the relative stretch factors as follows. For each cylinder $v$, we observe that  $\{\mathrm{str}(e)\mid \iota e=v \textrm{ and } e \textrm{ is an } \mathcal{R}\textrm{-edge}\}$ is finite. We therefore define $\mathrm{relstr}(f):=\mathrm{relstr}(e,f)$ for  an $\mathcal{R}$-edge $e$, where $e$ is chosen such that  $\mathrm{str}(e)$ is minimal amongst all edges $e$ that are incident to the same cylinder as $f$.

\begin{defn}\label{defn:initialdecoration}
Let $A(\Gamma)$ be a one-ended RAAG  with JSJ tree of cylinders $T$. We define a decoration on $T$ as follows: we decorate each vertex with its type (cylinder or rigid) and relative quasi-isometry class, and decorate each $\mathcal{R}$-edge $f$ with the normalised relative stretch factor $\mathrm{relstr}(f)$.  We call this the \emph{embellished decoration}.
\end{defn}
\begin{cor}\label{cor:embeldecgeom}
The embellished decoration of the JSJ tree of cylinders of a RAAG is geometric.
\end{cor}
\begin{proof}
This follows from Theorem \ref{thm:jsjtocqi}, Lemma \ref{lem:redgeqi} and Lemma \ref{lem:relstrqi}.
\end{proof}
\begin{rem}
We have defined the embellished decoration using all $\mathcal{R}$-geodesics, which may be hard to calculate in general. However, one can still define a similar decoration using a subcollection of  $\mathcal{R}$-geodesics, without knowing that these are all the $\mathcal{R}$-geodesics. This decoration is a priori finer than the na{\"\i}ve decoration and coarser than than the embellished decoration, and can thus be used to distinguish quasi-isometry classes of RAAGs.
\end{rem}

\subsection{Dovetail RAAGs}\label{sec:dovetail}
We now begin the proof of Theorem \ref{thm:mainthm}, which says for a large class of RAAGs, the embellished decoration is a complete quasi-isometry invariant. To do this we define the class of   \emph{dovetail} RAAGs. Recall that in carpentry, a dovetail join consists of two pieces of wood that fit together nicely via interlocking joints. Analogously, if we have a tree of dovetail RAAGs that are equivalent with respect to the embellished decoration, we can choose a tree of quasi-isometries that fit together nicely, i.e. the hypotheses of Proposition \ref{prop:qitree} are satisfied. 

\begin{defn}
If $l\subseteq X(\Gamma)$ is a standard geodesic and is not an $\mathcal{R}$-geodesic, we say that $l$ is an \emph{$\mathcal{F}$-geodesic}.
\end{defn}

To motivate the definition of a dovetail RAAG,  we  examine the behaviour of $\mathcal{F}$-geodesics.
It follows easily from the definitions that if  $l\subseteq X(\Gamma)$ is an $\mathcal{F}$-geodesic, there is a based quasi-isometry $f:(X(\Gamma),l)\rightarrow (X(\Gamma),l)$ such that $\mathrm{str}(f,l)\neq 1$, i.e. $\mathrm{str}(f,l)$ is either undefined or is defined and not equal to 1. By taking powers of $f$ and its coarse inverse, we see that rigidity fails spectacularly: there are infinitely many based self quasi-isometries of $(X(\Gamma),l)$ which induce (coarsely) distinct maps $v(l)\rightarrow v(l)$. Thus every standard geodesic is either an $\mathcal{R}$-geodesic or has a very large amount of flexibility. (The $\mathcal{F}$ stands for flexibility.)

We require $\mathcal{F}$-geodesic of a dovetail RAAG  to have an even larger amount of flexibilty. Namely, if two dovetail RAAGs are quasi-isometric, there is a nice quasi-isometry between them which can take any set of prescribed stretch factors on parallelism  classes of  $\mathcal{F}$-geodesics. Moreover, we require that if the original quasi-isometry preserves some peripheral structure and decoration in the sense described below, then the new quasi-isometry does too. This is necessary to construct a tree of quasi-isometries that piece together to satisfy the hypotheses of Proposition \ref{prop:qitree}. While the definition is quite technical, we later show that it is satisfied by a very large class of RAAGs. 
%
%

Let $A(\Gamma)$ be a RAAG and suppose $B\subseteq  V\Gamma$. We recall from Definition \ref{defn:pstrucraag} that $B$ induces the peripheral structure $\mathcal{P}_B$ of a RAAG. A \emph{decoration} on the peripheral structure $\mathcal{P}_B$ is an $A(\Gamma)$-invariant map $\delta:\mathcal{P}_B\rightarrow \mathcal{O}$, where $\mathcal{O}$ is an (arbitrary)  set of \emph{ornaments}. Suppose we are given RAAGs  $A(\Gamma)$ and $A(\Lambda)$, peripheral structures $\mathcal{P}_{B_\Gamma}$ and $\mathcal{P}_{B_\Lambda}$, decorations $\delta_\Gamma$ and $\delta_\Lambda$ with the same set ornaments, and standard geodesics $l\subseteq X(\Gamma)$ and $l'\subseteq X(\Lambda)$. Then a \emph{decoration-preserving based relative quasi-isometry} $$(X(\Gamma),\mathcal{P}_{B_\Gamma},\delta_\Gamma,l)\rightarrow (X(\Lambda),\mathcal{P}_{B_\Lambda},\delta_\Lambda,l')$$   consists of a relative quasi-isometry $(f,f_*):(X(\Gamma),\mathcal{P}_{B_\Gamma})\rightarrow (X(\Lambda),\mathcal{P}_{B_\Lambda})$ such that:
\begin{enumerate}
\item $f_*$ is decoration preserving, i.e. $\delta_\Gamma=\delta_\Lambda \circ f_*$;
\item  $[f(l)]=[l']$.
\end{enumerate}
When unambiguous, we simply call such an $f$ a \emph{quasi-isometry} $$f:(X(\Gamma),\mathcal{P}_{B_\Gamma},\delta_\Gamma,l)\rightarrow (X(\Lambda),\mathcal{P}_{B_\Lambda},\delta_\Lambda,l').$$

We can now give the definition of a dovetail RAAG.
\begin{defn} \label{defn:frsys}
A function $\lambda:T\rightarrow \mathbb{Q}_{>0}$ is said to be \emph{bi-bounded} if there exists an $M \geq 1$ such that  $\frac{1}{M}\leq \lambda(t)\leq M$ for all $t\in T$. We say that $A(\Gamma)$ is a \emph{dovetail RAAG} if the following holds.

Suppose that $f:(X(\Gamma),\mathcal{P}_{B_\Gamma},\delta_\Gamma,l)\rightarrow (X(\Lambda),\mathcal{P}_{B_\Lambda},\delta_\Lambda,l')$ is a quasi-isometry as above. Furthermore, suppose that  $\lambda:\mathcal{P}_{B_\Gamma}\rightarrow \mathbb{Q}_{>0}$ is a bi-bounded function that is constant on $\delta^{-1}(o)$ for every ornament $o\in \mathcal{O}$. Then there exist a constant $C$ and a quasi-isometry $$(g,g_*):(X(\Gamma),\mathcal{P}_{B_\Gamma},\delta,l)\rightarrow( X(\Lambda),\mathcal{P}_{B_\Lambda},\delta',l')$$ such that:
\begin{enumerate}
\item \label{def:fsys} $\mathrm{str}(g,l'',C)=\lambda(\mathcal{S}_{l''})$ for every $\mathcal{F}$-geodesic $l''$ labelled by an element of $B_\Gamma$.
 \item \label{def:rsys} $\mathrm{str}(g,l'',C)=\mathrm{str}(g,l'')$ for every $\mathcal{R}$-geodesic $l''$ labelled by an element of $B_\Gamma$.
 
\end{enumerate}
\end{defn}
\begin{rem}
The value of $\lambda$ on parallelism classes defined by $\mathcal{R}$-geodesics is redundant: it is not used in either the above definition or any subsequent argument. Thus one should really think of $\lambda$ as a function defined on parellism classes of $\mathcal{F}$-geodesics. However, we live with the redundant values for ease of notation.
\end{rem}

\begin{ques}\label{ques:dovetail}
Is every RAAG a dovetail RAAG?
\end{ques}
 We provide positive evidence for Question \ref{ques:dovetail} by showing that free groups, free abelian groups, tree RAAGs, RAAGs with finite outer automorphism group, and more generally RAAGs of type II, are all dovetail RAAGs. We also show that the class of dovetail RAAGs is closed under taking free products, direct products, and amalgamating along special cyclic subgroups. An affirmative answer to this question would strengthen the results of this paper.

 \begin{prop}\label{prop:cliquesystem}
Free abelian groups are dovetail RAAGs, all of whose standard geodesics are $\mathcal{F}$-geodesics.
 \end{prop}
 \begin{proof}
By defining a quasi-isometry of $\mathbb{Z}^n$  as a product of $n$ homotheties, we can define arbitrary stretch factors on standard geodesics. Moreover, by permuting coordinates, we can assume any peripheral structure and decoration is preserved. Thus $\mathbb{Z}^n$ is a dovetail RAAG.
 \end{proof}
We now show that the class of dovetail RAAGs is closed under taking direct and free products.

 \begin{prop}\label{prop:derahmsystem}
 Suppose $A(\Gamma)$ has a de Rahm decomposition $A(\Gamma_1)\times \dots \times A(\Gamma_k)$ (see Section \ref{sec:raags}). Then $A(\Gamma)$ is a dovetail RAAG if and only if each $A(\Gamma_i)$ is a dovetail RAAG. 
 \end{prop}

 \begin{proof}
We first suppose that each $A(\Gamma_i)$ is a dovetail RAAG. Let $$f:(X(\Gamma),\mathcal{P}_{B_\Gamma},\delta,l)\rightarrow( X(\Lambda),\mathcal{P}_{B_\Lambda},\delta',l')$$ be a  quasi-isometry and $\lambda:\mathcal{P}_{B_\Gamma}\rightarrow \mathbb{Q}_{>0}$ be a bi-bounded function. 
By Proposition \ref{prop:prodrelqi}, there is a relative quasi-isometry $g:(X(\Gamma),\mathcal{P}_{B_\Gamma})\rightarrow( X(\Lambda),\mathcal{P}_{B_\Lambda})$ such that $g_*=f_*$ and $g$ splits as product $g_1\times \dots \times g_k$ of quasi-isometries.  Each $g_i:X(\Gamma_i)\rightarrow X(\Lambda_i)$ preserves the induced peripheral structure and decoration on each de Rahm factor. As each $A(\Gamma_i)$ is a dovetail RAAG, we can thus define a quasi-isometry $h_i:X(\Gamma_i)\rightarrow X(\Lambda_i)$ with stretch factors as prescribed by $\lambda$ along $\mathcal{F}$-geodesics. This can be done whilst preserving the peripheral structure and decoration. The required quasi-isometry $h:(X(\Gamma),\mathcal{P}_{B_\Gamma},\delta,l)\rightarrow( X(\Lambda),\mathcal{P}_{B_\Lambda},\delta',l')$ is defined to be the product $h_1\times \dots \times h_k$. The converse follows readily from Theorem \ref{thm:derahm}.
 \end{proof}
 
\begin{cor}\label{cor:dirprodsystem}
The direct product of dovetail RAAGs is itself a dovetail RAAG.
\end{cor}
\begin{proof}
It is sufficient to show that if $A(\Gamma)$ and $A(\Lambda)$ are dovetail RAAGs, then so is $A(\Gamma)\times A(\Lambda)$.
Let $A(\Gamma)=\mathbb{Z}^n\times A(\Gamma_2)\times \dots \times A(\Gamma_k)$ and $A(\Lambda)=\mathbb{Z}^m\times A(\Lambda_2)\times \dots \times A(\Lambda_l)$ be de Rahm decompositions. By Proposition \ref{prop:derahmsystem}, all the $A(\Gamma_i)$ and $A(\Lambda_i)$ must be dovetail RAAGs. Then $A(\Gamma)\times A(\Lambda)$ has a de Rahm decomposition of the form $$\mathbb{Z}^{n+m}\times A(\Gamma_2)\times \dots \times A(\Gamma_k)\times A(\Lambda_2)\times \dots \times A(\Lambda_l).$$ We are done by  Propositions \ref{prop:cliquesystem} and \ref{prop:derahmsystem}.
\end{proof}

\begin{prop}
Let $A(\Gamma)$ be an arbitrary RAAG of type II, possibly with non-trivial centre. Then $A(\Gamma)$ is a dovetail RAAG.
\end{prop}
\begin{proof}
We may write $\Gamma=\Gamma_1\circ \Gamma_2$, where $\Gamma_1$ is a maximal clique join factor. Thus $A(\Gamma_2)$ is a RAAG of type II with trivial centre. By Lemma \ref{lem:stretcoarse} and Proposition \ref{prop:cliquesystem}, both $A(\Gamma_1)$ and $A(\Gamma_2)$ are dovetail RAAGs, hence Corollary \ref{cor:dirprodsystem} tells us that $A(\Gamma)$ is a dovetail RAAG.
\end{proof}

\begin{prop}\label{prop:systemfreeprod}

Suppose $A(\Gamma)$ has a Gru\v{s}ko decomposition $A(\Gamma_1)* \dots *A(\Gamma_k)$ and that each $A(\Gamma_i)$ is a dovetail RAAG.  Then $A(\Gamma)$ is a dovetail RAAG. 
\end{prop}
\begin{proof}
This will be proved using the machinery of  Section \ref{sec:freeproduct}. Suppose $f:(X(\Gamma),\mathcal{P}_{B_\Gamma},\delta,l)\rightarrow( X(\Lambda),\mathcal{P}_{B_\Lambda},\delta',l')$ is a quasi-isometry and let $\lambda:\mathcal{P}_{B_\Gamma}\rightarrow \mathbb{Q}_{>0}$ be a bi-bounded function. 
We may think of $f$ as a map $\hat f:\Sigma_\Gamma\rightarrow \Sigma_\Lambda$  between free products of graphs. Let $S_\Gamma$ be the collection of special subgraphs of $\Sigma_\Gamma$ that are either one-ended or are two-ended and correspond to a standard geodesic labelled by an element of $B_\Gamma$. We define $S_\Lambda$ similarly. Every standard geodesic corresponds to a bi-infinite geodesic contained in a special subgraph of $\Sigma_\Gamma$ or $\Sigma_\Lambda$. We thus decorate special subgraphs in $S_\Gamma$ and $S_\Lambda$ as follows:  two special subgraphs have the same ornament if and only if there is a relative quasi-isometry between these special subgraphs that preserves the decoration assigned to standard geodesics labelled by $\delta$ and $\delta'$.

 Using Lemma 3.2 of \cite{papasoglu2002quasi} and the definition of relative quasi-isometry, we see that for each $J\in S_\Gamma$, $\hat f(J)$ is within finite Hausdorff from a unique special subgraph $K\in S_\Lambda$ of $\Sigma_\Lambda$. Since $f$ is decoration preserving, clearly $J$ and $K$ are both decorated with the same ornament. Note that the decoration assigned to each special subgraph is invariant under the group action. Thus the sets of  special subgraphs attached to any two isolated vertices (in either $\Sigma_\Gamma$ and $\Sigma_\Lambda$) have the same sets of ornaments.

We recall each special subgraph $J\in S_\Gamma$ naturally corresponds to a coset $g A(\Gamma_i)$ for some $i$ and $g\in A(\Gamma)$. Suppose $K\in S_\Lambda$ is any special subgraph decorated with same ornament as $J$. As  $A(\Gamma_i)$ is a dovetail RAAG, there is a constant $C_{J,K}$ and a  relative quasi-isometry $f_{J,K}:J\rightarrow K$ such that $\mathrm{str}(f_{J,K},l,C_{J,K})=\lambda(\mathcal{S}_l)$ for every $\mathcal{F}$-geodesic  $l\subseteq g X(\Gamma_i)$ labelled by an element of $B_\Gamma$. We may also assume that $f_{J,K}$ preserves the induced decoration on standard geodesics. Moreover, as there are only finitely many isometry types of special subgraphs in $\Sigma_\Gamma$ and $\Sigma_\Lambda$, we can assume there are only finitely many distinct quasi-isometries $f_{J,K}$. 

We now build a quasi-isometry $g:\Sigma_\Gamma\rightarrow \Sigma_\Lambda$ using moves of type I, II and III in a similar manner to Proposition \ref{prop:relqicyl}. This can be done so that $g$ preserves the decoration on special subgraphs.
Moreover, whenever we perform a move of type II, we use  Lemma \ref{lem:qilink} with the set of quasi-isometries $f_{J,K}$ as defined above. By construction, there is a sufficiently large $C$ such that $\mathrm{str}(g,l,C)=\lambda(V_l)$ for every $V_\mathcal{F}$-geodesic in $X(\Gamma)$. We can `rebase' the free product of graphs by choosing depth one isolated vertices, thus ensuring that  $[g(l)]=[l']$. The resulting $g$ will be the required  relative quasi-isometry.
\end{proof}

Another result along these lines is Proposition \ref{prop:frsysjsj}, which shows that the class of dovetail RAAGs is closed under amalgamating along infinite cyclic subgroups corresponding to standard geodesics.

\section{Constructing the quasi-isometry}\label{sec:constructqi}
We are now in a position to state and prove our main theorem.  Before doing this, we briefly recap what we have already shown.
Suppose that $G$ and $G'$ are one-ended RAAGs. The JSJ trees of cylinders of $G$ and $G'$ are  graph of group decompositions that can visually be read off from their defining graphs, as explained in Section \ref{sec:jsjtoc}.  It follows from Corollary \ref{cor:embeldecgeom} that if $G$ and $G'$ are quasi-isometric, then their JSJ trees of cylinders, decorated with the embellished decoration, are equivalent, i.e.  $(T,\delta_0)$ and $(T',\delta'_0)$ are equivalent. We now prove a partial converse.

\begin{defn}
Given a class $\mathcal{C}$ of RAAGs, let $\mathcal{J}(\mathcal{C})$ denote the class of one-ended RAAGs whose JSJ tree of cylinders have rigid vertex stabilizers in $\mathcal{C}$.
\end{defn}

\begin{thm}\label{thm:mainthm}
Let $\mathcal{D}$ denote the class of dovetail RAAGs. Suppose that $G$ and $G'$ are two  RAAGs in $\mathcal{J}(\mathcal{D})$, and that $(T,\delta_0)$ and $(T',\delta'_0)$ are the JSJ trees of cylinders of $G$ and $G'$ respectively, decorated with the embellished decoration. Then $G$ and $G'$ are quasi-isometric if and only if $(T,\delta_0)$ and $(T',\delta'_0)$ are equivalent.
\end{thm}

We give a brief overview of the proof of Theorem \ref{thm:mainthm}. As discussed above, Corollary \ref{cor:embeldecgeom} ensures that if two  RAAGs in $\mathcal{J}(\mathcal{D})$ are quasi-isometric, then their JSJ trees of cylinders are equivalent when endowed with the embellished decoration; this proves one direction of Theorem \ref{thm:mainthm}. The remainder of this section is devoted to a proof of the converse.

This is done by applying Proposition \ref{prop:qitree}, which says that we can construct a quasi-isometry from $G$ to $G'$ by choosing relative quasi-isometries of vertex spaces  that agree along common edge spaces. We first show that we can always pick ``nice'' quasi-isometries between  rigid vertex spaces, which is done Corollary \ref{cor:admissible}. In particular, we can explicitly see how such quasi-isometries behave when restricted to incident edge spaces. In Lemma \ref{lem:constructcylqi} we define quasi-isometries between cylindrical vertex spaces which agree with quasi-isometries of neighbouring vertex spaces on common edge spaces.  The proof of this Lemma  makes heavy use of the machinery developed in Section \ref{sec:freeproduct}. We then use Corollary \ref{cor:admissible} and Lemma \ref{lem:constructcylqi} to construct suitable quasi-isometries of vertex spaces that satisfy the hypotheses of Proposition \ref{prop:qitree}, and thus construct the required quasi-isometry from $G$ to $G'$.

Throughout this section, we fix RAAGs $G,G' \in \mathcal{J}(\mathcal{D})$ and let $(T,\delta_0)$ and $(T',\delta'_0)$ be as in Theorem \ref{thm:mainthm}. We  assume that $(T,\delta_0)$ and $(T',\delta'_0)$ are equivalent and will show that $G$ and $G'$ are quasi-isometric.
 We perform neighbour and vertex refinement on $(T,\delta_0)$ and $(T',\delta'_0)$, obtaining refinements $(T,\delta)$ and $(T,\delta')$. By Corollary \ref{cor:dectreeisom}, we may assume $\mathrm{im}(\delta)=\mathrm{im}(\delta')$ and so the conclusions of Corollary \ref{cor:dectreeisom} hold. We let $X$ and $X'$ be the trees of spaces associated to $T$ and $T'$ respectively.

\begin{defn}\label{defn:cylstretch}
Suppose $v\in T$ and $w\in T'$ are cylindrical vertices such that $\delta(v)=\delta'(w)$. 
If $v$ (and hence $w$) is an  $\mathcal{R}$-cylinder, we define $\lambda_{v,w}:=\frac{\mathrm{str}(e)}{\mathrm{str}(f)}$, where $e$ and $f$ are $\mathcal{R}$-edges  incident to $v$ and $w$ respectively such that $\delta(e)=\delta'(f)$.
We define $\lambda_{v,w}:=1$ for all remaining cylinders $v\in T$ and $w\in T'$.
\end{defn}

We now aim to construct a quasi-isometry $\phi:G\rightarrow G'$ such that if $v$ is a cylinder corresponding to the standard geodesic  $l$, then $\mathrm{str}(\phi,l)=\lambda_{v,\phi_*(v)}$.
We note that if $v$ is an $\mathcal{R}$-cylinder, $\lambda_{v,w}$ doesn't depend on the choice of edges $e$ and $f$. Indeed, suppose $e'$ and $f'$ are $\mathcal{R}$-edges incident to $v$ and $w$ respectively such that $\delta(e')=\delta'(f')$. Since the decoration incorporates relative stretch factors, we have $\frac{\mathrm{str}(e')}{\mathrm{str}(e)}=\frac{\mathrm{str}(f')}{\mathrm{str}(f)}$. Therefore $\frac{\mathrm{str}(e)}{\mathrm{str}(f)}=\frac{\mathrm{str}(e')}{\mathrm{str}(f')}$ as desired.

\subsection*{Relative quasi-isometries between rigid vertex spaces}
We first explain how to pick suitable relative quasi-isometries of rigid vertex spaces.
Recall every edge space $X_e$ has a cylindrical decomposition $\mathbb{E}^1\times  Y_e$ as in Definition \ref{defn:cyldecomp}. For each edge $e\in ET$, we  pick a standard geodesic $l_e\subseteq X_e$ corresponding to a subset of the form $\mathbb{E}^1\times \{y\}$.

\begin{defn}\label{defn:admissible}
We say a  $(K,A)$-quasi-isometry $$\phi_v:(X_v,\mathcal{P}_v,\delta_v,l_v)\rightarrow (X_w,\mathcal{P}_w,\delta_w,l_w)$$ of rigid vertices is  \emph{admissible} if there exists a $C\geq 0$ such that for every edge $e\in ET$ with $\iota e=v$, we have $\mathrm{str}(\phi_v,l_e,C)=\lambda_{\tau e,\tau((\phi_v)_* e)}$.
\end{defn}

\begin{lem}\label{lem:admisdefine}
Suppose  $e\in ET$ and $f\in ET'$ are chosen such that  $\delta(e)=\delta'(f)$, with both  $\iota e=v$ and $\iota e'=w$  rigid vertices. Then there exists an admissible  quasi-isometry $\phi:(X_v,\mathcal{P}_v,\delta_v,l_e)\rightarrow (X_w,\mathcal{P}_w,\delta_w,l_f)$.
\end{lem}
\begin{proof}
By Corollary \ref{cor:dectreeisom}, we know that there exists a  quasi-isometry $\phi_v:(X_v,\mathcal{P}_v,\delta_v,l_e)\rightarrow (X_w,\mathcal{P}_w,\delta_w,l_f)$. Since $G\in \mathcal{J}(\mathcal{D})$, we know the vertex stabilizer of $v$ is a dovetail RAAG. We can thus choose a  quasi-isometry $\hat \phi_v:(X_v,\mathcal{P}_v,\delta_v,l_e)\rightarrow (X_w,\mathcal{P}_w,\delta_w,l_f)$ and $C\geq 0$ such that for every  $e'\in ET$ with $\iota e'=v$, we have $\mathrm{str}(\hat\phi_v,l_{e'},C)=\lambda_{\tau {e'},\tau((\phi_v)_* e')}$.
Thus $\hat \phi_v$ is admissible. 
\end{proof}

\begin{lem}\label{lem:almostprod}
Suppose $\phi_v:(X_v,\mathcal{P}_v,\delta_v,l_v)\rightarrow (X_w,\mathcal{P}_w,\delta_w,l_w)$ is an admissible $(K,A)$-relative quasi-isometry of rigid vertices. There are constants $K'$, $A'$ and $C$, depending only $K$, $A$, $\Gamma$ and $\Lambda$, such that for each edge $e$ with $\iota e=v$, the map $\phi_v|_{X_e}$ is $C$-close to a map $\phi_e:\mathbb{E}^1\times Y_e\rightarrow \mathbb{E}^1\times Y_{(\phi_v)_*(e)}$  of the form $$(t,y)\mapsto (B_et+\gamma_e(y),\psi_e(y)),$$ where  $\lvert B_e\rvert=\lambda_{\tau e,\tau ((\phi_v)_*e)}$. Moreover,  $\psi_e$ is a  $(K',A')$-quasi-isometry  and $\gamma_e$ is a $(K',A')$-coarse Lipschitz map. 
\end{lem}

\begin{proof}
Recall Remark \ref{rem:peripheralstrucutreraags}, which tells us that if  $e$ is an edge incident to $v$, then the corresponding element $\mathcal{S}_e\in  \mathcal{P}_v$ consists precisely of subspaces of the form $\mathbb{E}^1\times \{y\}$, where $y$ is a vertex in $Y_e$. 
As $\phi_v$ is a $(K,A)$-relative quasi-isometry, it  sends subsets of the form $\mathbb{E}^1\times \{y\}$ to within Hausdorff distance $A$ of subsets of the form $\mathbb{E}^1\times \{y'\} \subseteq X_{(\phi_v)_*(e)}$. Therefore $\phi_v|_{X_e}$ is $A$-close to a map of the form $(t,y)\mapsto (\omega_e(t,y),\psi_e(y))$.

We now fix $y\in Y_e$. As $\phi_v$ is admissible, $t\mapsto \omega_e(t,y)$ is $C'=C'(K,A,\Gamma,\Lambda)$-close to a map of the form $t\mapsto (Bt+\gamma_e(y),\psi_e(y))$, where $\lvert B\rvert=\lambda_{\tau e,\tau ((\phi_v)_*e)}$. Note $\lvert B\rvert$ is independent of the choice of $y$.  Since $\lvert B\rvert>0$ and $\phi_v$ is a quasi-isometry, the sign of $B$ is also independent of $y$. Thus $\phi_v|_{X_e}$ is $C'$-close to a map $\phi_e$ of the form $$(t,y)\mapsto (Bt+\gamma_e(y),\psi_e(y)).$$

Finally, we claim that $\psi_e$ is a quasi-isometry and $\gamma_e$ is coarse Lipschitz, both of whose constants depend only on $K$ and $A$. Since $\phi_v$ is a $(K,A)$-quasi-isometry, the map $\phi_e$  is a $(K,A+2C')$ quasi-isometry. For all $y,y'\in Y_e$ and $t\in \mathbb{E}^1$, we note that $$d(\psi_e(y),\psi_e(y'))\leq d(\phi_e(t,y),\phi_e(t,y'))\leq K d(y,y')+A+2C'.$$ Similarly, we see $\lvert\gamma_e(y)-\gamma_e(y')\rvert\leq K d(y,y')+A+2C'$.
We now pick $t'\in \mathbb{E}^1$ such that $Bt+\gamma_e(y)=Bt'+\gamma_e(y')$. This means that \begin{align*}
d(\psi_e(y),\psi_e(y'))=d(\phi_e(t,y),\phi_e(t',y'))\geq \frac{1}{K}d(y,y')-A-2C'.
\end{align*}
\end{proof}

\begin{rem}\label{rem:flip}
For a RAAG $A(\Gamma)$, there is an automorphism defined by $w\mapsto w^{-1}$ for any $w\in V\Gamma$, whilst keeping all other generators fixed. By conjugating such an automorphism, we can perform an orientation-reversing `flip' along any standard geodesic in $X(\Gamma)$.
We can also perform translations along a standard geodesic in $X(\Lambda)$. Indeed, suppose $l$ is a standard geodesic whose vertex set is $g\langle w\rangle$. Then multiplication by $gw^dg^{-1}$ corresponds to  translating $l$ a distance $d$.

\end{rem}

\begin{cor}\label{cor:admissible}
Suppose $v\in VT$ and $w\in VT'$ are rigid vertices, and there are edges $e$ and $f$ such that  $\iota e=v$, $\iota f=w$ and $\delta(e)=\delta'(f)$. Let $y_0\in Y_e$ and $y'_0\in Y_f$.
Then for any $A\in \mathbb{\mathbb{E}}^1$ and $\varepsilon\in \{1,-1\}$  we may choose an admissible quasi-isometry $\phi_v:(X_v,\mathcal{P}_v,\delta_v,l_v)\rightarrow (X_w,\mathcal{P}_w,\delta_w,l_w)$ such that $(\phi_v)_*(e)=f$  and the map $\phi_e$, as defined in Lemma \ref{lem:almostprod}, is of the form $$(t,y)\mapsto (B_et+\gamma_e(y),\psi_e(y)),$$
 with  $B_e=\varepsilon \hspace{1pt}\lambda_{\tau e,\tau ((\phi_v)_*e)}$, $\psi_e(y_0)=y'_0$ and $d(\gamma_e(y_0),A)\leq 1$.
\end{cor}
\begin{proof}
Using Lemma \ref{lem:admisdefine}, we construct an admissible quasi-isometry $$\phi_v:(X_v,\mathcal{P}_v,\delta_v,l_v)\rightarrow (X_w,\mathcal{P}_w,\delta_w,l_w)$$ such that $(\phi_v)_*(e)=f$.  
We define $\psi_e$ and $\gamma_e$ as in Lemma \ref{lem:almostprod}.
There exists $g\in \mathrm{stab}(f)\subseteq A(\Lambda)$ such that $g(\mathbb{E}^1\times \{\psi_e(y_0)\})=\mathbb{E}^1\times \{y'_0\}$. We therefore modify $\phi_v$ by postcomposing with an isometry corresponding to left multiplication by $g$.
Using the operations described in Remark \ref{rem:flip}, we can also postcompose by an isometry so that $B_e$ has the correct sign and $d(\gamma_e(y_0),A)\leq 1$. These isometries preserve $A(\Lambda)$-orbits of standard geodesics, so preserve decoration. Thus the modified quasi-isometry is still admissible.
\end{proof}

\subsection*{Relative quasi-isometries between cylindrical vertex spaces}
We make use of the following Lemma to construct suitable relative quasi-isometries of cylinders.

\begin{lem}\label{lem:constructqicyl}
Let $X$ and $Y$ be geodesic metric spaces and suppose $\gamma:X\rightarrow \mathbb{E}^1$ is $(K,A)$-coarse Lipschitz and $\psi:X\rightarrow Y$ is a $(K,A)$-quasi-isometry. Let $B\in \mathbb{R}\backslash\{0\}$. Then the map $\phi:\mathbb{E}^1\times X\rightarrow \mathbb{E}^1\times Y$ defined by $$(t,x)\mapsto (Bt+\gamma(x),\psi(x))$$ is a $(K',A')$-quasi-isometry for some $K'=K'(K,A,B)$ and $A'=A'(K,A,B)$.
\end{lem}

\begin{proof}
Let $t,t'\in \mathbb{E}^1$ and $x,x'\in X$. Then \begin{align*}d(\phi(t,x),\phi(t',x'))&\leq d(\phi(t,x),\phi(t,x'))+d(\phi(t,x'),\phi(t',x'))\\
& \leq \sqrt{[\gamma(x)-\gamma(x')]^2+d_Y(\psi(x),\psi(x'))^2}+B\lvert t-t'\rvert \\
&\leq  \sqrt{2}(Kd_X(x,x')+A)+B\lvert t-t'\rvert\\
&\leq (\sqrt{2}K+B)d((t,x),(t',x'))+\sqrt{2} A.
\end{align*}
To give a lower bound for $d(\phi(t,x),\phi(t',x'))$, we first deal with the case when $\lvert \gamma(x)-\gamma(x')\rvert\leq \frac{B}{2}\lvert t-t'\rvert$. Then
\begin{align*}
d(\phi(t,x),\phi(t',x'))
&\geq\sqrt{\Bigg(\frac{B}{2}\lvert t-t'\rvert\Bigg)^2+d_Y(\psi(x),\psi(x'))^2}\\
&\geq \frac{B}{4}\lvert t-t'\rvert+\frac{1}{2}d_Y(\psi(x),\psi(x'))\\
& \geq \frac{B}{4}\lvert t-t'\rvert+\frac{1}{2K}d_X(x,x')-\frac{A}{2}\\
&\geq C d((t,x),(t',x'))-\frac{A}{2}
\end{align*}
where $C:=\min(\frac{B}{4},\frac{1}{2K})$ and the last line follows from the triangle inequality. We now conclude with the case where $\frac{B}{2}\lvert t-t'\rvert\leq \lvert \gamma(x)-\gamma(x')\rvert\leq Kd_X(x,x')+A$.
Then
\begin{align*}
d(\phi(t,x),\phi(t',x'))
&\geq d_Y(\psi(x),\psi(x'))\\
&\geq \frac{1}{2K}d_X(x,x')+\frac{1}{2K}d_X(x,x')-A\\
& \geq \frac{1}{2K}d_X(x,x')+\frac{B}{4K^2}\lvert t-t'\rvert-\frac{A}{2K^2}-A\\
& \geq C' d((t,x),(t',x'))-\frac{A}{2K^2}-A
\end{align*}
where $C':=\min(\frac{B}{4K^2},\frac{1}{2K})$.
\end{proof}

The following lemma is the main tool required in the proof of Theorem \ref{thm:mainthm}. The key idea is to apply the machinery of Section \ref{sec:freeproduct} to encode the required quasi-isometries of adjacent edge spaces into the quasi-isometry of cylinders.

\begin{lem}\label{lem:constructcylqi}
There are constants $K\geq 1$ and $A\geq 0$ such that for any cylinders  $v\in VT$ and $w\in VT'$ such that $\delta(v)=\delta'(w)$, the following hold.
\begin{enumerate}
\item \label{lem:constructcylqirelqicyl} There is a $(K,A)$ relative quasi-isometry $\phi_v:X_v\rightarrow X'_{w}$. 
\item \label{lem:constructcylqirelqiedge} For every edge $e\in ET$ with $\iota e=v$, there is a map $\phi_e:X_e\rightarrow X'_{(\phi_v)_*(e)}$ such that $\phi_e$ and $\phi_v|_{X_e}$ are $A$-close. 
\item \label{lem:constructcylqirelqirig} For every edge $e\in ET$ with $\iota e=v$, there is an admissible $(K,A)$-quasi-isometry $$\phi_{\tau e}:(X_{\tau e},\mathcal{P}_{\tau e},\delta_{\tau e},l_{\overline e})\rightarrow (X_{\tau((\phi_v)_* e)},\mathcal{P}_{\tau((\phi_v)_* e)},\delta_{\tau((\phi_v)_* e)},l_{\overline{(\phi_v)_* e}})$$ such that the maps $\alpha_{(\phi_v)_*e}\circ \phi_e$ and $\phi_{\tau e} \circ \alpha_e$ are $A$-close. 
\end{enumerate}
Moreover, suppose that there exist edges $e\in ET$ and $f\in ET'$ such that $\iota e=v$, $\iota f=w$ and $\delta(e)=\delta'(f)$. Suppose also that there is an admissible $(K,A)$ relative quasi-isometry $\rho:X_{\tau e}\rightarrow X'_{\tau f}$ such that $\rho_*(e)=f$. Then \ref{lem:constructcylqirelqicyl}-\ref{lem:constructcylqirelqirig} hold  with the restriction $(\phi_v)_*(e)=f$ and $\phi_{\tau e}=\rho$.   
\end{lem}

\begin{proof}
Every cylindrical vertex space $X_v$ admits a cylindrical decomposition $X_v=\mathbb{E}^1\times Y_v$ as in Definition \ref{defn:cyldecomp}. As in Section \ref{sec:cyl}, we replace the trees of spaces with blown up trees of spaces. We thus assume that each vertex space $X_v$ is of the form $\mathbb{E}^1\times \Sigma_v$, where $\Sigma_v$ is a free product of graphs. An edge space  $X_e\subseteq X_v$ is of the form $\mathbb{E}^1\times A_e$, where  $A_e$ is some peripheral special subgraph of $\Sigma_v$. 
The same is true for the cylinder $w$. Let $v_0$ and $w_0$ be depth one isolated vertices of $\Gamma_v$ and $\Gamma_w$ respectively. Let $B:=\lambda_{v,w}$. For ease of notation, we will freely identify the edge space $X_e\subseteq X_{\iota e}$ with $\alpha_e(X_e)\subseteq X_{\tau e}$; omitting the $\alpha_e$ makes the following argument easier to follow. 

We will define maps $\psi: \Sigma_v\rightarrow \Sigma_w$ and $\gamma: \Sigma_v\rightarrow \mathbb{R}$ so that the hypotheses of Lemma \ref{lem:constructqicyl} hold. It will be clear from the construction that the resulting quasi-isometry between vertex spaces does satisfy the necessary conditions.  The maps $\psi$ and $\gamma$ are defined inductively on special subgraphs, working out from depth one isolated vertices.

The construction of $\psi$ is the most involved part of the proof. It is a modification of the  construction of the quasi-isometry between free products of graphs in the proof of Proposition \ref{prop:relqicyl}. The reader is advised to understand this simpler construction first before proceeding with the current proof.

We  decorate special subgraphs of  $\Sigma_v$ and $\Sigma_w$ with the following data:
\begin{enumerate}
\renewcommand{\labelenumi}{(\roman{enumi})}
\item the type of each special subgraph: peripheral or non-peripheral;
\item the quasi-isometry class of each special subgraph;
\item for each peripheral special subgraph, the ornament assigned to the edge  corresponding to that special subgraph. 
\end{enumerate} 

Using Corollary \ref{cor:dectreeisom}, there is a relative quasi-isometry $\rho_v:(X_v,\mathcal{P}_v)\rightarrow (X'_w,\mathcal{P}_w)$ that preserves the decoration on incident edge spaces. We thus deduce, extending the proof of Proposition \ref{prop:relqicyl}, that special subgraphs in $\mathrm{lk}^+(v_0)$ and $\mathrm{lk}^+(w_0)$ that are either peripheral   or are one-ended and  non-peripheral, are decorated with the same sets of ornaments. This is in general stronger than just assuming they have the same quasi-isometry class.

 We set $\psi(v_0)=w_0$ and $\gamma(v_0):=t_0$, where $t_0\in \mathbb{E}^1$ is chosen arbitrarily.
As in the proof of Proposition \ref{prop:relqicyl}, we apply type I and III moves at $v_0$ to ensure  that there is a decoration-preserving bijection $\chi:\mathrm{lk}^+(v_0)\rightarrow \mathrm{lk}^+(w_0)$.
Let $A\in \mathrm{lk}^+(v_0)$ and $\chi(A)\in \mathrm{lk}^+(w_0)$ denote peripheral special subgraphs, and suppose $e_A$ and $e'_A$ denote the corresponding edges in  $\mathrm{lk}(v)$ and $\mathrm{lk}(w)$ respectively. We orient these edges so that $\iota e_A=v$ and $\iota e_A'=w$. Let $a_0\in A$ and $a'_0\in \chi(A)$ be basepoints. 

By  Corollary \ref{cor:admissible}, we may choose an admissible relative quasi-isometry $\phi_{\tau e_A}:X_{\tau e_A}\rightarrow X'_{\tau e'_A}$ such that $(\phi_{\tau e_A})_*(\overline{e_A})=\overline{e'_A}$, and such that the map $\phi_{\bar e_A}:X_{\overline{e_A}}\rightarrow X_{\overline{e'_A}}$, as in Lemma \ref{lem:almostprod}, is of the form $$(t,y)\mapsto (Bt+\gamma_{e_A}(y),\psi_{e_A}(y)).$$

We then use $\psi_{e_A}$ and $\gamma_{e_A}$ to extend $\psi$ and $\gamma$ over the special subgraph $A$. Using Remark \ref{rem:flip}, there is enough freedom in the application of Lemma \ref{lem:almostprod} that we can assume $\psi_{e_A}$ and $\gamma_{e_A}$ coarsely uniformly agree with $\psi$ and $\gamma$ at the basepoint. We continue in this way, extending the domain of $\psi$ and $\gamma$ over special subgraphs.

We can be more precise in how we `extend over a special subgraph' to define $\psi$; this requires a bit of care. What this actually involves is applying Lemma \ref{lem:qilink} using the quasi-isometry $\psi_{e_A}$ as above. The composite quasi-isometry obtained by applying several type I, II and III moves at all isolated vertices is therefore the desired quasi-isometry $\psi:\Sigma_v\rightarrow \Sigma_w$. This quasi-isometry has all the maps $\psi_{e_A}$ (for all special subgraphs $A$) built into it. Similarly we construct the coarse Lipschitz map $\gamma: \Sigma_v\rightarrow \mathbb{R}$, with all the maps $\gamma_{e_A}$ built in. It follows that the resulting quasi-isometry $X_v\rightarrow X_w$, constructed via Lemma \ref{lem:constructqicyl}, necessarily agrees with all the maps $\phi_{\tau e_A}$ on adjacent edge spaces. 

We now prove the final part of the proposition: how to modify this construction if one already has $e$, $f$ and $\rho$. By Lemma \ref{lem:almostprod}, the map $\rho|_{X_{\overline e}}$ is close to a map of the form $$(t,y)\mapsto (Bt+\gamma(y),\psi(y)),$$ where $B=\epsilon\mathrm{str}(v,w)$ for $\epsilon=\pm 1$.

 Suppose $a_0$ is the basepoint of $A_e$. We choose isolated vertices $x\in \Sigma_v$ and $y\in \Sigma_w$ so that $A_e$ and $A'_f$ are attached to $x$ and $y$ at $a_0$ and $\psi(a_0)$ respectively.  We now `rebase' $\Sigma_v$ and $\Sigma_w$ so that $x$ and $y$ are the depth one isolated vertices. We now proceed as above, setting $t:=\gamma(a_0)$, using $\epsilon$ as above when applying Corollary \ref{cor:admissible}, and ensuring that $\chi(A)=A'$ and $\phi_{\tau e}=\rho$.
\end{proof}

\subsection*{Combining relative quasi-isometries}
\begin{proof}[Proof of Theorem \ref{thm:mainthm}]
Theorem \ref{thm:jsjtocqi} and  Corollary \ref{cor:embeldecgeom}  ensure that if $G$ and $G'$ are quasi-isometric, then they have equivalent JSJ trees of cylinders when decorated with the embellished decoration. We prove the converse.

We suppose the trees  $T$ and $T'$ are equivalent. We construct an isomorphism $\chi:T\rightarrow T'$ and maps $\{\phi_v\}$ and $\{\phi_e\}$ such that Proposition \ref{prop:qitree} holds. We choose arbitrary cylindrical vertices  $v\in VT$ and $w\in VT'$ decorated with the same ornament. We construct a relative quasi-isometry $\phi_v:X_v\rightarrow X'_w$ as in Lemma \ref{lem:constructcylqi} and define $\chi|_{\mathrm{lk}(v)}:=(\phi_v)_*$. 	
For each vertex $v'$ adjacent to $v$, Lemma \ref{lem:constructcylqi} gives a quasi-isometry $\phi_{v'}:X_{v'}\rightarrow X_{\chi(v')}$. We extend $\chi$ to $\mathrm{lk}(v')$ by $(\phi_{v'})_*$.
 The quasi-isometries $\{\phi_e\}$ and  $\{\phi_v\}$, where defined thus far,  do indeed satisfy the conditions required  for Proposition \ref{prop:qitree} to be applied.

We work outwards from $v$, applying Lemma \ref{lem:constructcylqi} at each cylindrical vertex, extending the domain of $\chi$ and defining new relative quasi-isometries  as we move away from $v$. For instance, suppose $v_1$ is a rigid vertex adjacent to $v$ and $v_2\neq v$ is a cylindrical vertex adjacent to $v_1$. 
 Let $e=(v_2,v_1)$. We use Lemma \ref{lem:constructcylqi} to construct a relative quasi-isometry $\phi_{v_2}:X_{v_2}\rightarrow X'_{\chi(v_2)}$. 
 By the final part of  Lemma \ref{lem:constructcylqi}, $\phi_{v_2}$  can be chosen so that $(\phi_{v_2})_*(e)=\chi(e)$, and that $\phi_{v_2}$ and $\phi_{v_1}$ agree on the edge space $X_e$. We continue in this way, defining more maps $\phi_v$ and $\phi_e$ and extending the domain of $\chi$.
 
We conclude with a  uniformity argument: as there are only finitely many isometry classes of vertex and edge spaces in $T$ and $T'$,  all maps $\phi_v$ and $\phi_e$ can be chosen from a finite set of model maps.  We thus have uniform quasi-isometry constants for all maps $\{\phi_v\}$ and $\{\phi_e\}$.  Thus Proposition \ref{prop:qitree} can be applied, showing that $G$ and $G'$ are indeed quasi-isometric.
\end{proof}

Using the proof of Theorem \ref{thm:mainthm}, we can provide further evidence for Question \ref{ques:dovetail}. In particular, the following proposition shows that tree RAAGs are dovetail.
\begin{prop}\label{prop:frsysjsj}
Let $A(\Gamma_1)$ and $A(\Gamma_2)$ be one-ended dovetail RAAGs. Let $\Gamma$ be the simplicial graph obtained by attaching $\Gamma_1$ and $\Gamma_2$ along vertices $v\in V\Gamma_1$ and $w\in V\Gamma_2$. Then $A(\Gamma)=A(\Gamma_1)*_{\langle v\rangle=\langle w\rangle}A(\Gamma_2)$ is a dovetail RAAG.
\end{prop}
\begin{proof}
We suppose $f:(X(\Gamma),\mathcal{P}_{B},\delta,l)\rightarrow( X(\Lambda),\mathcal{P}_{B_\Lambda},\delta',l')$ is a quasi-isometry so that $X(\Gamma)$ and $X(\Lambda)$ admit a decoration-preserving isomorphism between $T$ and $T'$, their JSJ tree of cylinders. Let $\lambda:\mathcal{P}_{B}\rightarrow \mathbb{Q}_{>0}$ be a function which is constant on $\delta^{-1}(o)$ for every $o\in \mathcal{O}$. 

We proceed as in the proof of Theorem \ref{thm:mainthm}, building the required quasi-isometry $g:(X(\Gamma),\mathcal{P}_{B},\delta,l)\rightarrow( X(\Lambda),\mathcal{P}_{B_\Lambda},\delta',l')$ by defining a tree of quasi-isometries. We can do this by using the dovetail property to encode the required stretch factors into the quasi-isometries of vertex spaces.  We modify the cylinder stretch factors $\lambda_{v,w}$, defined  in Definition \ref{defn:cylstretch}, as follows. Suppose there is some  $\mathcal{F}$-geodesic $l$, labelled by an element of $B_\Lambda$, whose parallelism class isn't an $\mathcal{R}$-cyclinder. Then for some $w\in VT'$ with $\delta(v)=\delta'(w)$, we define $\lambda_{v,w}:=\lambda(S_l)$. We may now proceed as in the proof of Theorem \ref{thm:mainthm}.
\end{proof}
\begin{rem}\label{rem:ifunivfrfails}
Should Question \ref{ques:dovetail} fail to hold, the methods of Sections \ref{sec:stretch} and \ref{sec:constructqi} can still be adapted to determine necessary and sufficient criteria for decorated JSJ tree of cylinders to determine the quasi-isometry type of a RAAG. To do this, one needs a sufficient understanding of which possible stretch factors can arise and to what extent one can construct quasi-isometries by assigning such stretch factors to geodesics independently. 
\end{rem}

\section{The quasi-isometry problem}\label{sec:algor}
We now wish to investigate whether one can algorithmically decide if two RAAGs are quasi-isometric. 
Let $\mathcal{D}$ be an arbitrary class of RAAGs.
\begin{defn}
 We say that $\mathcal{D}$ has \emph{solvable membership problem} if there is an algorithm that takes as input a finite simplicial graph $\Gamma$ and outputs whether $A(\Gamma)\in \mathcal{D}$.
\end{defn}

\begin{defn}
 We say that $\mathcal{D}$  has \emph{solvable quasi-isometry problem} if there is an algorithm that takes as input finite simplicial graphs $\Gamma$ and $\Lambda$ with $A(\Gamma),A(\Lambda)\in \mathcal{D}$, and outputs whether $A(\Gamma)$ and $A(\Lambda)$ are quasi-isometric.
\end{defn}

Suppose we are given the tuple $(\Gamma,B_\Gamma,\hat\delta_\Gamma,v)$, where $\Gamma$ is a finite simplicial graph, $B_\Gamma\subseteq V\Gamma$, $\hat\delta$ is a function $B_\Gamma\rightarrow \mathcal{O}$ and $v\in B_\Gamma$. We associate to  this data the tuple $(X(\Gamma),\mathcal{P}_{B_\Gamma},\delta_\Gamma,l_v)$, where $\delta_\Gamma(l)=\hat \delta_\Gamma(w)$ for every $w$-geodesic $l$, and $l_v$ is the standard geodesic corresponding to the subgroup $\langle v\rangle$.

\begin{defn}
We say that $\mathcal{D}$ has \emph{solvable relative quasi-isometry problem} if there is an algorithm that takes as input $(\Gamma,B_\Gamma)$ and $(\Lambda,B_\Lambda)$ such that $A(\Gamma),A(\Lambda)\in \mathcal{D}$,  and outputs whether there is a decoration preserving based relative quasi-isometry $(X(\Gamma),\mathcal{P}_{B_ \Gamma})\rightarrow (X(\Lambda),\mathcal{P}_{B_ \Lambda})$.

We say that $\mathcal{D}$ has \emph{solvable strong quasi-isometry problem} if there is an algorithm that takes as input $(\Gamma,B_\Gamma,\hat\delta_\Gamma,v)$ and $(\Lambda,B_\Lambda,\hat\delta_\Lambda,w)$ such that $A(\Gamma),A(\Lambda)\in \mathcal{D}$,  and outputs whether there is a decoration preserving based relative quasi-isometry $(X(\Gamma),\mathcal{P}_{B_ \Gamma},\delta_\Gamma,l_v)\rightarrow (X(\Lambda),\mathcal{P}_{B_ \Lambda},\delta_\Lambda,l_w)$.
\end{defn}

\begin{prop}
Abelian RAAGs have solvable strong relative quasi-isometry problem.
\end{prop}

\begin{proof}
Suppose we are given the data $(\Gamma,B_\Gamma,\hat\delta_\Gamma,v)$ and $(\Lambda,B_\Lambda,\hat\delta_\Lambda,w)$, where $\Gamma$ and $\Lambda$ are complete graphs on $n$ and $m$ vertices respectively. Suppose that $f:(X(\Gamma),\mathcal{P}_{B_ \Gamma},\delta_\Gamma,l_v)\rightarrow (X(\Lambda),\mathcal{P}_{B_ \Lambda},\delta_\Lambda,l_w)$ is a decoration-preserving based relative quasi-isometry. Then $n=m$, and there exists a decoration-preserving bijection $r_*:B_\Gamma\rightarrow B_\Lambda$ such that $r_*(v)=w$.

Conversely, if such an $r_*$ exists, it can be arbitrarily extended to a decoration preserving graph isomorphism $\hat r_*:\Gamma\rightarrow \Lambda$, which then induces the required decoration-preserving based relative quasi-isometry. Thus there exists a suitable relative quasi-isometry if and only if there exists a bijection $r_*:B_\Gamma\rightarrow B_\Lambda$ as above. Since $B_\Gamma$ and $B_\Lambda$ are finite, it is easy to verify whether or not such an $r_*$ exists. 
\end{proof}

I would like to thank Jingyin Huang for explaining how to prove the following.
\begin{prop}\label{prop:wt1strongrelqi}
RAAGs with finite outer automorphism group have solvable strong relative quasi-isometry problem.
\end{prop}

\begin{proof}
Suppose that $A(\Gamma)$ and $A(\Lambda)$ have finite outer automorphism group and we are given $(\Gamma,B_\Gamma,\hat\delta_\Gamma,v)$ and $(\Lambda,B_\Lambda,\hat\delta_\Lambda,w)$ as above.
Given a vertex $x\in X(\Gamma)$ one can consider the set $L_x$ of standard  geodesics passing through $x$, which we identify with vertices of $\Gamma$. We define a graph $\Gamma_x$ with vertex set $L_x$ and $l,l'$ are joined by an edge in $\Gamma_x$ if and only if they span a standard 2-flat in $X(\Gamma)$. By construction, there is a graph isomorphism $\iota_x:\Gamma\rightarrow \Gamma_x$ such that $\iota_x(v)$ is the unique $v$-geodesic in $L_x$.

Suppose  there is a  quasi-isometry $f:(X(\Gamma),\mathcal{P}_{B_ \Gamma},\delta_\Gamma,l_v)\rightarrow (X(\Lambda),\mathcal{P}_{B_ \Lambda},\delta_\Lambda,l_w)$. 
The proof of Corollary 4.11 in \cite{huang2014quasi} then  implies that for a vertex  $x\in X(\Gamma)$, there exists a vertex $y\in X(\Lambda)$ such that $f$ induces a graph isomorphism $f_*:\Gamma_x\rightarrow \Lambda_y$, where $[f_*(l)]=[f(l)]$ for all $l\in \Gamma_x^{(0)}$. Thus there is a graph  isomorphism $r_*:=\iota_y^{-1}\circ f_*\circ\iota_x:\Gamma\rightarrow \Lambda$ such that $r_*(B_ \Gamma)=B_\Lambda$, $r_*(v)=w$ and $r_*$ is decoration preserving.

Conversely, the existence of such an $r_*$ implies the existence of a suitable quasi-isometry (which is in this case a group automorphism).
Thus there exists a suitable relative quasi-isometry if and only if there exists a graph isomorphism $r_*:\Gamma\rightarrow \Lambda$ satisfying the above properties. Since $\Gamma$ and $\Lambda$ are finite graphs, it is easy to verify whether or not such an $r_*$ exists.
\end{proof}

We will use the concept of a generalised star extension (GSE) from Section 6.3 of \cite{huang2014quasi}.  We will not define it here, but note that doubling a pentagon along a closed vertex star as in Example \ref{exmp:stretch} is an example of a GSE. GSEs may be used to construct finite index subgroups of RAAGs. We remark that GSEs increase the number of vertices of a graph. 

\begin{lem}\label{lem:gses}
There exists an algorithm that, when given as input a finite simplicial graph $\Gamma$, determines  whether or not $A(\Gamma)$ is quasi-isometric to a RAAG with finite outer automorphism group. If it is, the algorithm also outputs a graph $\Gamma'$ such that $A(\Gamma')$ has finite outer automorphism group, and a sequence of finitely many GSEs, starting with $\Gamma'$ and ending with $\Gamma$. 
\end{lem}
\begin{proof}
Suppose $\Gamma$ has $n$ vertices. We list all graphs $\Gamma'$ that have at most $n$ vertices such that $A(\Gamma')$ has finite outer automorphism group, and then take all possible sequences of GSEs of these graphs that have at most $n$ vertices. If we do obtain $\Gamma$ by taking GSEs of a RAAG of such a $\Gamma'$, then we are done. If not, then Theorem 1.5 of \cite{huang2014quasi} ensures that $A(\Gamma)$ is not quasi-isometric to a RAAG with finite outer automorphism group.
\end{proof}

We recall that a RAAG with finite outer automorphism group is of type II with trivial centre. Thus every standard geodesic in such RAAG is an $\mathcal{R}$-geodesic. We also recall that in order to define stretch factors, we needed to choose a set of representatives of pairs $(X(\Gamma_Q),l_Q)$ for every based quasi-isometry class $Q$. If such a based quasi-isometry class $Q$ contains a RAAG with finite outer automorphism group, we choose a representative $(X(\Gamma_Q),l_Q)$ such that $A(\Gamma_Q)$ does indeed have finite outer automorphism group.

If $\Gamma$ and $\Gamma'$ are as in Lemma \ref{lem:gses}, then by knowing how $\Gamma$ is obtained from $\Gamma'$ by GSEs, we can obtain the stretch factor. This is exactly the same principle as Example \ref{exmp:stretch}, where we can work out the stretch factor using the fact  that $A(\Lambda')$ is an index two subgroup of $A(P)$. We thus deduce the following:
\begin{cor}\label{cor:algorstretch}
There is an algorithm that takes as input a graph $\Gamma$ and a vertex $v\in V\Gamma$, such that $A(\Gamma)$ is quasi-isometric to a RAAG of finite outer automorphism group, and outputs   the stretch factor associated to a $v$-geodesic.
\end{cor}

\begin{prop}\label{prop:solvqiprob}
We fix classes of RAAGs $\mathcal{C}_1$ and $\mathcal{C}_2$ that have solvable strong quasi-isometry problem and solvable quasi-isometry problem respectively. We consider the subclass $\mathcal{C}_0$ of RAAGs  that satisfy the following. 
\begin{enumerate}
\item For each $A(\Gamma)\in \mathcal{C}_0$, all rigid vertex stabilizers of the JSJ tree of cylinders decomposition $C(\Gamma)$ lie in $\mathcal{C}_1$.
\item For each $A(\Gamma)\in \mathcal{C}_0$, all non-peripheral factors of cylinder stabilizers of $C(\Gamma)$ lie in $\mathcal{C}_2$.
\end{enumerate}
Then there exists an algorithm that takes as input $A(\Gamma), A(\Lambda)\in \mathcal{C}_0$ and outputs whether the trees of cylinders $(T(\Gamma),\delta)$ and $(T(\Lambda),\delta')$, equipped with the na{\"\i}ve decoration, are equivalent.
\end{prop}

\begin{proof}
We first show that cylinder stabilizers of RAAGs in $\mathcal{C}_0$ have solvable quasi-isometry problem. Indeed, by applying the  solvable strong quasi-isometry problem to adjacent rigid vertex stabilizers, we are able to solve the quasi-isometry problem for peripheral factors of cylinder stabilizers of $C(\Gamma)$.  By assumption, we are also able to solve the quasi-isometry problem for non-peripheral factors of cylinder stabilizers. In light of Proposition \ref{prop:relqicyl}, we thus have an algorithm to decide whether cylinder stabilizers are relatively quasi-isometric. In fact, a slight modification of the proof of Proposition \ref{prop:relqicyl} allows us to solve the strong quasi-isometry problem for cylinder stabilizers of RAAGs.

We now decorate vertex and edges in $C(\Gamma)$ with the na{\"\i}ve decoration. More precisely, we decorate each vertex and edge with a symbol from the set $\{o_1,o_2,\dots\}$ such that two vertices share the symbol if and only if they have the same ornament when assigned the na{\"\i}ve decoration. This can be done using Corollary \ref{cor:algorstretch} and the fact that all cylindrical (resp. rigid) vertex stabilizers have solvable relative quasi-isometry problem.

We now perform neighbour refinement, which can be done by simply counting adjacency data  of ornaments when lifted to the Bass-Serre tree. We pick unused letters from the set $\{o_1,o_2,\dots\}$ when refinement occurs. We may also perform vertex refinement, using the fact that the class of all cylindrical and rigid vertex stabilizers have solvable strong relative quasi-isometry problem. We apply this procedure till it stabilizes, giving a decoration $\delta$.

We apply the same procedure to $C(\Lambda)$, using a different set of ornaments $\{o'_1,\dots\}$, and obtaining a decoration $\delta'$. By examining all possible bijections $\mathrm{im}(\delta)\rightarrow \mathrm{im}(\delta')$ and seeing if the conditions of Proposition \ref{prop:equivbijec} hold, we are able to determine whether the decorated trees of cylinders are equivalent.
\end{proof}

\algor
\begin{proof}
Given a finite simplicial graph $\Lambda$, we first check if it is connected; if it is not, then $A(\Lambda)$ cannot be quasi-isometric to $A(\Gamma)$. If it is connected, we can compute the graph of groups $C(\Gamma)$ and by Lemma \ref{lem:gses}, we can  verify whether or not every rigid vertex group is either abelian or  a GSE of a RAAG with finite outer automorphism group. If this is not the case, then $A(\Lambda)$ cannot be quasi-isometric to $A(\Gamma)$. If it is, then both $\Lambda$ and $\Gamma$ satisfy the hypotheses of Proposition \ref{prop:solvqiprob}. 

We follow the proof of Proposition \ref{prop:solvqiprob}, except we can use Corollary \ref{cor:algorstretch} to decorate edges in $C(\Lambda)$ and $C(\Gamma)$ with stretch factors. We use the same argument as in the proof of Proposition \ref{prop:solvqiprob}, and are thus able to decide if $A(\Gamma)$ and $A(\Lambda)$ have equivalent JSJ trees of cylinders when decorated with the embellished decoration. By Theorem \ref{thm:mainthm}, this then determines whether or not $A(\Gamma)$ and $A(\Lambda)$  are quasi-isometric.
\end{proof}

\section{\texorpdfstring{$n$}{n}-clique tree-graded RAAGs}\label{sec:ncliquetreeraags}
We give a sample application of Theorem \ref{thm:mainthm}. We recall the definition of $n$-clique tree-graded RAAGs from Section \ref{sec:intro}.
\ncliquetreeraags

\begin{proof}
It is a well known application of Gromov's polynomial growth theorem that for any $k$, groups that are virtually $\mathbb{Z}^k$ are quasi-isometrically rigid \cite{gromov1981goups}. We thus deduce that any RAAG quasi-isometric to $\mathbb{Z}^k$ has defining graph a complete graph on $k$ vertices.  

Let $\Gamma$ be an $n$-clique tree-graded graph. We analyse cylindrical and rigid vertex groups in $C(\Gamma)$. Cylindrical vertex groups are of the form $$\mathbb{Z}\times (\mathbb{Z}^{n-1}*\dots * \mathbb{Z}^{n-1}*\mathbb{Z}*\dots * \mathbb{Z})$$ where the $\mathbb{Z}^{n-1}$ are peripheral factors and the $\mathbb{Z}$ are non-peripheral factors. In particular, there are no one-ended non-peripheral factors. Using Proposition \ref{prop:relqicyl}, we deduce that $C(\Gamma)$ has a single relative quasi-isometry class of cylindrical vertices.

The rigid vertex groups are all of the form $A(\Gamma')\cong \mathbb{Z}^n$. The peripheral structure consists of $n$ elements, each containing all the standard geodesics in $X(\Gamma')$ labelled by some fixed $v\in V\Gamma'$. Thus there is a single relative quasi-isometry class of rigid vertices. Since all standard geodesics in rigid vertices are $\mathcal{F}$-geodesics (Proposition \ref{prop:cliquesystem}), there are no relative stretch factors. Thus the na{\"\i}ve decoration and embellished decoration are equal.

The na{\"\i}ve decoration consists of two ornaments, one given to all rigid vertices and one given to all cylinder vertices. Each cylinder (resp. rigid) vertex is adjacent to countably  infinitely many rigid (resp. cylinder) vertices, so this decoration is stable under neighbour refinement. All vertices with the same ornament are relatively quasi-isometric via a relative quasi-isometry that preserves the decoration on incident edges. Thus this decoration is stable under vertex refinement.

It follows from Proposition \ref{prop:equivbijec} and the above discussion, that for some fixed $n\geq 2$, any two $n$-clique tree-graded RAAGs have equivalent JSJ trees of cylinders  when decorated with the na{\"\i}ve decoration. Thus by Theorem \ref{thm:mainthm}, any two $n$-clique tree-graded RAAGs are quasi-isometric.

Now suppose $A(\Gamma)$ is a RAAG that is quasi-isometric to an $n$-clique tree-graded RAAG. All cylindrical (resp. rigid) vertex groups of $C(\Gamma)$ are  relatively quasi-isometric to cylindrical (resp. rigid) vertex groups   of an $n$-clique tree-graded RAAG. Using Proposition \ref{prop:relqicyl} and the quasi-isometric rigidity of $\mathbb{Z}^n$, cylindrical and rigid vertices are of the above form.

In particular, the subgraph associated to each rigid vertex is an $n$-clique, all of whose vertices are cut vertices of $\Gamma$. Every cylindrical vertex has peripheral $\mathbb{Z}^{n-1}$ factors corresponding to adjacent rigid vertices, and non-peripheral $\mathbb{Z}$ factors corresponding to adjacent leaves. Thus $\Gamma$ is  an $n$-clique tree-graded graph.
\end{proof}

\bibliographystyle{alpha}
\bibliography{../bibliography/bibtex}
\end{document}